\documentclass{article}
\usepackage{amssymb}
\usepackage{amsfonts}
\usepackage{amsmath}
\usepackage{graphicx,nicefrac}
\usepackage[letterpaper, total={6in, 8in}]{geometry}
\usepackage{hyperref}

\newtheorem{theorem}{Theorem}[section]
\newtheorem{acknowledgement}[theorem]{Acknowledgement}

\newtheorem{corollary}[theorem]{Corollary}

\newtheorem{definition}[theorem]{Definition}
\newtheorem{example}[theorem]{Example}

\newtheorem{lemma}[theorem]{Lemma}

\newtheorem{remark}[theorem]{Remark}

\newenvironment{proof}[1][Proof]{\noindent\textbf{#1.} }{\ \rule{0.5em}{0.5em}}

\newtheorem{thmalpha}{Theorem}[theorem]

\begin{document}

\title{The Coulomb gauge in non-associative gauge theory}
\author{Sergey Grigorian}
\maketitle

\begin{abstract}
The aim of this paper is to extend existence results for the Coulomb gauge
from standard gauge theory to a non-associative setting. Non-associative
gauge theory is based on smooth loops, which are the non-associative analogs
of Lie groups. The main components of the theory include a
finite-dimensional smooth loop $\mathbb{L}$, its tangent algebra $\mathfrak{l%
},$ a finite-dimensional Lie group $\Psi $, that is the pseudoautomorphism
group of $\mathbb{L}$, a smooth manifold $M$ with a principal $\Psi $-bundle 
$\mathcal{P}$, and associated bundles $\mathcal{Q}$ and $\mathcal{A}$ with
fibers $\mathbb{L}$ and $\mathfrak{l}$, respectively. A configuration in
this theory is defined as a pair $\left( s,\omega \right) $, where $s$ is a
section of $\mathbb{Q}$ and $\omega $ is a connection on $\mathcal{P}$. The
torsion $T^{\left( s,\omega \right) }$ is the key object in the theory, with
a role similar to that of a connection in standard gauge theory. The
original motivation for this study comes from $G_{2}$-geometry, and the
questions of existence of $G_{2}$-structures with particular torsion types.
In particular, given a fixed connection, we prove existence of
configurations with divergence-free torsion, given a sufficiently small
torsion in a Sobolev norm.
\end{abstract}

\tableofcontents

%
%

\section{Introduction}

\setcounter{equation}{0}The main goal of this work is to extend results on
existence of Coulomb gauge transformations from standard gauge theory to a 
\emph{non-associative} setting. One of highly successful areas at the
intersection of differential geometry, analysis, and mathematical physics is
gauge theory, which, as it is well-known, is the study of connections on
bundles with particular Lie groups as the structure groups. In \cite%
{GrigorianLoops}, the author initiated a theory of smooth loops, which are
non-associative analogs of Lie groups, and began the development of gauge
theory based on loops, i.e. a non-associative gauge theory. The key example
of a non-associative smooth loop is the loop of unit octonions. A
non-associative gauge theory has the following components:

\begin{enumerate}
\item A finite-dimensional smooth (right) loop $\mathbb{L}$, i.e. a smooth
manifold with a right multiplication diffeomorphism $R_{p}$ defined for
every $p\in \mathbb{L}$ and a distinguished identity element $1\in \mathbb{L}%
,$ with tangent algebra $\mathfrak{l}$ at identity.

\item A finite-dimensional Lie group $\Psi $ that is the \emph{%
pseudoautomorphism} group of $\mathbb{L}$, a non-associative generalization
of an automorphism group.

\item A smooth manifold $M$ with a principal $\Psi $-bundle $\mathcal{P},$
and associated bundles $\mathcal{Q}$ and $\mathcal{A}$ with fibers $\mathbb{L%
}$ and $\mathfrak{l},$ respectively.

\item A configuration is defined by a pair $\left( s,\omega \right) ,$ where 
$s$ is a section of $\mathcal{Q}$ and $\omega $ is a connection on $\mathcal{%
P}.$ Together they define the \emph{torsion }$T^{\left( s,\omega \right) },$
which is an $\mathcal{A}$-valued $1$-form on $M.$ The torsion is then the
key object in the theory, in the same way that a connection is the key
object in standard gauge theory.

\item In addition to standard gauge transformations of $\omega $ by $\Psi ,$
we now also have transformations of $s$ induced by loop multiplication. Both
of these kinds of transformations induce transformations of the torsion.
\end{enumerate}

The original motivation for studying non-associative gauge theories comes
from $G_{2}$-geometry \cite{GrigorianOctobundle}. A $G_{2}$-structure on a $7
$-dimensional Riemannian manifold is a reduction of the structure group of
the orthornormal frame bundle from $SO\left( 7\right) $ to $G_{2},$ which is
the automorphism group of the octonions. A further relationship between $%
G_{2}$-structures and the octonions is that unit norm sections of an
octonion bundle parametrize $G_{2}$-structures that are associated with the
same metric, also known as \emph{isometric }$G_{2}$-structures. A defining
characteristic of a $G_{2}$-structure is its torsion, and one of the
unanswered questions in $G_{2}$-geometry is which torsion types are
admissible within a fixed metric class. One of the main goals in the theory
of $G_{2}$-structures is to obtain existence results for torsion-free $G_{2}$%
-structures, similar to the Yau's Theorem \cite{CalabiYau}, that settled the
existence question for Calabi-Yau manifolds. While this goal is formulated
in terms of $G_{2}$-structures, the real question is the existence of a
Riemannian metric with holonomy group equal to $G_{2}.$ The fact that for
each metric there exists an entire family of compatible $G_{2}$-structures
presents a degeneracy in this problem. Some of the existing approaches
involve Laplacian flows of $G_{2}$-structures \cite%
{BagagliniFernandezFino1,BagagliniFino1,bryant-2003,BryantXu,GaoChen1,GrigorianSU3flow,GrigorianCoflow,KarigiannisMcKayTsui,LotayWei1,LotayWei2,LotayWei1a}%
, with the hope of a flow eventually converging to a torsion-free $G_{2}$%
-structure. As shown in \cite{GrigorianCoflow}, the Laplacian flow of a
generic $G_{2}$-structure has a component that moves within a metric class,
and that component is precisely given by $\mathop{\rm div}T.$ Laplacian
flows have been more successful for \emph{closed }$G_{2}$-structures, in
which case $\mathop{\rm div}T$ automatically vanishes, and thus the
degeneracy is resolved. More generally, however, this degeneracy is the
source of non-parabolicity of Laplacian flows, such as in the case of \emph{%
co-closed} $G_{2}$-structures \cite{GrigorianCoflow}. Therefore, the
condition $\mathop{\rm div}T=0$ can be regarded as a \emph{gauge-fixing
condition}. Moreover, in \cite{GrigorianOctobundle}, it was found that on a
compact manifold, $G_{2}$-structures with $\mathop{\rm div}T=0$ are
precisely the critical points of the $L^{2}$-norm of the torsion when
restricted to a fixed metric class. As shown in \cite{GrigorianOctobundle,
GrigorianLoops}, within the loop bundle framework, this is the precise
analog of the Coulomb gauge condition.

Existence of $G_{2}$-structures with divergence-free torsion has been
studied from different perspectives by several authors: Bagaglini in \cite%
{Bagaglini2}; Dwivedi, Gianniotis, and Karigiannis in \cite{DGKisoflow}; the
author in \cite{GrigorianIsoflow}; Loubeau and S\'{a} Earp in \cite%
{SaEarpLoubeau}. All these approaches relied on a flow of isometric $G_{2}$%
-structures (or more generally, geometric structures in \cite{SaEarpLoubeau}
and \cite{SaEarpetAlflows}), and have shown existence of a $G_{2}$-structure
with divergence-free torsion as a long-term limit of the flow, given
sufficiently small pointwise initial torsion or another quantity, the
entropy.

The interpretation of isometric $G_{2}$-structures as an octonionic
non-associative gauge theory allows to adapt some gauge theory techniques in
this setting. Moreover, without much additional effort, more general loops
can be considered, with potential wider-reaching applications.

In gauge theory there are a number of versions of local and global existence
results for connections in the Coulomb gauge, depending on the desired
regularity \cite%
{DonaldsonKronheimer,Feehan2001,Feehan2017,Feehan2019,Feehan2020,UhlenbeckConnection,WehrheimBook}%
. In this paper we use the Quantitative Implicit Function Theorem for Banach
Spaces, as given in \cite{Feehan2020}, to prove the following main result.

\begin{thmalpha}
\label{thmA}Suppose $\mathbb{L}$ is a smooth compact loop with tangent
algebra $\mathfrak{l}$ and pseudoautomorphism group $\Psi .$ Let $\left(
M,g\right) $ be a closed, smooth Riemannian manifold of dimension $n\geq 2,$
and let $\mathcal{P}$ be a $\Psi $-principal bundle over $M$ and let $%
\mathcal{A}$ be the associated vector bundle to $\mathcal{P}$ with fibers
isomorphic to $\mathfrak{l}$. Let $\omega $ be a smooth connection on $%
\mathcal{P}.$ Also, suppose $k$ is a non-negative integer and $r\geq 0$ such
that $kr>n.$ Then, there exist constants $\delta \in (0,1]$ and $K\in \left(
0,\infty \right) ,$ such that if $s\in \Gamma \left( \mathcal{Q}\right) $ is
a smooth defining section for which 
\begin{equation*}
\left\Vert T^{\left( s,\omega \right) }\right\Vert _{W^{k-1,r}}<\delta ,
\end{equation*}%
then there exists a section $A\in W^{k,r}\left( \mathcal{Q}^{\prime }\right)
,$ such that 
\begin{equation*}
\left( d^{\omega }\right) ^{\ast }T^{\left( As,\omega \right) }=0
\end{equation*}%
and 
\begin{equation}
\left\Vert T^{\left( As,\omega \right) }\right\Vert _{W^{k-1,r}}<K\left\Vert
T^{\left( s,\omega \right) }\right\Vert _{W^{k-1,r}}\left( 1+\left\Vert
T^{\left( s,\omega \right) }\right\Vert _{W^{k-1,r}}^{k-1}\right) .
\end{equation}%
If moreover, $\left( k-1\right) r\geq n,$ then $A$ is smooth.
\end{thmalpha}

For $G_{2}$-structures, this gives the following result for existence of
smooth $G_{2}$-structures with divergence-free torsion.

\begin{thmalpha}
\label{thmB}Suppose $M$ is a closed $7$-dimensional manifold with a smooth $%
G_{2}$-structure $\varphi $ with torsion $T$ with respect to the Levi-Civita
connection $\nabla .$ Suppose $U\mathbb{O}M$ is the corresponding unit
octonion bundle. Also, suppose $k$ is a positive integer and $r$ is a
positive real number such that $kr>7.$ Then, there exist constants $\delta
\in (0,1]$ and $K\in \left( 0,\infty \right) ,$ such that if $T$ satisfies 
\begin{equation*}
\left\Vert T\right\Vert _{W^{k,r}}<\delta ,
\end{equation*}%
then there exists a smooth section $V\in \Gamma \left( U\mathbb{O}M\right) ,$
such that 
\begin{equation*}
\mathop{\rm div}T^{\left( V\right) }=0
\end{equation*}%
and 
\begin{equation}
\left\Vert T^{\left( V\right) }\right\Vert _{W^{k,r}}<K\left\Vert
T\right\Vert _{W^{k,r}}\left( 1+\left\Vert T\right\Vert
_{W^{k,r}}^{k}\right) .
\end{equation}
\end{thmalpha}

The results presented in this paper are of interest and importance in their
own right, but perhaps even more crucially, they show that some well-known
results and techniques from classical gauge theory can be reinterpreted and
adapted in a non-associative setting. In particular, this may open the door
to some analogues of Uhlenbeck compactness and a better understanding of the
torsion of non-associative gauge theories. Furthermore, a non-associative
version of Yang-Mills equations can be considered. Moreover, any such
advances will give immediate results in $G_{2}$-geometry.

The structure of this paper is the following. In Section \ref{sectLoop}, we
give an overview of smooth loops, extending \cite{GrigorianLoops}. We give
the basic properties of a smooth loop $\mathbb{L}$, define the
pseudoautomorphism group $\Psi $ and the tangent algebra $\mathfrak{l}$ at
identity. The algebra $\mathfrak{l}$ is a generalization of a Lie algebra,
but due to the non-associativity of $\mathbb{L}$ does not satisfy the Jacobi
identity. Similarly as for Lie algebras, there is a notion of an exponential
map. There is however a family of brackets $\left[ \cdot ,\cdot \right]
^{\left( s\right) }$ on $\mathfrak{l}$, defined for each point $s\in \mathbb{%
L}.$ For later use, we also give estimates for the exponential and adjoint
maps. In particular, we analyze solutions of the following initial value
problem for $\mathfrak{l}$-valued maps $\eta \left( t\right) :$ 
\begin{equation}
\left\{ 
\begin{array}{c}
\frac{d\eta \left( t\right) }{dt}=\left[ \xi ,\eta \left( t\right) \right]
^{\left( \exp _{s}\left( t\xi \right) s\right) } \\ 
\eta \left( 0\right) =\eta _{0}%
\end{array}%
\right. ,  \label{keyode}
\end{equation}%
where $\xi \in \mathfrak{l}$ and $s\in \mathbb{L}.$

In Section \ref{sectMaps}, we switch attention to loop-valued maps. In
particular, given a smooth manifold $M,$ consider a map $s:M\longrightarrow 
\mathbb{L}.$ Using this map, we may define products of $\mathbb{L}$-valued
maps and brackets of $\mathfrak{l}$-valued maps. Then, using the right
quotient, translating the differential $ds$ to the tangent space at $1\in 
\mathbb{L},$ we obtain an $\mathfrak{l}$-valued $1$-form $\theta _{s}$ on $M,
$ which is the analogue of the Darboux derivative of Lie group-valued maps 
\cite{SharpeBook}. The differentials of various operations defined by $s$
are then expressed in terms of $\theta _{s}.$ Suppose $A\left( t\right)
=\exp _{s}\left( t\xi \right) $ for some $\mathfrak{l}$-valued map $\xi .$
We show that $\theta _{A\left( t\right) s}$ satisfies a non-homogeneous
version of (\ref{keyode}). 

Further, we define Sobolev spaces of maps from $M$, and show in Lemma \ref%
{lemSobLoops} that, similarly as for Lie groups, $s\in W^{k,r}\left( M,%
\mathbb{L}\right) $ if and only if $\theta _{s}\in W^{k-1,r}\left( M,T^{\ast
}M\otimes \mathfrak{l}\right) .$  Using the evolution equation satisfied by $%
\theta _{A\left( t\right) s}$ then allows us to obtain Sobolev space
estimates of $\theta _{A\left( t\right) s}$ and other quantities that
satisfy equations based on (\ref{keyode}). 

\begin{thmalpha}
\label{thmC}Let $M$ be a compact Riemannian manifold and $\mathbb{L}$ is a
smooth compact loop. Suppose $kr>n=\dim M.$ Let $s\in W^{k,r}\left( M,%
\mathbb{L}\right) $ and $\xi \in W^{k,r}\left( M,\mathfrak{l}\right) ,$ and
suppose $A=\exp _{s}\left( \xi \right) .$ Then, 
\begin{equation}
\left\Vert \theta _{As}\right\Vert _{W^{k-1,r}}\lesssim e^{Ck\left\Vert \xi
\right\Vert _{C^{0}}}\left( \Theta ^{k}+\Theta \right) ,
\end{equation}%
where $\Theta =\left\Vert \theta _{s}\right\Vert _{W^{k-1,r}}+\left\Vert \xi
\right\Vert _{W^{k,r}}.$

Similarly, if $X=X\left( 1\right) ,$ where $X\left( t\right) $ is $1$%
-parameter family of $\mathfrak{l}$-valued maps that satisfies 
\begin{equation*}
\left\{ 
\begin{array}{c}
\frac{dX\left( t\right) }{dt}=\left[ \xi ,X\left( t\right) \right] ^{A\left(
t\right) s}+Y \\ 
X\left( 0\right) =X_{0},%
\end{array}%
\right. 
\end{equation*}%
for $\mathfrak{l}$-valued maps $X_{0}$ and $Y$, then, 
\begin{equation}
\left\Vert X\right\Vert _{W^{k,r}}\lesssim e^{C\left( k+1\right) \left\Vert
\xi \right\Vert _{C^{0}}}\left( \left\Vert X_{0}\right\Vert
_{W^{k,r}}+\left\Vert Y\right\Vert _{W^{k,r}}\right) \left( \Theta
^{k+1}+\Theta \right) .
\end{equation}
\end{thmalpha}

In Section \ref{sectGauge}, we introduce a principal $\Psi $-bundle $%
\mathcal{P}$ over a compact manifold $M$, and then apply the above results
to $\Psi $-equivariant maps from $\mathcal{P}$ to a loop $\mathbb{L}$ and
other related spaces. This immediately then allows to consider sections of
bundles over $M$ that are associated to $\mathcal{P}.$ In particular,
suppose we have a connection $\omega $ on $\mathcal{P}$ and suppose $s$ is a
section of the associated loop bundle $\mathcal{Q},$ with fibers
diffeomorphic to $\mathbb{L}.$ It uniquely corresponds to a $\Psi $%
-equivariant map $\tilde{s}:\mathcal{P}\longrightarrow \mathbb{L},$ and thus
we obtain the equivariant $\mathfrak{l}$-valued $1$-form $\theta _{\tilde{s}}
$ on $\mathcal{P}.$ On the other hand, the connection $\omega $ defines a
decomposition of $T\mathcal{P}$ into vertical and horizontal subspaces.
Therefore, we may compose $\theta _{\tilde{s}}$ with the horizontal
projection to obtain a \emph{basic}, i.e. horizontal and equivariant $%
\mathfrak{l}$-valued $1$-form on $\mathcal{P}.$ This then corresponds to a
section of a bundle over $M,$ and gives us the \emph{torsion }$T^{\left(
s,\omega \right) }$ of the configuration $\left( s,\omega \right) .$
Defining fiberwise loop multiplication, we see that all the possible
configurations with a fixed $\omega $ may be obtained by multiplying $s$ by
some section $A.$ Therefore, the \emph{loop gauge transformations} are
precisely the transformations $s\mapsto As.$ Moreover, as it was already
known previously, \cite%
{GrigorianOctobundle,GrigorianIsoflow,GrigorianIsoFlowSurvey,GrigorianLoops,SaEarpLoubeau}%
, given appropriate algebraic conditions on the loop, the critical points of
the functional $s\mapsto \int_{M}$ $\left\vert T^{\left( s,\omega \right)
}\right\vert ^{2}\mathop{\rm vol}\nolimits$ are precisely the sections $s$
for which $\left( d^{\omega }\right) ^{\ast }T^{\left( s,\omega \right) }=0,$
which relates to the previous discussion on divergence-free torsion and the
Coulomb gauge. 

Considering the transformations of $s$ of the form $s\mapsto \exp _{s}\left(
\xi \right) s$, for $\mathfrak{l}$-valued sections $\xi $, and using the
loop exponential map, as developed in Section \ref{sectLoop}, the quantity $%
T^{\left( \exp _{s}\left( t\xi \right) s\,,\omega \right) }$ is then shown
to satisfy an ODE with the homogeneous part of the form (\ref{keyode}). This
allows to rewrite $T^{\left( \exp _{s}\left( t\xi \right) s\,,\omega \right)
}$ in terms of the evolution operator for (\ref{keyode}). The equation $%
\left( d^{\omega }\right) ^{\ast }T^{\left( \exp _{s}\left( \xi \right)
s,\omega \right) }=0$ is then written as a second-order PDE for $\xi .$ This
operator then defines a smooth functions between appropriate Banach spaces
of sections of vector bundles, which allows to apply the Implicit Function
Theorem from \cite{Feehan2020}, to show existence of solutions for
sufficiently small initial $T^{\left( s,\omega \right) }$ in an appropriate
Sobolev norm, and thus prove Theorem \ref{thmA}.

In Section \ref{secG2}, we carefully apply the general theory of smooth
loops to the particular case of $G_{2}$-structures, and then Theorem \ref%
{thmB} follows as an immediate corollary of Theorem \ref{thmA}.

\subsubsection*{Acknowledgements}

The author is supported by the National Science Foundation grant DMS-1811754.

\section{Smooth Loops}

\setcounter{equation}{0}\label{sectLoop}For a detailed introduction to
smooth loops, the reader is referred to \cite{GrigorianLoops}. The reader
can also refer to \cite%
{HofmannStrambach,KiechleKloops,NagyStrambachBook,SabininBook,SmithJDHQuasiReps}
for a discussion of these concepts.

\begin{definition}
A \emph{loop }$\mathbb{L}$ is a set with a binary operation $p\cdot q$ with
identity $1,$ and compatible left and right quotients $p\backslash q$ and $%
p/q$, respectively.
\end{definition}

In particular, existence of quotients is equivalent to saying that for any $%
q\in \mathbb{L},$ the left and right product maps $L_{q}$ and $R_{q}$ are
invertible maps. Restricting to the smooth category, we obtain the
definition of a smooth loop.

\begin{definition}
A \emph{smooth loop }is a smooth manifold $\mathbb{L}$ with a loop structure
such that the left and right product maps are diffeomorphisms of $\mathbb{L}%
. $
\end{definition}

\begin{remark}
In this paper we will not use the left quotient, so in fact everything that
follows also holds true for smooth \emph{right }loops, i.e. where only the
right quotient is defined, but the left product is not necessarily
invertible. However, for brevity, we will keep referring to \emph{loops},
rather than \emph{right loops}. As Example \ref{transversalEx} below shows,
smooth right loops are plentiful and easy to construct.
\end{remark}

\begin{example}
\label{transversalEx} Suppose $G$ is a Lie group with a Lie subgroup $H$ and
consider the left quotient $K=\scalebox{-1}[1]{\nicefrac{%
\scalebox{-1}[1]{$G$}}{\scalebox{-1}[1]{$H$}}}$. Suppose $\sigma
:K\longrightarrow G$ is a section of $G$, regarded as a bundle over $K.$ In
particular, $\sigma $ maps each right coset to a particular representative
of that coset. Suppose $\sigma \left( H\right) =1.$ A product structure on $K
$ is then given by 
\begin{equation}
\left( Ha\right) \circ \left( Hb\right) =Ha\left( \sigma \left( Hb\right)
\right) .  \label{transversal}
\end{equation}%
Consider the equation $\left( Hx\right) \circ \left( Ha\right) =Hb,$ Since $%
\sigma $ is a section, we can see right away that we have a unique solution $%
Hx=Hb\left( \sigma \left( Ha\right) \right) ^{-1}.$ Thus, $\left( \mathbb{L}%
,\circ \right) $ has \emph{right }division, and is thus a \emph{right loop }%
\cite{KiechleKloops,NagyStrambachBook}. To define left division, and hence
to obtain a full loop structure, more structure is needed.
\end{example}

\begin{definition}
\label{defPseudo}A \emph{pseudoautomorphism} of a smooth loop $\mathbb{L}$
is a diffeomorphism $h:\mathbb{L}\longrightarrow \mathbb{L}$ for which there
exists another diffeomorphism $h^{\prime }:\mathbb{L}\longrightarrow \mathbb{%
L}$, known as the partial pseudoautomorphism corresponding to $h,$ such that
for any $p,q\in \mathbb{L}$, 
\begin{equation}
h\left( pq\right) =h^{\prime }\left( p\right) h\left( q\right) .
\label{pseudoauto}
\end{equation}
\end{definition}

In particular, $h^{\prime }=R_{h\left( 1\right) }^{-1}\circ h$ and $%
h^{\prime }\left( 1\right) =1.$ The element $h\left( 1\right) \in \mathbb{L}$
is the \emph{companion} of $h^{\prime }.$ From (\ref{pseudoauto}), we also
see the following property of $h^{\prime }$ with respect to quotients: 
\begin{equation}
h^{\prime }\left( p/q\right) =h\left( p\right) /h\left( q\right) .
\label{PsAutquot2}
\end{equation}

It is easy to see that the sets of pseudoautomorphisms and partial
pseudoautomorphisms are both groups. Denote the former by $\Psi $ and the
latter by $\Psi ^{\prime }$. We also see that the \emph{automorphism }group
of the loop $\mathbb{L}$ is the subgroup $H\subset \Psi $ which is the
stabilizer of $1\in \mathbb{L}.$

\begin{remark}
To avoid introducing additional notation, but at the risk of some ambiguity,
we will use the same notation $\mathbb{L}$ to denote the underlying
manifold, the loop, and the $G$-set with the full action of $\Psi $.
However, since $\mathbb{L}$ also admits the action of $\Psi ^{\prime },$ if
a distinction between the $G$-sets is needed, we will use $\mathbb{L}%
^{\prime }$ to denote the set $\mathbb{L}$ with the action of $\Psi ^{\prime
},$
\end{remark}

Let $r\in \mathbb{L}$, then we may define a modified product $\circ _{r}$ on 
$\mathbb{L}$ via $p\circ _{r}q={\left( p\cdot qr\right)}/{r}$, so that $%
\mathbb{L}$ equipped with product $\circ _{r}$ will be denoted by $\left( 
\mathbb{L},\circ _{r}\right) ,$ the corresponding quotient will be denoted
by $/_{r}.$ We have the following properties \cite{GrigorianLoops}.

\begin{lemma}
\label{lemPseudoHom}Let $h\in \Psi $. Then, for any $p,q,r\in \mathbb{L},$ 
\begin{equation}
h^{\prime }\left( p\circ _{r}q\right) =h^{\prime }\left( p\right) \circ
_{h\left( r\right) }h^{\prime }\left( q\right) \ \ \ \ \ h^{\prime }\left(
p/_{r}q\right) =h^{\prime }\left( p\right) /_{h\left( r\right) }h^{\prime
}\left( q\right) .  \label{PsiAct}
\end{equation}%
Also, for any $A\in \mathbb{L},$ 
\begin{equation}
\left( p\circ _{r}\left( q\circ _{r}A\right) \right) /_{r}A=p\circ _{Ar}q.
\label{Arprod}
\end{equation}
\end{lemma}

\begin{lemma}
\label{lemQuotient}Suppose $A\left( t\right) $ and $B\left( t\right) $ are
smooth curves in $\mathbb{L}$ with $A\left( t_{0}\right) =A_{0}$ and $%
B\left( t_{0}\right) =B_{0}$, then 
\begin{equation}
\left. \frac{d}{dt}A\left( t\right) /B\left( t\right) \right\vert
_{t=t_{0}}=\left. \frac{d}{dt}A\left( t\right) /B_{0}\right\vert
_{t=t_{0}}-\left. \frac{d}{dt}\left( A_{0}/B_{0}\cdot B\left( t\right)
\right) /B_{0}\right\vert _{t=t_{0}}.  \label{ddtrighquot}
\end{equation}
\end{lemma}

Consider the tangent space $\mathfrak{l:=}T_{1}\mathbb{L}$ at $1\in \mathbb{L%
}.$ For any $q\in \mathbb{L}$, the pushforward $\left( R_{q}\right) _{\ast }$
of the right translation map $R_{q}$ defines a linear isomorphism from $%
\mathfrak{l}$ to $T_{q}\mathbb{L}.$ In particular, let us denote the linear
map $\left. \left( R_{q}\right) _{\ast }\right\vert _{1}:\mathfrak{l}%
\longrightarrow T_{q}\mathbb{L}$ by $\rho _{q}$, and correspondingly, $%
\left. \left( R_{q}^{-1}\right) _{\ast }\right\vert _{1}:T_{q}\mathbb{L}%
\longrightarrow \mathfrak{l}$ by $\rho _{q}^{-1}.$ Similarly, for left
multiplication, define $\lambda _{q}=\left. \left( L_{q}\right) _{\ast
}\right\vert _{1}:\mathfrak{l}\longrightarrow T_{q}\mathbb{L}.$ On a smooth
right loop, $\lambda _{q}$ will not necessarily be invertible. The
corresponding maps with respect to the product $\circ _{r}$ will be denoted
by $R_{q}^{\left( r\right) },$ $\rho _{q}^{\left( r\right) },$ $\lambda
_{q}^{\left( r\right) }.$

\begin{definition}
For any $\xi \in \mathfrak{l,}$ define the\emph{\ fundamental vector field }$%
\rho \left( \xi \right) $ for any $q\in \mathbb{L},$ $\rho \left( \xi
\right) _{q}=\rho _{q}\left( \xi \right) .$
\end{definition}

The above definition of the fundamental vector field is the analog of a
right-invariant vector field in Lie theory. However, in the loop case,
although this vector field is canonical in some sense, it is not \emph{%
invariant} under right translations. We use fundamental vector fields to
define the loop exponential map.

\begin{definition}
Suppose $\mathbb{L}$ is a smooth loop and suppose $q\in \mathbb{L}.$ Then,
given $\xi \in \mathfrak{l},$ for sufficiently small $t,$ define 
\begin{equation}
p\left( t\right) =\exp _{q}\left( t\xi \right) q.
\end{equation}%
to be the solution of the equation 
\begin{equation}
\left\{ 
\begin{array}{c}
\frac{dp\left( t\right) }{dt}=\left. \rho \left( \xi \right) \right\vert
_{p\left( t\right) } \\ 
p\left( 0\right) =q%
\end{array}%
\right. .  \label{floweq4}
\end{equation}%
Equivalently, $\tilde{p}\left( t\right) =\exp _{q}\left( t\xi \right) $
satisfies 
\begin{equation}
\left\{ 
\begin{array}{c}
\frac{d\tilde{p}\left( t\right) }{dt}=\left. \rho ^{\left( q\right) }\left(
\xi \right) \right\vert _{\tilde{p}\left( t\right) } \\ 
\tilde{p}\left( 0\right) =1%
\end{array}%
\right. .  \label{floweq5}
\end{equation}
\end{definition}

\begin{remark}
In general, the solution $\exp _{q}$ will only be defined in a neighborhood
of $0\in \mathfrak{l},$ however as shown in \cite{Kuzmin1971,Malcev1955}, if
the loop $\mathbb{L}$ is power-associative, so that powers of an element $%
p\in \mathbb{L}$ associate, then $p\left( nh\right) =p\left( h\right) ^{n}$
can be defined unambiguously. We will show this from a different perspective
further below. This can then be used to define the solution $p\left(
t\right) $ for all $t,$ and thus this extends $\exp _{q}$ to all of $%
\mathfrak{l}.$
\end{remark}

Let us consider $d\exp _{q}.$ From the definition of $\exp _{q}$, for any $%
\xi \in \mathfrak{l,}$ have 
\begin{equation}
\left. d\exp _{q}\right\vert _{0}\left( \xi \right) =\left. \frac{d}{dt}\exp
_{q}\left( t\xi \right) \right\vert _{t=0}=\xi .  \label{dexp}
\end{equation}%
In particular, $\exp _{s}$ is smooth and since the identity map is a linear
isomorphism, by the Inverse Function Theorem, we have the following.

\begin{lemma}
\label{lemexpdiffeo}For any $q\in \mathbb{L},$ the map $\exp _{q}:\mathfrak{l%
}\longrightarrow \mathbb{L}$ is a local diffeomorphism around $0\in 
\mathfrak{l}$.
\end{lemma}

\begin{remark}
To distinguish the exponential map on $\mathfrak{l}$ from the exponential
map on $\mathfrak{p,}$ we will use a subscript to denote with respect to
which element of $\mathbb{L}$ the exponential map is used. The exponential
map on $\mathfrak{p}$ will be without the subscript.
\end{remark}

On smooth loops, we can define an analog of the Lie group Maurer-Cartan form.

\begin{definition}[\protect\cite{GrigorianLoops}]
The Maurer-Cartan form $\theta $ is an $\mathfrak{l}$-valued $1$-form on $%
\mathbb{L}$, such that for any vector field $X$, and any $p\in \mathbb{L},$ $%
\left. \theta \left( X\right) \right\vert _{p}=\rho _{p}^{-1}\left(
X_{p}\right) \in \mathfrak{l}.$ Equivalently, for any $\xi \in \mathfrak{l},$
$\theta \left( \rho \left( \xi \right) \right) =\xi .$
\end{definition}

The loop Maurer-Cartan form allows us to define brackets on $\mathfrak{l.}$
For each $p\in \mathbb{L}$ define the bracket $\left[ \cdot ,\cdot \right]
^{\left( p\right) }$ given for any $\xi ,\eta \in \mathfrak{l}$ by 
\begin{equation*}
\left[ \xi ,\eta \right] ^{\left( p\right) }=-\left. \theta \left( \left[
\rho \left( \xi \right) ,\rho \left( \eta \right) \right] \right)
\right\vert _{p}.
\end{equation*}%
As shown in \cite[Theorem 3.7]{GrigorianLoops}, we can equivalently define 
\begin{eqnarray}
\left[ \xi ,\gamma \right] ^{\left( p\right) } &=&\left. \frac{d}{dt}\left( %
\mathop{\rm Ad}\nolimits_{\exp \left( t\xi \right) }^{\left( p\right)
}\gamma \right) \right\vert _{t=0}  \notag \\
&=&\left. \frac{d^{2}}{dtd\tau }\exp \left( t\xi \right) \circ _{p}\exp
\left( \tau \gamma \right) \right\vert _{t,\tau =0}  \label{brack2deriv} \\
&&-\left. \frac{d^{2}}{dtd\tau }\exp \left( \tau \gamma \right) \circ
_{p}\exp \left( t\xi \right) \right\vert _{t,\tau =0},  \notag
\end{eqnarray}%
where, for $p,q\in \mathbb{L}$, $\mathop{\rm Ad}\nolimits_{q}^{\left(
p\right) }:\mathfrak{l}\longrightarrow \mathfrak{l}$ is the differential at $%
1\in \mathbb{L}$ of the conjugation map $r\mapsto \left( q\circ _{p}r\right)
/_{p}q\in \mathbb{L}.$

\begin{remark}
In \cite{GrigorianLoops}, the conjugation map $r\mapsto \left( q\circ
_{p}r\right) /_{p}q$ was denoted by $\mathop{\rm Ad}\nolimits_{q}^{\left(
p\right) },$ and its differential as $\left( \mathop{\rm Ad}%
\nolimits_{q}^{\left( p\right) }\right) _{\ast }. $ However here we adopt
notation that is more in line with standard usage in Lie theory.
\end{remark}

\begin{definition}
The vector space $\mathfrak{l}$ equipped with the bracket $\left[ \cdot
,\cdot \right] ^{\left( p\right) }$ is known as the \emph{loop tangent
algebra }$\mathfrak{l}^{\left( p\right) }.$
\end{definition}

Define the \emph{bracket function }$b:\mathbb{L}\longrightarrow \mathfrak{l}%
\otimes \Lambda ^{2}\mathfrak{l}^{\ast }$ to be the map that takes $p\mapsto %
\left[ \cdot ,\cdot \right] ^{\left( p\right) }\in \mathfrak{l}\otimes
\Lambda ^{2}\mathfrak{l}^{\ast }$, so that $b\left( \theta ,\theta \right) $
is an $\mathfrak{l}$-valued $2$-form on $\mathbb{L}$, i.e. $b\left( \theta
,\theta \right) \in \Omega ^{2}\left( \mathfrak{l}\right) $.

\begin{definition}
For any $\eta ,\gamma ,\xi \in \mathfrak{l},$ and $p\in \mathbb{L},$ the%
\emph{\ associator }$\left[ \cdot ,\cdot ,\cdot \right] ^{\left( p\right) }$
on $\mathfrak{l}^{\left( p\right) }$ given by 
\begin{eqnarray}
\left[ \eta ,\gamma ,\xi \right] ^{\left( p\right) } &=&\left. \frac{d^{3}}{%
dtd\tau d\tau ^{\prime }}\exp \left( \tau \eta \right) \circ _{p}\left( \exp
\left( \tau ^{\prime }\gamma \right) \circ _{p}\exp \left( t\xi \right)
\right) \right\vert _{t,\tau ,\tau ^{\prime }=0}  \label{Lalgassoc} \\
&&-\left. \frac{d^{3}}{dtd\tau d\tau ^{\prime }}\left( \exp \left( \tau \eta
\right) \circ _{p}\exp \left( \tau ^{\prime }\gamma \right) \right) \circ
_{p}\exp \left( t\xi \right) \right\vert _{t,\tau ,\tau ^{\prime }=0}. 
\notag
\end{eqnarray}%
Moreover, define mixed associators between elements of $\mathbb{L}$ and $%
\mathfrak{l}$. An $\left( \mathbb{L},\mathbb{L},\mathfrak{l}\right) $%
-associator is defined for any $p,q\in \mathbb{L}$ and $\xi \in \mathfrak{l}$
as 
\begin{equation}
\left[ p,q,\xi \right] ^{\left( s\right) }=\left( L_{p}^{\left( s\right)
}\circ L_{q}^{\left( s\right) }\right) _{\ast }\xi -\left( L_{p\circ
_{s}q}^{\left( s\right) }\right) _{\ast }\xi \in T_{p\circ _{s}q}\mathbb{L}
\label{pqxiassoc}
\end{equation}%
and an $\left( \mathbb{L},\mathfrak{l},\mathfrak{l}\right) $-associator is
defined for an $p\in \mathbb{L}$ and $\eta ,\xi \in \mathfrak{l}$ as 
\begin{eqnarray}
\left[ p,\eta ,\xi \right] ^{\left( s\right) } &=&\left. \frac{d}{dtd\tau }%
\left( p\circ _{s}\left( \exp \left( t\eta \right) \circ _{s}\exp \left(
\tau \xi \right) \right) \right) \right\vert _{t\,,\tau =0}  \notag \\
&&-\left. \frac{d}{dtd\tau }\left( \left( p\circ _{s}\exp \left( t\eta
\right) \right) \circ _{s}\exp \left( \tau \xi \right) \right) \right\vert
_{t,\tau =0},  \label{etapxiassoc}
\end{eqnarray}%
where we see that $\left[ p,\eta ,\xi \right] ^{\left( s\right) }\in T_{p}%
\mathbb{L}.$ Similarly, for other combinations. Also define the \emph{%
left-alternating associator }$a:\mathbb{L}\longrightarrow \mathfrak{l}%
\otimes \Lambda ^{2}\mathfrak{l}^{\ast }\otimes \mathfrak{l}^{\ast },$ given
by 
\begin{equation}
a_{p}\left( \eta ,\gamma ,\xi \right) =\left[ \eta ,\gamma ,\xi \right]
^{\left( p\right) }-\left[ \gamma ,\eta ,\xi \right] ^{\left( p\right) }.
\label{ap}
\end{equation}%
which we can call the \emph{left-alternating associator}.
\end{definition}

\begin{remark}
From the definitions of the associators, it is easy to see that if $\left( 
\mathbb{L},\circ _{s}\right) $ is power-associative, given $\xi \in 
\mathfrak{l},$ associators with any combinations of $\xi $ and $\exp
_{s}\left( t\xi \right) ,$ for any values of $t$, in the three entries, will
vanish. For example, 
\begin{subequations}%
\label{powerassoc}%
\begin{eqnarray}
\left[ \xi ,\xi ,\xi \right] ^{\left( s\right) } &=&0 \\
\lbrack \xi ,\xi ,\exp _{s}\left( t\xi \right) ]^{\left( s\right) } &=&0 \\
\left[ \xi ,\exp _{s}\left( t\xi \right) ,\exp _{s}\left( \tau \xi \right) %
\right] ^{\left( s\right) } &=&0,
\end{eqnarray}%
\end{subequations}
as well as any permutations.

Similarly, if $\left( \mathbb{L},\circ _{s}\right) $ is
left-power-associative, then associators with any combination of $\xi $ and $%
\exp _{s}\left( t\xi \right) $ in the first two entries will vanish, for
example 
\begin{subequations}%
\begin{eqnarray}
\left[ \xi ,\xi ,\eta \right] ^{\left( s\right) } &=&0 \\
\lbrack \xi ,\exp _{s}\left( t\xi \right) ,\eta ]^{\left( s\right) } &=&0 \\
\left[ \exp _{s}\left( t\xi \right) ,\exp _{s}\left( \tau \xi \right) ,\eta %
\right] ^{\left( s\right) } &=&0,
\end{eqnarray}%
\end{subequations}
for any $\eta \in \mathfrak{l}$ and similarly with the third entry replaced
by an element of $\mathbb{L}.$
\end{remark}

From \cite{GrigorianLoops} we cite several useful properties of these
brackets and associators.

\begin{theorem}[{\protect\cite[Theorem 3.20]{GrigorianLoops}}]
Suppose $p,s\in \mathbb{L}$, and $\xi ,\eta \in \mathfrak{l}.$ Then the
bracket $\left[ \cdot ,\cdot \right] ^{\left( ps\right) }$ is related to $%
\left[ \cdot ,\cdot \right] ^{\left( s\right) }$ via the expression 
\begin{equation}
\left[ \xi ,\eta \right] ^{\left( ps\right) }=\left[ \xi ,\eta \right]
^{\left( s\right) }+\left( \rho _{p}^{\left( s\right) }\right)
^{-1}a_{s}\left( \xi ,\eta ,p\right) .  \label{Adbrack1a}
\end{equation}
\end{theorem}

\begin{theorem}[{\protect\cite[Theorem 3.10]{GrigorianLoops}}]
\label{thmMC}The form $\theta $ satisfies 
\begin{equation}
d\theta =\frac{1}{2}b\left( \theta ,\theta \right) ,  \label{dtheta}
\end{equation}%
where wedge product of $1$-forms is implied. Also, for any $\xi ,\eta \in 
\mathfrak{l},$ we have 
\begin{equation}
db\left( \xi ,\eta \right) =a\left( \xi ,\eta ,\theta \right) .  \label{db2}
\end{equation}

It follows that $\xi ,\eta ,\gamma \in \mathfrak{l,}$ the generalized Jacobi
identity is satisfied: 
\begin{equation}
\mathop{\rm Jac}\nolimits^{\left( s\right) }\left( \xi ,\eta ,\gamma \right)
=a_{s}\left( \xi ,\eta ,\gamma \right) +a_{s}\left( \eta ,\gamma ,\xi
\right) +a_{s}\left( \gamma ,\xi ,\eta \right) ,  \label{Jac2}
\end{equation}%
where 
\begin{equation}
\mathop{\rm Jac}\nolimits^{\left( s\right) }\left( \xi ,\eta ,\gamma \right)
=\left[ \xi ,\left[ \eta ,\gamma \right] ^{\left( s\right) }\right] ^{\left(
s\right) }+\left[ \eta ,\left[ \gamma ,\xi \right] ^{\left( s\right) }\right]
^{\left( s\right) }+\left[ \gamma ,\left[ \xi ,\eta \right] ^{\left(
s\right) }\right] ^{\left( s\right) }.  \label{Jac}
\end{equation}
\end{theorem}

\begin{remark}
Equation (\ref{dtheta}) is the loop Maurer-Cartan equation. The key
difference from the Maurer-Cartan equation on Lie groups is that on
non-associative loops, $b\left( s\right) $ is non-constant on $\mathbb{L},$
unlike on Lie groups, where there is a unique bracket on the Lie algebra,
and hence $b\left( s\right) $ is constant. In particular, the non-constant $b
$ leads to a non-trivial associator (\ref{db2}) and the failure of the
standard Jacobi identity to hold.
\end{remark}

With respect to the action of $\Psi ,$ the bracket and the associator
satisfy the following properties.

\begin{lemma}
\label{corLoppalghom}If $h\in \Psi \mathbb{\ }$and $q\in \mathbb{L}$, then,
for any $\xi ,\eta ,\gamma \in \mathfrak{l}$, 
\begin{eqnarray*}
h_{\ast }^{\prime }\left[ \xi ,\eta \right] ^{\left( q\right) } &=&\left[
h_{\ast }^{\prime }\xi ,h_{\ast }^{\prime }\eta \right] ^{h\left( q\right) }
\\
h_{\ast }^{\prime }\left[ \xi ,\eta ,\gamma \right] ^{\left( q\right) } &=&%
\left[ h_{\ast }^{\prime }\xi ,h_{\ast }^{\prime }\eta ,h_{\ast }^{\prime
}\gamma \right] ^{h\left( q\right) }.
\end{eqnarray*}%
If $A\left( t\right) $ is a path on $\mathbb{L},$ with $A\left( t_{0}\right)
=A_{0},$ and $\left. \frac{d}{dt}A\left( t\right) /A_{0}\right\vert
_{t=t_{0}}=\xi \in \mathfrak{l,}$ then for any $p,q\in \mathbb{L}$, 
\begin{equation}
\left. \frac{d}{dt}p\circ _{A\left( t\right) }q\right\vert _{t=t_{0}}=\left[
p,q,\xi \right] ^{\left( A_{0}\right) }\in T_{p\circ _{A_{0}}q}\mathbb{L}.
\label{ddtmodprod}
\end{equation}%
Also, for any $\eta ,\gamma \in \mathfrak{l},$ 
\begin{equation}
\left. \frac{d}{dt}\left[ \eta ,\gamma \right] ^{A\left( t\right)
}\right\vert _{t=t_{0}}=a_{A_{0}}\left( \eta ,\gamma ,\xi \right)
\label{ddtbrack}
\end{equation}
\end{lemma}

\begin{proof}
The first part is given in \cite[Lemma 3.17]{GrigorianLoops}. To show (\ref%
{ddtmodprod}), consider%
\begin{eqnarray*}
\left. \frac{d}{dt}p\circ _{A\left( t\right) }q\right\vert _{t=t_{0}}
&=&\left. \frac{d}{dt}\left( p\left( qA\left( t\right) \right) \right)
/A\left( t\right) \right\vert _{t=t_{0}} \\
&=&\left. \frac{d}{dt}\left( p\left( qA\left( t\right) \right) \right)
/A_{0}\right\vert _{t=t_{0}}-\left. \frac{d}{dt}\left( \left( p\circ
_{A_{0}}q\right) \cdot A\left( t\right) \right) /A_{0}\right\vert _{t=t_{0}},
\end{eqnarray*}%
where we've used (\ref{ddtrighquot}). Now, 
\begin{eqnarray*}
\left( p\left( qA\left( t\right) \right) \right) /A_{0} &=&\left( p\left(
q\left( A\left( t\right) /A_{0}\cdot A_{0}\right) /A_{0}\cdot A_{0}\right)
\right) /A_{0} \\
&=&p\circ _{A_{0}}\left( q\circ _{A_{0}}\left( A\left( t\right)
/A_{0}\right) \right) \\
\left( \left( p\circ _{A_{0}}q\right) \cdot A\left( t\right) \right) /A_{0}
&=&\left( p\circ _{A_{0}}q\right) \circ _{A_{0}}\left( A\left( t\right)
/A_{0}\right) .
\end{eqnarray*}%
Hence, 
\begin{eqnarray*}
\left. \frac{d}{dt}p\circ _{\left( t\right) }q\right\vert _{t=t_{0}}
&=&\left. \frac{d}{dt}p\circ _{A_{0}}\left( q\circ _{A_{0}}\left( A\left(
t\right) /A_{0}\right) \right) \right\vert _{t=0} \\
&&-\left. \frac{d}{dt}\left( p\circ _{A_{0}}q\right) \circ _{A_{0}}\left(
A\left( t\right) /A_{0}\right) \right\vert _{t=0}.
\end{eqnarray*}%
To show (\ref{ddtbrack}), we could use (\ref{db2}), but more directly, we
can obtain it from the definition, using (\ref{brack2deriv}): 
\begin{eqnarray}
\left[ \xi ,\gamma \right] ^{\left( A\left( t\right) \right) } &=&\left. 
\frac{d^{2}}{d\tau d\tau ^{\prime }}\exp \left( \tau \eta \right) \circ
_{A\left( t\right) }\exp \left( \tau ^{\prime }\gamma \right) \right\vert
_{\tau ,\tau ^{\prime }=0} \\
&&-\left. \frac{d^{2}}{d\tau d\tau ^{\prime }}\exp \left( \tau ^{\prime
}\gamma \right) \circ _{A\left( t\right) }\exp \left( \tau \eta \right)
\right\vert _{\tau ,\tau ^{\prime }=0}.  \notag
\end{eqnarray}%
Then, from (\ref{ddtmodprod}), 
\begin{eqnarray*}
\frac{d}{dt}\left. \exp \left( \tau \eta \right) \circ _{A\left( t\right)
}\exp \left( \tau ^{\prime }\gamma \right) \right\vert _{t=t_{0}} &=&\left[
\exp \left( \tau \eta \right) ,\exp \left( \tau ^{\prime }\gamma \right)
,\xi \right] ^{\left( A_{0}\right) } \\
\frac{d}{dt}\left. \exp \left( \tau ^{\prime }\gamma \right) \circ _{A\left(
t\right) }\exp \left( \tau \eta \right) \right\vert _{t=t_{0}} &=&\left[
\exp \left( \tau ^{\prime }\gamma \right) ,\exp \left( \tau \eta \right)
,\xi \right] ^{\left( A_{0}\right) }
\end{eqnarray*}%
and from the definition (\ref{Lalgassoc}), we obtain (\ref{ddtbrack}).
\end{proof}

Let $\xi \in \mathfrak{l}$ and $s\in \mathbb{L}.$ Also let $A\left( t\right)
=\exp _{s}\left( t\xi \right) $ for $t$ in some interval $I\subset \mathbb{R}%
\ $that contains $0$. Then consider a family $\eta \left( t\right) \in 
\mathfrak{l}$ that satisfies the following initial value problem:%
\begin{equation}
\left\{ 
\begin{array}{c}
\frac{d\eta \left( t\right) }{dt}=\left[ \xi ,\eta \left( t\right) \right]
^{\left( A\left( t\right) s\right) } \\ 
\eta \left( 0\right) =\eta _{0}%
\end{array}%
\right. .  \label{etabrackeq}
\end{equation}%
In other words, this is linear first-order ODE $\dot{\eta}=\mathop{\rm ad}%
\nolimits_{\xi }^{\left( A\left( t\right) s\right) }\eta $ , so for all $%
t\in I$ there exists an evolution operator $U_{\xi }^{\left( s\right)
}\left( t\right) \in GL\left( \mathfrak{l}\right) $, with $U_{\xi }^{\left(
s\right) }\left( 0\right) =\mathop{\rm id}\nolimits_{\mathfrak{l}},$ such
that 
\begin{equation}
\eta \left( t\right) =U_{\xi }^{\left( s\right) }\left( t\right) \eta _{0}.
\label{Utsol}
\end{equation}%
From standard ODE theory, recall that if $\tau ^{\prime },\tau ^{\prime
\prime }\in I,$ then $U_{\xi }^{\left( s\right) }\left( \tau ^{\prime \prime
}\right) U_{\xi }^{\left( s\right) }\left( \tau ^{\prime }\right) ^{-1}$ is
the evolution operator from $\tau ^{\prime }$ to $\tau ^{\prime \prime }$
and is given by: 
\begin{eqnarray}
U_{\xi }^{\left( s\right) }\left( \tau ^{\prime \prime }\right) U_{\xi
}^{\left( s\right) }\left( \tau ^{\prime }\right) ^{-1} &=&\mathop{\rm id}%
\nolimits_{\mathfrak{l}}+\int_{\tau ^{\prime }}^{\tau ^{\prime \prime }}%
\mathop{\rm ad}\nolimits_{\xi }^{\left( \left( \exp _{s}t_{1}\xi \right)
s\right) }dt_{1}  \label{Utsol2} \\
&&+\sum_{n=2}^{\infty }\int_{\tau ^{\prime }}^{\tau ^{\prime \prime
}}\int_{\tau ^{\prime }}^{t_{n}}...\int_{\tau ^{\prime }}^{t_{2}}\mathop{\rm
ad}\nolimits_{\xi }^{\left( \left( \exp _{s}t_{n}\xi \right) s\right) }...%
\mathop{\rm ad}\nolimits_{\xi }^{\left( \left( \exp _{s}t_{1}\xi \right)
s\right) }dt_{1}...dt_{n}.  \notag
\end{eqnarray}%
The following properties of $U_{\xi }^{\left( s\right) }\left( t\right) $
follow immediately.

\begin{lemma}
\label{lemUprop}The evolution operator $U_{\xi }^{\left( s\right) }\left(
t\right) $ satisfies the following properties:

\begin{enumerate}
\item $U_{\tau \xi }^{\left( s\right) }\left( t\right) =U_{\xi }^{\left(
s\right) }\left( \tau t\right) ,$ for any $t$ and $\tau ,$ as long as $\exp
_{s}\left( t\xi \right) $ and $\exp _{s}\left( \tau t\xi \right) $ are both
defined.

\item $U_{\xi }^{\left( s\right) }\left( t\right) \xi =\xi .$

\item If $\mathbb{L}$ is compact, and $\mathfrak{l}$ is equipped with an
inner product, then in a compatible operator norm, there exists a constant $%
C=\sup_{s\in \mathbb{L}}\left\vert b_{s}\right\vert ,$ such 
\begin{equation}
\left\vert U_{\xi }^{\left( s\right) }\left( \tau ^{\prime \prime }\right)
U_{\xi }^{\left( s\right) }\left( \tau ^{\prime }\right) ^{-1}-\mathop{\rm
id}\nolimits_{\mathfrak{l}}\right\vert \leq e^{C\left\vert \tau ^{\prime
\prime }-\tau ^{\prime }\right\vert \left\vert \xi \right\vert }-1.
\label{Uest}
\end{equation}
\end{enumerate}
\end{lemma}

\begin{proof}
Item 1 follows from a change of variables in (\ref{etabrackeq}). For item 2,
consider 
\begin{equation*}
X\left( t\right) =U_{\xi }^{\left( s\right) }\left( t\right) \xi -\xi .
\end{equation*}%
Then, 
\begin{eqnarray*}
\frac{dX\left( t\right) }{dt} &=&\frac{d\left( U_{\xi }^{\left( s\right)
}\left( t\right) \xi \right) }{dt}=\left[ \xi ,U_{\xi }^{\left( s\right)
}\left( t\right) \xi \right] ^{\left( A\left( t\right) s\right) } \\
&=&\left[ \xi ,X\left( t\right) \right] ^{\left( A\left( t\right) s\right) },
\end{eqnarray*}%
since $\left[ \xi ,\xi \right] ^{\left( A\left( t\right) \right) }=0.$
Hence, $X\left( t\right) =U_{\xi }^{\left( s\right) }\left( t\right) X\left(
0\right) ,$ but $X\left( 0\right) =0,$ so $X\left( t\right) =0$ for all $t$.

For the estimate, from (\ref{Utsol2}), we obtain 
\begin{eqnarray*}
\left\vert U_{\xi }^{\left( s\right) }\left( \tau ^{\prime \prime }\right)
U_{\xi }^{\left( s\right) }\left( \tau ^{\prime }\right) ^{-1}-\mathop{\rm
id}\nolimits_{\mathfrak{l}}\right\vert &\leq &\exp \left( \int_{\tau
^{\prime }}^{\tau ^{\prime \prime }}\left\vert \mathop{\rm ad}\nolimits_{\xi
}^{\left( \left( \exp _{s}t\xi \right) s\right) }\right\vert dt\right) -1 \\
&\leq &\exp \left( \left\vert \tau ^{\prime \prime }-\tau ^{\prime
}\right\vert \left\vert \xi \right\vert \sup_{s\in \mathbb{L}}\left\vert
b_{s}\right\vert \right) -1
\end{eqnarray*}%
Now, $s\mapsto \left\vert b_{s}\right\vert $ is a smooth real-valued map on
a compact manifold, and is hence bounded. Therefore, there exists a constant 
$C=\sup_{s\in \mathbb{L}}\left\vert b_{s}\right\vert $ and hence $\sup_{t\in %
\left[ 0,1\right] }\left\vert b_{\left( \exp _{s}t\xi \right) s}\right\vert
\leq C.$ Thus, 
\begin{equation*}
\left\vert U_{\xi }^{\left( s\right) }\left( \tau ^{\prime \prime }\right)
U_{\xi }^{\left( s\right) }\left( \tau ^{\prime }\right) ^{-1}-\mathop{\rm
id}\nolimits_{\mathfrak{l}}\right\vert \leq e^{C\left\vert \tau ^{\prime
\prime }-\tau ^{\prime }\right\vert \left\vert \xi \right\vert }-1.
\end{equation*}
\end{proof}

\begin{remark}
Since $U_{\xi }^{\left( s\right) }\left( t\right) =U_{t\xi }^{\left(
s\right) }\left( 1\right) ,$ for brevity let us denote the operator $U_{\xi
}^{\left( s\right) }\left( t\right) $ by $U_{t\xi }^{\left( s\right) }.$
\end{remark}

If $\mathfrak{l}$ is a Lie algebra, then $\mathop{\rm ad}\nolimits_{\xi
}^{\left( A\left( t\right) s\right) }=\mathop{\rm ad}\nolimits_{\xi }$ is
independent of $t,$ and then $U_{t\xi }^{\left( s\right) }=\exp \left( t%
\mathop{\rm ad}\nolimits_{\xi }\right) =\mathop{\rm Ad}\nolimits_{\exp t\xi
}.$ In the non-associative case, this is no longer true in general, but
needs additional assumptions, as Theorem \ref{thmAd} below shows.

\begin{theorem}
\label{thmAd}Let $s\in \mathbb{L}$, $\xi ,\eta \in \mathfrak{l,}$ and $%
A\left( t\right) =\exp _{s}\left( t\xi \right) .$ Suppose $U_{t\xi }^{\left(
s\right) }$ is the evolution operator for the equation (\ref{etabrackeq}) as
in (\ref{Utsol}). Then, 
\begin{equation}
\mathop{\rm Ad}\nolimits_{A\left( t\right) }^{\left( s\right) }\eta =U_{t\xi
}^{\left( s\right) }\eta +U_{t\xi }^{\left( s\right) }\int_{0}^{t}\left(
U_{\tau \xi }^{\left( s\right) }\right) ^{-1}\left( \rho _{A\left( \tau
\right) }^{\left( s\right) }\right) ^{-1}\left( \left[ \xi ,A\left( \tau
\right) ,\eta \right] ^{\left( s\right) }\right) d\tau .  \label{AdAtsol}
\end{equation}%
Moreover,

\begin{enumerate}
\item If $\mathbb{L}$ is compact, and $\mathfrak{l}$ is equipped with an
inner product, then in a compatible operator norm, there exists a constant $%
C $ that depends only on $\mathbb{L},$ such that, 
\begin{equation}
\left\vert \mathop{\rm Ad}\nolimits_{A\left( t\right) }^{\left( s\right)
}-U_{t\xi }^{\left( s\right) }\right\vert \leq C\left( e^{C\left\vert \xi
\right\vert t}-1\right) .  \label{Adest}
\end{equation}

\item If $\left( \mathbb{L},\circ _{s}\right) $ is left-power-alternative,
then 
\begin{equation}
\mathop{\rm Ad}\nolimits_{A\left( t\right) }^{\left( s\right) }=U_{t\xi
}^{\left( s\right) }.  \label{AdU1}
\end{equation}

\item If $\left( \mathbb{L},\circ _{s}\right) $ is both
left-power-alternative and right-power-alternative, then%
\begin{equation}
\mathop{\rm Ad}\nolimits_{A\left( t\right) }^{\left( s\right) }=\exp \left( t%
\mathop{\rm ad}\nolimits_{\xi }^{\left( s\right) }\right) .  \label{AdU2}
\end{equation}
\end{enumerate}
\end{theorem}

\begin{proof}
Let $x\left( t\right) =\mathop{\rm Ad}\nolimits_{A\left( t\right) }^{\left(
s\right) }\eta $ and note that $x\left( 0\right) =\eta .$ Then, consider the
derivative of $x\left( t\right) .$ Let $B\left( t\right) =\exp _{s}\left(
t\eta \right) .$ From (\ref{floweq5}) we have 
\begin{equation*}
\frac{dA\left( t\right) }{dt}=\rho _{A\left( t\right) }^{\left( s\right)
}\left( \eta \right) =\left. \frac{d}{d\tau }B\left( \tau \right) \circ
_{s}A\left( t\right) \right\vert _{\tau =0}
\end{equation*}%
Then, using (\ref{ddtrighquot}) we have 
\begin{eqnarray}
\frac{dx\left( t\right) }{dt} &=&\left. \frac{d^{2}}{dtd\tau }\left( A\left(
t\right) \circ _{s}\left( B\left( \tau \right) \right) \right) /_{s}A\left(
t\right) \right\vert _{\tau =0}  \notag \\
&=&\left. \frac{d^{2}}{d\tau d\tau ^{\prime }}\left( A\left( \tau ^{\prime
}\right) \circ _{s}A\left( t\right) \right) \circ _{s}\left( B\left( \tau
\right) \right) /_{s}A\left( t\right) \right\vert _{\tau ^{\prime },\tau =0}
\notag \\
&&-\left. \frac{d^{2}}{d\tau d\tau ^{\prime }}(\left( \left( A\left(
t\right) \circ _{s}\left( B\left( \tau \right) \right) \right) /_{s}A\left(
t\right) \right) \circ _{s}\left( A\left( \tau ^{\prime }\right) \circ
_{s}A\left( t\right) \right) )/_{s}A\left( t\right) \right\vert _{\tau
^{\prime },\tau =0}  \notag \\
&=&-\left( \rho _{A\left( t\right) }^{\left( s\right) }\right) ^{-1}\left[
\xi ,A\left( t\right) ,\eta \right] ^{\left( s\right) }+\left. \frac{d}{%
d\tau ^{\prime }}\left( A\left( \tau ^{\prime }\right) \circ _{s}\left( %
\mathop{\rm Ad}\nolimits_{A\left( t\right) }^{\left( s\right) }\eta \circ
_{s}A\left( t\right) \right) \right) /_{s}A\left( t\right) \right\vert
_{\tau ^{\prime }=0}  \label{dxt1} \\
&&-\left. \frac{d}{d\tau }(\mathop{\rm Ad}\nolimits_{A\left( t\right)
}^{\left( s\right) }\eta \circ _{s}\left( A\left( \tau ^{\prime }\right)
\circ _{s}A\left( t\right) \right) )/_{s}A\left( t\right) \right\vert _{\tau
=0}.  \notag
\end{eqnarray}%
Using (\ref{Arprod}), the second term in (\ref{dxt1}) becomes 
\begin{equation*}
\left. \frac{d}{d\tau ^{\prime }}\left( A\left( \tau ^{\prime }\right) \circ
_{s}\left( \mathop{\rm Ad}\nolimits_{A\left( t\right) }^{\left( s\right)
}\eta \circ _{s}A\left( t\right) \right) \right) /_{s}A\left( t\right)
\right\vert _{\tau ^{\prime }=0}=\left. \frac{d}{d\tau ^{\prime }}A\left(
\tau ^{\prime }\right) \circ _{A\left( t\right) s}x\left( t\right)
\right\vert _{\tau ^{\prime }=0}
\end{equation*}%
and similarly, the third term in (\ref{dxt1}) becomes 
\begin{equation*}
\left. \frac{d}{d\tau }(\mathop{\rm Ad}\nolimits_{A\left( t\right) }^{\left(
s\right) }\eta \circ _{s}\left( A\left( \tau ^{\prime }\right) \circ
_{s}A\left( t\right) \right) )/_{s}A\left( t\right) \right\vert _{\tau
=0}=\left. \frac{d}{d\tau ^{\prime }}x\left( t\right) \circ _{A\left(
t\right) s}A\left( \tau ^{\prime }\right) \right\vert _{\tau ^{\prime }=0}.
\end{equation*}
Using (\ref{brack2deriv}) we then conclude that 
\begin{equation}
\frac{dx\left( t\right) }{dt}=\left[ \xi ,x\left( t\right) \right] ^{\left(
A\left( t\right) s\right) }-\left( \rho _{A\left( t\right) }^{\left(
s\right) }\right) ^{-1}\left[ \xi ,A\left( t\right) ,\eta \right] ^{\left(
s\right) }.  \label{Adeq}
\end{equation}%
This is an inhomogeneous linear first order ODE. The homogeneous part is
precisely (\ref{etabrackeq}), and hence we obtain precisely (\ref{AdAtsol}).

For the estimate, suppose $\mathbb{L}$ is compact. Then, using (\ref{AdAtsol}%
) and (\ref{Uest}), we have 
\begin{eqnarray*}
\left\vert \mathop{\rm Ad}\nolimits_{A\left( t\right) }^{\left( s\right)
}-U_{t\xi }^{\left( s\right) }\right\vert &\leq &\left\vert \xi \right\vert
\int_{0}^{t}\left\vert U_{t\xi }^{\left( s\right) }\left( U_{\tau \xi
}^{\left( s\right) }\right) ^{-1}\right\vert \left\vert \left( \rho
_{A\left( \tau \right) }^{\left( s\right) }\right) ^{-1}\left( \left[ \cdot
,A\left( \tau \right) ,\cdot \right] ^{\left( s\right) }\right) \right\vert
d\tau \\
&\leq &\left\vert \xi \right\vert \int_{0}^{t}e^{C\left\vert \xi \right\vert
\left( t-\tau \right) }\left\vert \left( \rho _{A\left( \tau \right)
}^{\left( s\right) }\right) ^{-1}\left( \left[ \cdot ,A\left( \tau \right)
,\cdot \right] ^{\left( s\right) }\right) \right\vert d\tau .
\end{eqnarray*}%
However, $\left( A,s\right) \mapsto \left\vert \rho _{A}^{-1}\left( \left[
\cdot ,A,\cdot \right] ^{\left( s\right) }\right) \right\vert $ is a
real-valued function on a compact manifold, and hence there exists a
constant $C^{\prime }$ which is the supremum of this function over $\mathbb{L%
}\times \mathbb{L}.$ Hence, 
\begin{eqnarray*}
\left\vert \mathop{\rm Ad}\nolimits_{A\left( t\right) }^{\left( s\right)
}-U_{t\xi }^{\left( s\right) }\right\vert &\leq &C^{\prime }\left\vert \xi
\right\vert \int_{0}^{t}e^{C\left\vert \xi \right\vert \left( t-\tau \right)
}d\tau \\
&=&\frac{C^{\prime }}{C}\left( e^{C\left\vert \xi \right\vert t}-1\right) .
\end{eqnarray*}%
Renaming the constant $C,$ we get (\ref{Adest}).

Now if $\left( \mathbb{L},\circ _{s}\right) $ is left-alternative, the
second term on the right hand side of (\ref{Adeq}) vanishes, since 
\begin{equation*}
\left[ \xi ,A\left( t\right) ,\eta \right] ^{\left( s\right) }=\left[ \xi
,\exp _{s}\left( t\xi \right) ,\eta \right] ^{\left( s\right) }=0.
\end{equation*}%
Then, $\mathop{\rm Ad}\nolimits_{A\left( t\right) }^{\left( s\right) }\eta $
satisfies the homogeneous equation, so the solution is just $\mathop{\rm Ad}%
\nolimits_{A\left( t\right) }^{\left( s\right) }\eta =U_{t\xi }^{\left(
s\right) }\eta .$

From (\ref{Adbrack1a}), we have 
\begin{equation*}
\left[ \xi ,x\left( t\right) \right] ^{\left( A\left( t\right) s\right) }=%
\left[ \xi ,x\left( t\right) \right] ^{\left( s\right) }+\left( \rho
_{A\left( t\right) }^{\left( s\right) }\right) ^{-1}a_{s}\left( \xi ,x\left(
t\right) ,A\left( t\right) \right) .
\end{equation*}%
If $\left( \mathbb{L},\circ _{s}\right) $ is left-power-alternative, then
first of all, the associator is skew-symmetric in the first two entries, so 
\begin{equation*}
a_{s}\left( \xi ,x\left( t\right) ,A\left( t\right) \right) =-2\left[
x\left( t\right) ,\xi ,A\left( t\right) \right] ^{\left( s\right) },
\end{equation*}%
however due to right-power-alternativity, $\left[ x\left( t\right) ,\xi
,\exp _{s}\left( t\xi \right) \right] ^{\left( s\right) }=0.$ Hence, in this
case, 
\begin{equation*}
\left[ \xi ,x\left( t\right) \right] ^{\left( A\left( t\right) s\right) }=%
\left[ \xi ,x\left( t\right) \right] ^{\left( s\right) }.
\end{equation*}%
Then, $x\left( t\right) $ satisfies the first order homogeneous ODE with
constant coefficients:%
\begin{equation*}
\frac{dx\left( t\right) }{dt}=\mathop{\rm ad}\nolimits_{\xi }^{\left(
s\right) }\left( x\left( t\right) \right) ,
\end{equation*}%
so $U_{\xi }\left( t\right) =\exp \left( t\mathop{\rm ad}\nolimits_{\xi
}^{\left( s\right) }\right) $ and hence the solution is now 
\begin{equation*}
x\left( t\right) =\exp \left( t\mathop{\rm ad}\nolimits_{\xi }^{\left(
s\right) }\right) \eta .
\end{equation*}
\end{proof}

\begin{corollary}
\label{corPowAssoc}If $\left( \mathbb{L},\circ _{s}\right) $ is
power-associative, then

\begin{enumerate}
\item $\mathop{\rm Ad}\nolimits_{\exp _{s}\left( t\xi \right) }^{\left(
s\right) }\xi =\xi $,

\item $\exp _{s}\left( 2t\xi \right) =\exp _{s}\left( t\xi \right) \circ
_{s}\exp _{s}\left( t\xi \right) .$
\end{enumerate}
\end{corollary}

\begin{proof}
From (\ref{AdAtsol}), 
\begin{equation}
\mathop{\rm Ad}\nolimits_{A\left( t\right) }^{\left( s\right) }\xi =U_{t\xi
}^{\left( s\right) }\xi +U_{t\xi }^{\left( s\right) }\int_{0}^{t}\left(
U_{t\xi }^{\left( s\right) }\right) ^{-1}\rho _{A\left( \tau \right)
}^{-1}\left( \left[ \xi ,A\left( \tau \right) ,\xi \right] ^{\left( s\right)
}\right) d\tau .
\end{equation}%
However, by power-associativity (\ref{powerassoc}), 
\begin{equation*}
\left[ \xi ,A\left( \tau \right) ,\xi \right] ^{\left( s\right) }=\left[ \xi
,\exp _{s}\left( \tau \xi \right) ,\xi \right] ^{\left( s\right) }=0.
\end{equation*}%
Hence, using Lemma \ref{lemUprop}, we obtain 
\begin{equation*}
\mathop{\rm Ad}\nolimits_{A\left( t\right) }^{\left( s\right) }\xi =U_{t\xi
}^{\left( s\right) }\xi =\xi .
\end{equation*}%
For the second part, define 
\begin{equation*}
r\left( t\right) =\exp _{s}\left( t\xi \right) \circ \exp _{s}\left( t\xi
\right) .
\end{equation*}%
Then, informally, we write 
\begin{eqnarray*}
\frac{dr\left( t\right) }{dt} &=&\exp _{s}\left( t\xi \right) \circ
_{s}\left( \xi \circ _{s}\exp _{s}\left( t\xi \right) \right) \\
&&+\left( \xi \circ _{s}\exp _{s}\left( t\xi \right) \right) \circ _{s}\exp
_{s}\left( t\xi \right) .
\end{eqnarray*}%
Now using power-associativity, we see that $\xi $ associates with $\exp
_{s}\left( t\xi \right) $ and since $\mathop{\rm Ad}\nolimits_{\exp
_{s}\left( t\xi \right) }^{\left( s\right) }\xi =\xi ,$ moreover $\xi $ and $%
\exp _{s}\left( t\xi \right) $ commute. Hence, we can rewrite 
\begin{equation*}
\frac{dr\left( t\right) }{dt}=2\xi \circ _{s}\left( \exp _{s}\left( t\xi
\right) \circ _{s}\exp _{s}\left( t\xi \right) \right) =2\xi \circ
_{s}r\left( t\right) .
\end{equation*}%
The solution is thus 
\begin{equation*}
r\left( t\right) =\exp _{s}\left( 2t\xi \right) ,
\end{equation*}%
so by uniqueness of solutions we then have the needed equality.
\end{proof}

\begin{remark}
Corollary \ref{corPowAssoc} thus shows that indeed, power-associativity
allows to extend $\exp _{s}\left( t\xi \right) $ for all $t$. This is a
slightly different proof of this fact compared to \cite%
{Kuzmin1971,Malcev1955}. The result from Corollary \ref{corPowAssoc} also
allows to conclude that if $\left( \mathbb{L},\circ _{s}\right) $ is
power-associative, then 
\begin{equation}
\exp _{s}\left( t\xi \right) \circ _{s}\exp _{s}\left( \tau \xi \right)
=\exp _{s}\left( \left( t+\tau \right) \xi \right) .  \label{expadd}
\end{equation}
\end{remark}

\begin{theorem}
\label{thmdtexp}Suppose $\xi \left( t\right) $ is a path in $\mathfrak{l,}$
with derivative $\dot{\xi}\left( t\right) $, then 
\begin{equation*}
\frac{d}{dt}\left( \exp _{s}\left( \xi \left( t\right) \right) \right) =\rho
_{\exp _{s}\left( \xi \left( t\right) \right) }^{\left( s\right) }U_{\xi
\left( t\right) }^{\left( s\right) }\int_{0}^{1}\left( U_{\tau \xi \left(
t\right) }^{\left( s\right) }\right) ^{-1}\dot{\xi}\left( t\right) d\tau .
\end{equation*}%
Moreover,

\begin{enumerate}
\item If $\left( \mathbb{L},\circ _{s}\right) $ is left-power-alternative,
then 
\begin{equation*}
\frac{d}{dt}\left( \exp _{s}\left( \xi \left( t\right) \right) \right)
=\lambda _{\exp _{s}\left( \xi \left( t\right) \right) }^{\left( s\right)
}\int_{0}^{1}\left( \mathop{\rm Ad}\nolimits_{\exp _{s}\tau \xi \left(
t\right) }^{\left( s\right) }\right) ^{-1}\dot{\xi}\left( t\right) d\tau .
\end{equation*}

\item If $\left( \mathbb{L},\circ _{s}\right) $ is both
left-power-alternative and right-power-alternative, then%
\begin{equation*}
\frac{d}{dt}\left( \exp _{s}\left( \xi \left( t\right) \right) \right)
=\lambda _{\exp _{s}\left( \xi \left( t\right) \right) }^{\left( s\right)
}\int_{0}^{1}\exp \left( -\tau \mathop{\rm ad}\nolimits_{\xi \left( t\right)
}^{\left( s\right) }\right) \dot{\xi}\left( t\right) d\tau .
\end{equation*}
\end{enumerate}
\end{theorem}

\begin{proof}
Using a similar approach as in the Lie group case, let 
\begin{equation*}
\Gamma \left( \tau ,t\right) =\left( \rho _{\exp _{s}\left( \tau \xi \left(
t\right) \right) }^{\left( s\right) }\right) ^{-1}\frac{\partial }{\partial t%
}\exp _{s}\left( \tau \xi \left( t\right) \right) .
\end{equation*}%
We can write this (somewhat informally) as 
\begin{equation}
\Gamma \left( \tau ,t\right) =\left( \frac{\partial }{\partial t}\exp
_{s}\left( \tau \xi \left( t\right) \right) \right) /_{s}\exp _{s}\left(
\tau \xi \left( t\right) \right)
\end{equation}%
Note that $\Gamma \left( 0,t\right) =0.$ Then, consider 
\begin{eqnarray*}
\frac{\partial \Gamma }{\partial \tau } &=&\left( \frac{\partial }{\partial t%
}\left( \xi \left( t\right) \circ _{s}\exp _{s}\left( \tau \xi \left(
t\right) \right) \right) \right) /_{s}\exp _{s}\left( \tau \xi \left(
t\right) \right) \\
&&-\left( \Gamma \left( \tau ,t\right) \circ _{s}\left( \xi \left( t\right)
\circ _{s}\exp _{s}\left( \tau \xi \left( t\right) \right) \right) \right)
/_{s}\exp _{s}\left( \tau \xi \left( t\right) \right) \\
&=&\frac{\partial \xi \left( t\right) }{\partial t}+\left( \xi \left(
t\right) \circ _{s}\frac{\partial }{\partial t}\exp _{s}\left( \tau \xi
\left( t\right) \right) \right) /_{s}\exp _{s}\left( \tau \xi \left(
t\right) \right) \\
&&-\Gamma \left( \tau ,t\right) \circ _{\exp _{s}\left( \tau \xi \left(
t\right) \right) s}\xi \left( t\right) \\
&=&\frac{\partial \xi \left( t\right) }{\partial t}+\left[ \xi \left(
t\right) ,\Gamma \left( \tau ,t\right) \right] ^{\exp _{s}\left( \tau \xi
\left( t\right) \right) s}.
\end{eqnarray*}%
For each $t,$ the homogeneous part of ODE is precisely (\ref{etabrackeq}),
and since the initial condition is $\Gamma \left( 0,t\right) =0,$ we find
that the solution of the inhomogeneous equation is 
\begin{equation*}
\Gamma \left( \tau ,t\right) =U_{\tau \xi \left( t\right) }^{\left( s\right)
}\int_{0}^{\tau }\left( U_{\tau ^{\prime }\xi \left( t\right) }^{\left(
s\right) }\right) ^{-1}\frac{\partial \xi \left( t\right) }{\partial t}d\tau
^{\prime }.
\end{equation*}%
Setting $\tau =1$, we obtain 
\begin{equation*}
\left( \rho _{\exp _{s}\left( \xi \left( t\right) \right) }^{\left( s\right)
}\right) ^{-1}\frac{\partial }{\partial t}\exp _{s}\left( \xi \left(
t\right) \right) =U_{\xi \left( t\right) }^{\left( s\right) }\left( 1\right)
\int_{0}^{1}U_{\xi \left( t\right) }^{\left( s\right) }\left( \tau ^{\prime
}\right) ^{-1}\frac{\partial \xi \left( t\right) }{\partial t}d\tau ^{\prime
}.
\end{equation*}%
The special cases now follow immediately from (\ref{AdU1}) and (\ref{AdU2}).
\end{proof}

\begin{corollary}
Let $\xi ,\eta \in \mathfrak{l},$ then 
\begin{equation}
\left( \rho _{\exp _{s}\xi }^{\left( s\right) }\right) ^{-1}\left. d\exp
_{s}\right\vert _{\xi }\left( \eta \right) =U_{\xi }^{\left( s\right)
}\left( 1\right) \int_{0}^{1}U_{\xi }^{\left( s\right) }\left( \tau ^{\prime
}\right) ^{-1}\eta d\tau ^{\prime }.  \label{dexp1}
\end{equation}%
Moreover, if $\mathbb{L}$ is compact, given a norm on $\mathfrak{l}$ and a
corresponding operator norm, 
\begin{equation}
\left\vert \left( \rho _{\exp _{s}\xi }^{\left( s\right) }\right)
^{-1}\left. d\exp _{s}\right\vert _{\xi }-\mathop{\rm id}\nolimits_{%
\mathfrak{l}}\right\vert \leq e^{C\left\vert \xi \right\vert }-1,
\label{opnorm}
\end{equation}%
where $C>0$ is a constant that depends on $\mathbb{L}.$
\end{corollary}

\begin{proof}
The expression (\ref{dexp1}) follows directly from Theorem \ref{thmdtexp}.
We thus have 
\begin{equation}
\left( \rho _{\exp _{s}\xi }^{\left( s\right) }\right) ^{-1}\left. d\exp
_{s}\right\vert _{\xi }-\mathop{\rm id}\nolimits_{\mathfrak{l}%
}=\int_{0}^{1}\left( U_{\xi }^{\left( s\right) }\left( U_{\tau ^{\prime }\xi
}^{\left( s\right) }\right) ^{-1}-\mathop{\rm id}\nolimits_{l}\right) d\tau
^{\prime }.  \label{dexpmid}
\end{equation}%
From (\ref{dexpmid}) we then obtain (\ref{opnorm}).
\end{proof}

Let us now explore the dependence of $\exp _{s}$ on $s.$ In particular,
suppose we have a smooth $1$-parameter family $s\left( \tau \right) ,$ with $%
s\left( 0\right) =s\in \mathbb{L}$ and $\left. \frac{d}{d\tau }s\left( \tau
\right) \right\vert _{\tau =0}=\rho _{s}\left( \eta \right) \in T_{s}\mathbb{%
L}\mathfrak{,}$ with $\eta \in \mathfrak{l}.$ For each $\tau ,$ $p\left(
t,\tau \right) =\exp _{s\left( \tau \right) }\left( t\xi \right) $ satisfies 
\begin{equation}
\left\{ 
\begin{array}{c}
\frac{dp\left( t,\tau \right) }{dt}=\left. \rho ^{\left( s\left( \tau
\right) \right) }\left( \xi \right) \right\vert _{p\left( t,\tau \right) }
\\ 
p\left( 0,\tau \right) =1%
\end{array}%
\right. .  \label{pttau}
\end{equation}%
Now for each $t,$ $\left. \frac{d}{d\tau }p\left( t,\tau \right) \right\vert
_{\tau =0}\in T_{p\left( t,0\right) }\mathbb{L},$ so define 
\begin{eqnarray}
\sigma \left( t\right) &=&\left( \rho _{p\left( t,0\right) }^{\left(
s\right) }\right) ^{-1}\left. \frac{d}{d\tau }p\left( t,\tau \right)
\right\vert _{\tau =0}  \label{sigma} \\
&=&\left. \frac{d}{d\tau }\left( p\left( t,\tau \right) /_{s}p\left(
t,0\right) \right) \right\vert _{\tau =0},  \notag
\end{eqnarray}%
so that in particular, $\sigma \left( t\right) \in \mathfrak{l}.$

\begin{lemma}
\label{lemSigma}Let $A\left( t\right) =\exp _{s}\left( t\xi \right) $, then
the quantity $\sigma \left( t\right) $ is given by 
\begin{equation}
\sigma \left( t\right) =\left( U_{t\xi }^{\left( s\right) }-\mathop{\rm Ad}%
\nolimits_{A\left( t\right) }^{\left( s\right) }\right) \eta .
\end{equation}%
In particular, if $\left( \mathbb{L},\circ _{s}\right) $ is
left-power-alternative, then $\sigma \left( t\right) =0$ for all $t.$
\end{lemma}

\begin{proof}
From (\ref{pttau}), we have 
\begin{eqnarray*}
\frac{dp\left( t,\tau \right) }{dt} &=&\left. \rho ^{\left( s\left( \tau
\right) \right) }\left( \xi \right) \right\vert _{p\left( t,\tau \right) } \\
&=&\left. \frac{d}{d\varepsilon }\exp _{s\left( \tau \right) }\left(
\varepsilon \xi \right) \circ _{s\left( \tau \right) }p\left( t,\tau \right)
\right\vert _{\varepsilon =0} \\
&=&\left. \frac{d}{d\varepsilon }p\left( \varepsilon ,\tau \right) \circ
_{s\left( \tau \right) }p\left( t,\tau \right) \right\vert _{\varepsilon =0}%
\text{,}
\end{eqnarray*}%
and therefore, 
\begin{eqnarray}
\frac{d\sigma \left( t\right) }{dt} &=&\left. \frac{d^{2}}{dtd\tau }\left(
p\left( t,\tau \right) /_{s}p\left( t,0\right) \right) \right\vert _{\tau =0}
\notag \\
&=&\left. \frac{d^{2}}{d\varepsilon d\tau }\left( p\left( \varepsilon ,\tau
\right) \circ _{s\left( \tau \right) }p\left( t,\tau \right) \right)
/_{s}p\left( t,0\right) \right\vert _{\varepsilon ,\tau =0}  \label{dsigma}
\\
&&-\left. \frac{d^{2}}{d\varepsilon d\tau }\left( p\left( t,\tau \right)
/_{s}p\left( t,0\right) \circ _{s}\left( p\left( \varepsilon ,0\right) \circ
_{s}p\left( t,0\right) \right) \right) /_{s}p\left( t,0\right) \right\vert
_{\varepsilon ,\tau =0},  \notag
\end{eqnarray}%
where we used the derivative of the quotient formula (\ref{ddtrighquot}) and
also the fact that $p\left( 0,\tau \right) =1$ for all $\tau .$ Noting that
for any $\tau $, $\left. \frac{d}{d\varepsilon }p\left( \varepsilon ,\tau
\right) \right\vert _{\varepsilon =0}=\xi ,$ consider the first term of (\ref%
{dsigma}): 
\begin{eqnarray*}
\left. \frac{d^{2}}{d\varepsilon d\tau }\left( p\left( \varepsilon ,\tau
\right) \circ _{s\left( \tau \right) }p\left( t,\tau \right) \right)
/_{s}p\left( t,0\right) \right\vert _{\varepsilon ,\tau =0} &=&\left. \frac{%
d^{2}}{d\varepsilon d\tau }\left( p\left( \varepsilon ,0\right) \circ
_{s\left( \tau \right) }p\left( t,\tau \right) \right) /_{s}p\left(
t,0\right) \right\vert _{\varepsilon ,\tau =0} \\
&=&\left. \frac{d^{2}}{d\varepsilon d\tau }\left( p\left( \varepsilon
,0\right) \circ _{s\left( \tau \right) }p\left( t,0\right) \right)
/_{s}p\left( t,0\right) \right\vert _{\varepsilon ,\tau =0} \\
&&+\left. \frac{d^{2}}{d\varepsilon d\tau }\left( p\left( \varepsilon
,0\right) \circ _{s}p\left( t,\tau \right) \right) /_{s}p\left( t,0\right)
\right\vert _{\varepsilon ,\tau =0}.
\end{eqnarray*}%
Now, 
\begin{eqnarray*}
\left. \frac{d^{2}}{d\varepsilon d\tau }\left( p\left( \varepsilon ,0\right)
\circ _{s\left( \tau \right) }p\left( t,0\right) \right) /_{s}p\left(
t,0\right) \right\vert _{\varepsilon ,\tau =0} &=&\left( \rho _{\exp
_{s}\left( t\xi \right) }^{\left( s\right) }\right) ^{-1}\left[ \xi ,\exp
_{s}\left( t\xi \right) ,\eta \right] ^{\left( s\right) } \\
\left. \frac{d^{2}}{d\varepsilon d\tau }\left( p\left( \varepsilon ,0\right)
\circ _{s}p\left( t,\tau \right) \right) /_{s}p\left( t,0\right) \right\vert
_{\varepsilon ,\tau =0} &=&\left. \frac{d^{2}}{d\varepsilon d\tau }\left(
p\left( \varepsilon ,0\right) \circ _{s}\left( p\left( t,\tau \right)
/_{s}p\left( t,0\right) \right) \circ _{s}p\left( t,0\right) \right)
/_{s}p\left( t,0\right) \right\vert _{\varepsilon ,\tau =0} \\
&=&\left. \frac{d^{2}}{d\varepsilon d\tau }p\left( \varepsilon ,0\right)
\circ _{p\left( t,0\right) s}\left( p\left( t,\tau \right) /_{s}p\left(
t,0\right) \right) \right\vert _{\varepsilon ,\tau =0}.
\end{eqnarray*}%
The second term of (\ref{dsigma}) gives 
\begin{equation*}
\left. \frac{d^{2}}{d\varepsilon d\tau }\left( p\left( t,\tau \right)
/_{s}p\left( t,0\right) \circ _{s}\left( p\left( \varepsilon ,0\right) \circ
_{s}p\left( t,0\right) \right) \right) /_{s}p\left( t,0\right) \right\vert
_{\varepsilon ,\tau =0}=\left. \frac{d^{2}}{d\varepsilon d\tau }\left(
p\left( t,\tau \right) /_{s}p\left( t,0\right) \right) \circ _{p\left(
t,0\right) s}p\left( \varepsilon ,0\right) \right\vert _{\varepsilon ,\tau
=0}.
\end{equation*}%
Overall, (\ref{dsigma}) becomes 
\begin{eqnarray*}
\frac{d\sigma \left( t\right) }{dt} &=&\left. \frac{d^{2}}{d\varepsilon
d\tau }p\left( \varepsilon ,0\right) \circ _{p\left( t,0\right) s}\left(
p\left( t,\tau \right) /_{s}p\left( t,0\right) \right) \right\vert
_{\varepsilon ,\tau =0} \\
&&-\left. \frac{d^{2}}{d\varepsilon d\tau }\left( p\left( t,\tau \right)
/_{s}p\left( t,0\right) \right) \circ _{p\left( t,0\right) s}p\left(
\varepsilon ,0\right) \right\vert _{\varepsilon ,\tau =0} \\
&&+\left( \rho _{\exp _{s}\left( t\xi \right) }^{\left( s\right) }\right)
^{-1}\left[ \xi ,\exp _{s}\left( t\xi \right) ,\eta \right] ^{\left(
s\right) } \\
&=&\left[ \xi ,\sigma \left( t\right) \right] ^{\left( \exp _{s}\left( t\xi
\right) s\right) }+\left( \rho _{\exp _{s}\left( t\xi \right) }^{\left(
s\right) }\right) ^{-1}\left[ \xi ,\exp _{s}\left( t\xi \right) ,\eta \right]
^{\left( s\right) }.
\end{eqnarray*}%
This is precisely the equation (\ref{Adeq}) satisfied by $-\mathop{\rm Ad}%
\nolimits_{A\left( t\right) }^{\left( s\right) }\eta ,$ however with initial
condition $\sigma \left( 0\right) =0.$ Therefore, the solution is 
\begin{equation*}
\sigma \left( t\right) =\left( U_{t\xi }^{\left( s\right) }-\mathop{\rm Ad}%
\nolimits_{A\left( t\right) }^{\left( s\right) }\right) \eta .
\end{equation*}%
If $\left( \mathbb{L},\circ _{s}\right) $ is left-power-associative, then
from Theorem \ref{thmAd}, $\mathop{\rm Ad}\nolimits_{A\left( t\right)
}^{\left( s\right) }=U_{t\xi }^{\left( s\right) },$ and thus $\sigma \left(
t\right) =0$ for all $t$.

We will assume that group $\Psi $ of pseudoautomorphisms of $\mathbb{L}$ is
a finite-dimensional Lie group, and suppose the Lie algebras of $\Psi $ and $%
H_{s}=\mathop{\rm Aut}\nolimits\left( \mathbb{L},\circ _{s}\right) $ are $%
\mathfrak{p}$ and $\mathfrak{h}_{s}$, respectively. In particular, $%
\mathfrak{h}_{s}$ is a Lie subalgebra of $\mathfrak{p}$. Also, we will
assume that $\Psi $ acts transitively on $\mathbb{L}$. The action of $\Psi $
on $\mathbb{L}\ $induces an action of the Lie algebra $\mathfrak{p}$ on $%
\mathfrak{l}$, which we will denote by $\cdot .$
\end{proof}

\begin{definition}
Define the map $\varphi :\mathbb{L}\longrightarrow $ $\mathfrak{l}\otimes 
\mathfrak{p}^{\ast }$ such that for each $s\in \mathbb{L}$ and $\gamma \in 
\mathfrak{p}$, 
\begin{equation}
\varphi _{s}\left( \gamma \right) =\left. \frac{d}{dt}{\left( \exp \left(
t\gamma \right) \left( s\right) \right)}/{s}\right\vert _{t=0}\in \mathfrak{%
l.}  \label{phis}
\end{equation}
\end{definition}

\begin{lemma}[{\protect\cite[Theorem 3.25]{GrigorianLoops}}]
\label{lemGammahatsurj} The map $\varphi $ as in (\ref{phis}) is equivariant
with respect to corresponding actions of $\Psi ,$ in particular for $h\in
\Psi ,$ $s\in \mathbb{L}$, $\gamma \in \mathfrak{p},$ we have%
\begin{equation}
\varphi _{h\left( s\right) }\left( \left( \mathop{\rm Ad}\nolimits%
_{h}\right) _{\ast }\gamma \right) =\left( h^{\prime }\right) _{\ast
}\varphi _{s}\left( \gamma \right) .  \label{phihs}
\end{equation}%
Moreover, the image of $\varphi _{s}$ is $\mathfrak{l}^{\left( s\right) }$
and the kernel is $\mathfrak{h}_{s}$, and hence, $\mathfrak{p\cong h}%
_{s}\oplus \mathfrak{l}^{\left( s\right) }$.
\end{lemma}

\begin{lemma}[{\protect\cite[Lemma 3.33 and Lemma 3.35]{GrigorianLoops}}]
Suppose $\xi \in \mathfrak{p}$ and $\eta ,\gamma \in \mathfrak{l},$ then 
\begin{subequations}
\begin{eqnarray}
\xi \cdot \left[ \eta ,\gamma \right] ^{\left( s\right) } &=&\left[ \xi
\cdot \eta ,\gamma \right] ^{\left( s\right) }+\left[ \eta ,\xi \cdot \gamma %
\right] ^{\left( s\right) }+a_{s}\left( \eta ,\gamma ,\varphi _{s}\left( \xi
\right) \right)   \label{xilbrack} \\
\xi \cdot \varphi _{s}\left( \eta \right)  &=&\eta \cdot \varphi _{s}\left(
\xi \right) +\varphi _{s}\left( \left[ \xi ,\eta \right] _{\mathfrak{p}%
}\right) +\left[ \varphi _{s}\left( \xi \right) ,\varphi _{s}\left( \eta
\right) \right] ^{\left( s\right) }.  \label{xiphi}
\end{eqnarray}%
\end{subequations}%
\end{lemma}

Similarly as for Lie groups, we may define a Killing form $K^{\left(
s\right) }$ on $\mathfrak{l}^{\left( s\right) }$. For $\xi ,\eta \in 
\mathfrak{l}$, we have 
\begin{equation}
K^{\left( s\right) }\left( \xi ,\eta \right) =\mathop{\rm Tr}\left( %
\mathop{\rm ad}\nolimits_{\xi }^{\left( s\right) }\circ \mathop{\rm ad}%
\nolimits_{\eta }^{\left( s\right) }\right) ,  \label{Killing}
\end{equation}%
where $\circ $ is just composition of linear maps on $\mathfrak{l}$ and $%
\mathop{\rm ad}\nolimits_{\xi }^{\left( s\right) }\left( \cdot \right) =%
\left[ \xi ,\cdot \right] ^{\left( s\right) }$. Clearly $K^{\left( s\right)
} $ is a symmetric bilinear form on $\mathfrak{l}.$ In \cite{GrigorianLoops}
it is shown that for $h\in \Psi ,$ and $\xi ,\eta \in \mathfrak{l}$ it
satisfies $K^{\left( h\left( s\right) \right) }\left( h_{\ast }^{\prime }\xi
,h_{\ast }^{\prime }\eta \right) =K^{\left( s\right) }\left( \xi ,\eta
\right) $ .

General criteria for a loop algebra to admit a non-degenerate Killing form
are currently not known, but it is known \cite{LoosMalcev} that for a \emph{%
semisimple} \emph{Malcev} algebra, the Killing form is non-degenerate. A 
\emph{Malcev }algebra is the tangent algebra of a Moufang loop and is an 
\emph{alternative} algebra that also satisfies the following identity \cite%
{Kuzmin1971,Malcev1955}:%
\begin{equation}
\left[ \xi ,\gamma ,\left[ \xi ,\eta \right] ^{\left( s\right) }\right]
^{\left( s\right) }=\left[ \left[ \xi ,\gamma ,\eta \right] ^{\left(
s\right) },\xi \right] ^{\left( s\right) }.  \label{MalcevId}
\end{equation}%
Moreover, in this case, $K^{\left( s\right) }$ is $\mathfrak{p}$-invariant
and $\mathop{\rm ad}\nolimits^{\left( s\right) }$-invariant \cite%
{GrigorianLoops}. Suppose $s\left( t\right) =\exp _{s}\left( t\gamma \right)
s,$ then from (\ref{ddtbrack}), we see that generally,

\begin{eqnarray}
\frac{dK^{\left( s\left( t\right) \right) }}{dt}\left( \xi ,\xi \right)  &=&%
\frac{d}{dt}\left( \mathop{\rm Tr}\left( \left[ \xi ,\left[ \xi ,\cdot %
\right] ^{\left( s\left( t\right) \right) }\right] ^{\left( s\left( t\right)
\right) }\right) \right)   \label{dKsdt} \\
&=&2\mathop{\rm Tr}\left( \left[ \xi ,a_{s\left( t\right) }\left( \xi ,\cdot
,\gamma \right) \right] ^{\left( s\left( t\right) \right) }\right) .  \notag
\end{eqnarray}%
In the special case of $\mathbb{L}$ being a Moufang loop, and thus every $%
\mathfrak{l}^{\left( s\right) }$ being a Malcev algebra, we have the
following.

\begin{lemma}
\label{lemKMalcev}Suppose $\mathbb{L}$ is a Moufang loop. Then, $K^{\left(
s\right) }$ is independent of $s$ and for each $\gamma \in \mathfrak{l},$
then map $\mathop{\rm ad}\nolimits_{\gamma }^{\left( s\right) }$ is
skew-adjoint with respect to $K^{\left( s\right) }.$
\end{lemma}

\begin{proof}
If $\mathbb{L}$ is Moufang, then any $\left( \mathbb{L},\circ _{s}\right) $
is also a Moufang loop, and hence for any $s$, $\mathfrak{l}^{\left(
s\right) }$ is a Malcev algebra. Since a Malcev algebra is alternative, $%
a^{\left( s\right) }\left( \cdot ,\cdot ,\cdot \right) =2\left[ \cdot ,\cdot
,\cdot \right] ^{\left( s\right) },$ and is moreover totally skew-symmetric.
In particular, the Malcev identity (\ref{MalcevId}) can be written as%
\begin{equation}
a_{s}\left( \xi ,\gamma ,\left[ \xi ,\eta \right] ^{\left( s\right) }\right)
=\left[ a_{s}\left( \xi ,\gamma ,\eta \right) ,\xi \right] ^{\left( s\right)
}.  \label{Malcev2}
\end{equation}%
In particular, taking the trace, we get 
\begin{eqnarray}
\mathop{\rm Tr}a_{s}\left( \xi ,\gamma ,\left[ \xi ,\cdot \right] ^{\left(
s\right) }\right) &=&\mathop{\rm Tr}\left[ a_{s}\left( \xi ,\gamma ,\cdot
\right) ,\xi \right] ^{\left( s\right) }  \notag \\
&=&-\mathop{\rm Tr}\left[ a_{s}\left( \xi ,\gamma ,\left[ \xi ,\cdot \right]
\right) \right] ^{\left( s\right) }  \notag \\
&=&0.  \label{Malcev3}
\end{eqnarray}%
Then, (\ref{dKsdt}) gives $\frac{dK^{\left( s\left( t\right) \right) }}{dt}%
\left( \xi ,\xi \right) =0.$ This shows that $K^{\left( s\right) }$ is
constant on $\mathbb{L}.$

For the second part, from the generalized Jacobi identity (\ref{Jac2}), we
obtain 
\begin{eqnarray}
K^{\left( s\right) }\left( \left[ \gamma ,\eta \right] ^{\left( s\right)
},\xi \right) &=&-K^{\left( s\right) }\left( \eta ,\left[ \gamma ,\xi \right]
^{\left( s\right) }\right)  \label{Ksad1} \\
&&+\mathop{\rm Tr}\left[ \eta ,a_{s}\left( \cdot ,\xi ,\gamma \right)
+a_{s}\left( \xi ,\gamma ,\cdot \right) +a_{s}\left( \gamma ,\cdot ,\xi
\right) \right] ^{\left( s\right) }  \notag \\
&&+\mathop{\rm Tr}\left[ \xi ,a_{s}\left( \cdot ,\eta ,\gamma \right)
+a_{s}\left( \eta ,\gamma ,\cdot \right) +a_{s}\left( \gamma ,\cdot ,\eta
\right) \right] ^{\left( s\right) }.  \notag
\end{eqnarray}%
However, for an alternative algebra, this simplifies to 
\begin{eqnarray}
K^{\left( s\right) }\left( \left[ \gamma ,\eta \right] ^{\left( s\right)
},\xi \right) &=&-K^{\left( s\right) }\left( \eta ,\left[ \gamma ,\xi \right]
^{\left( s\right) }\right) \\
&&+3\mathop{\rm Tr}\left[ \eta ,a_{s}\left( \cdot ,\xi ,\gamma \right) %
\right] ^{\left( s\right) }+3\mathop{\rm Tr}\left[ \xi ,a_{s}\left( \cdot
,\eta ,\gamma \right) \right] ^{\left( s\right) }.  \notag
\end{eqnarray}%
The second line is symmetric in $\xi $ and $\eta $, so it is sufficient to
consider the case $\xi =\eta .$ Indeed, for $\xi =\eta ,$using (\ref{Malcev3}%
), this vanishes, so we get 
\begin{equation}
K^{\left( s\right) }\left( \mathop{\rm ad}\nolimits_{\gamma }^{\left(
s\right) }\eta ,\xi \right) =-K^{\left( s\right) }\left( \eta ,\mathop{\rm
ad}\nolimits_{\gamma }^{\left( s\right) }\xi \right) .  \label{adskew}
\end{equation}
\end{proof}

\begin{remark}
\label{remMalc}Note that in Lemma \ref{lemKMalcev}, we only used the trace
of the Malcev identity. The non-degeneracy of the Killing form in a
semi-simple Malcev algebra also hinges on the property (\ref{adskew}), same
as for semi-simple Lie algebras. This suggests that weaker conditions could
be sufficient for these key properties.
\end{remark}

\section{Loop-valued maps}

\setcounter{equation}{0}\label{sectMaps} Let $M$ be a smooth, $n$%
-dimensional manifold and let $s:M\longrightarrow \mathbb{L}$ be a smooth
map. The map $s$ can be used to define a product on $\mathbb{L}^{\prime }$%
-valued maps from $M$ and a corresponding bracket on $\mathfrak{l}$-valued
maps. Indeed, let $A,B:M\longrightarrow \mathbb{L}^{\prime }$ and $\xi ,\eta
:M\longrightarrow \mathfrak{l}$ be smooth maps, then at each $x\in M$,
define 
\begin{subequations}%
\label{maniproducts} 
\begin{eqnarray}
\left. As\right\vert _{x} &=&A_{x}s_{x}\in \mathbb{L} \\
\left. A\circ _{s}B\right\vert _{x} &=&A_{x}\circ _{s_{x}}B_{x}\in \mathbb{L}
\\
\left. A/_{s}B\right\vert _{x} &=&A_{x}/_{s_{x}}B_{x}\in \mathbb{L} \\
\left. \left[ \xi ,\eta \right] ^{\left( s\right) }\right\vert _{x} &=&\left[
\xi _{x},\eta _{x}\right] ^{\left( s_{x}\right) }\in \mathfrak{l.}
\end{eqnarray}%
\end{subequations}%
In particular, the bracket $\left[ \cdot ,\cdot \right] ^{\left( s\right) }$
defines the map $b_{s}:M\longrightarrow \Lambda ^{2}\mathfrak{l}^{\ast
}\otimes \mathfrak{l.}$ We also have the corresponding associator $\left[
\cdot ,\cdot ,\cdot \right] ^{\left( s\right) }$ and the left-alternating
associator map $a_{s}:M\longrightarrow \Lambda ^{2}\mathfrak{l}^{\ast
}\otimes \mathfrak{l}^{\ast }\otimes \mathfrak{l}.$ Similarly, define the
map $\varphi _{s}:M\longrightarrow \mathfrak{p}^{\ast }\otimes \mathfrak{l}.$

Then, similarly as for maps to Lie groups, we may define the (right) \emph{%
Darboux derivative} $\theta _{s}$ of $s,$ which is an $\mathfrak{l}$-valued $%
1$-form on $M$ given by the pull-back $s^{\ast }\theta $ of the
Maurer-Cartan form on $\mathbb{L}$ \cite{SharpeBook}. In particular, at
every $x\in M$, 
\begin{equation}
\left. \left( \theta _{s}\right) \right\vert _{x}=\rho _{s\left( x\right)
}^{-1}\left. ds\right\vert _{x},  \label{Darbouxf}
\end{equation}%
and for any vector $X\in T_{x}M$%
\begin{equation*}
\left. \left( \theta _{s}\right) \right\vert _{x}\left( X\right) =\left.
\theta \right\vert _{s\left( x\right) }\left( \left. ds\right\vert
_{x}\left( X\right) \right) \in \mathfrak{l.}
\end{equation*}%
It is then clear that $\theta _{s}$, being a pullback of $\theta $,
satisfies the loop Maurer-Cartan structural equation (\ref{dtheta}). In
particular, for any vectors $X,Y\in T_{x}M$, 
\begin{equation}
d\theta _{s}\left( X,Y\right) -\left[ \theta _{s}\left( X\right) ,\theta
_{s}\left( Y\right) \right] ^{\left( s\right) }=0.  \label{DarbouxMC}
\end{equation}

We can then calculate the derivatives of these maps (\ref{maniproducts}).

\begin{theorem}[{\protect\cite[Theorem 3.51]{GrigorianLoops}}]
\label{thmmaniDeriv}Let $M$ be a smooth manifold and suppose $s\in C^{\infty
}\left( M,\mathbb{L}\right) \ $and $A,B\in C^{\infty }\left( M,\mathbb{L}%
^{\prime }\right) ,$ then 
\begin{subequations}
\begin{eqnarray}
d\left( As\right) &=&\rho _{s}\left( dA\right) +\lambda _{A}\left( ds\right)
\label{dAsB1} \\
d\left( A\circ _{s}B\right) &=&\rho _{B}^{\left( s\right) }\left( dA\right)
+\lambda _{A}^{\left( s\right) }\left( dB\right) +\left[ A,B,\theta _{s}%
\right] ^{\left( s\right) } \\
d\left( A/_{s}B\right) &=&\left( \rho _{B}^{\left( s\right) }\right)
^{-1}dA-\left( \rho _{B}^{\left( s\right) }\right) ^{-1}\left( \rho
_{A/_{s}B}^{\left( s\right) }dB\right)  \label{drquot} \\
&&-\left( \rho _{B}^{\left( s\right) }\right) ^{-1}\left[ A/_{s}B,B,\theta
_{s}\right] ^{\left( s\right) }
\end{eqnarray}%
\end{subequations}%
Suppose now $\xi ,\eta \in C^{\infty }\left( M,\mathfrak{l}\right) $, then 
\begin{equation}
d\left[ \xi ,\eta \right] ^{\left( s\right) }=\left[ d\xi ,\eta \right]
^{\left( s\right) }+\left[ \xi ,d\eta \right] ^{\left( s\right)
}+a_{s}\left( \xi ,\eta ,\theta _{s}\right) .  \label{dbrack}
\end{equation}

The $\mathfrak{l}\otimes \mathfrak{p}^{\ast }$-valued map $\varphi
_{s}:M\longrightarrow $ $\mathfrak{l}\otimes \mathfrak{p}^{\ast }$ satisfies 
\begin{equation}
d\varphi _{s}=\mathop{\rm id}\nolimits_{\mathfrak{p}}\cdot \theta _{s}-\left[
\varphi _{s},\theta _{s}\right] ^{\left( s\right) },  \label{dphis0}
\end{equation}%
where $\mathop{\rm id}\nolimits_{\mathfrak{p}}\ $is the identity map of $%
\mathfrak{p}$ and $\cdot $ denotes the action of the Lie algebra $\mathfrak{p%
}$ on $\mathfrak{l} $ .

The Killing form $K^{\left( s\right) }:M\longrightarrow \mathop{\rm Sym}%
\nolimits^{2}\mathfrak{l}^{\ast }$ satisfies 
\begin{equation}
dK^{\left( s\right) }\left( \xi ,\xi \right) =2\mathop{\rm Tr}%
\nolimits\left( \left[ \xi ,a_{s}\left( \xi ,\cdot ,\theta _{s}\right) %
\right] ^{\left( s\right) }\right) .  \label{dKs}
\end{equation}
\end{theorem}

Given $A:M\longrightarrow \mathbb{L}^{\prime }$ and $s:M\longrightarrow 
\mathbb{L},$ as shown in \cite{GrigorianLoops}, we have the following
expression for $\theta _{As}$%
\begin{equation}
\theta _{As}=\mathop{\rm Ad}\nolimits_{A}^{\left( s\right) }\theta
_{s}+\left( \rho _{A}^{\left( s\right) }\right) ^{-1}dA.  \label{thetaAs2}
\end{equation}%
Moreover, let us consider the evolution equation satisfied by $\theta
_{A\left( t\right) s}$ for $A\left( t\right) =\exp _{s}\left( t\xi \right) $
for some $\xi :M\longrightarrow \mathfrak{l}.$ This gives us the following.

\begin{lemma}
\label{lemThetaPath}Let $A\left( t\right) =\exp _{s}\left( t\xi \right) ,$
for $\xi \in C^{\infty }\left( M,\mathfrak{l}\right) ,$ then%
\begin{equation}
\frac{d\theta _{A\left( t\right) s}}{dt}=\left[ \xi ,\theta _{A\left(
t\right) s}\right] ^{A\left( t\right) s}+d\xi   \label{dthetadt}
\end{equation}%
and hence%
\begin{equation}
\theta _{A\left( t\right) s}=U_{t\xi }^{\left( s\right) }\theta _{s}+\left(
U_{t\xi }^{\left( s\right) }\int_{0}^{t}\left( U_{\tau \xi }^{\left(
s\right) }\right) ^{-1}d\tau \right) d\xi .  \label{thetaexp}
\end{equation}%
Moreover, if $\mathbb{L}$ is compact and given a metric on $M$ and an inner
product on $\mathfrak{l},$ 
\begin{equation}
\left\vert \theta _{A\left( t\right) s}\right\vert \leq e^{Ct\left\vert \xi
\right\vert }\left( \left\vert \theta _{s}\right\vert +t\left\vert d\xi
\right\vert \right) .  \label{thetaAsest}
\end{equation}
\end{lemma}

\begin{proof}
We will write symbolically 
\begin{equation*}
\frac{dA}{dt}=\xi \circ _{s}A,
\end{equation*}%
so 
\begin{eqnarray*}
\frac{d\theta _{A\left( t\right) s}}{dt} &=&\frac{d}{dt}\left( d\left(
As\right) /As\right)  \\
&=&d\left( \frac{dA}{dt}s\right) /As \\
&&-\left( \left( d\left( As\right) /As\right) \cdot \left( \frac{dA}{dt}%
s\right) \right) /As \\
&=&d\left( \left( \xi \circ _{s}A\right) s\right) /As \\
&&-\left( \theta _{As}\cdot \left( \left( \xi \circ _{s}A\right) s\right)
\right) /As \\
&=&d\left( \xi \left( As\right) \right) /As \\
&&+\left( \theta _{As}\cdot \left( \xi \left( As\right) \right) \right) /As
\\
&=&d\xi +\left( \xi \cdot d\left( As\right) \right) /As \\
&&-\theta _{As}\circ _{As}\xi  \\
&=&d\xi +\left[ \xi ,\theta _{As}\right] ^{\left( As\right) }
\end{eqnarray*}%
Solving this ODE, with $\theta _{A\left( 0\right) s}=\theta _{s}$, we find (%
\ref{thetaexp}). To obtain the estimate, we first have 
\begin{equation*}
\left\vert \theta _{A\left( t\right) s}\right\vert \leq \left\vert U_{t\xi
}^{\left( s\right) }\right\vert \left\vert \theta _{s}\right\vert
+\left\vert d\xi \right\vert \int_{0}^{t}\left\vert U_{t\xi }^{\left(
s\right) }\left( U_{\tau \xi }^{\left( s\right) }\right) ^{-1}\right\vert
d\tau ,
\end{equation*}%
but from (\ref{Uest}), $\left\vert U_{t\xi }^{\left( s\right) }\right\vert
\leq e^{Ct\left\vert \xi \right\vert }\ $and $\left\vert U_{t\xi }^{\left(
s\right) }\left( U_{\tau \xi }^{\left( s\right) }\right) ^{-1}\right\vert
\leq e^{C\left( t-\tau \right) \left\vert \xi \right\vert },$ so 
\begin{equation}
\int_{0}^{t}\left\vert U_{t\xi }^{\left( s\right) }\left( U_{\tau \xi
}^{\left( s\right) }\right) ^{-1}\right\vert d\tau \leq te^{Ct\left\vert \xi
\right\vert },  \label{intUest}
\end{equation}%
since $t\geq 0,$ and hence, 
\begin{equation*}
\left\vert \theta _{A\left( t\right) s}\right\vert \leq e^{Ct\left\vert \xi
\right\vert }\left\vert \theta _{s}\right\vert +te^{Ct\left\vert \xi
\right\vert }\left\vert d\xi \right\vert ,
\end{equation*}%
and thus indeed, we obtain (\ref{thetaAsest}).
\end{proof}

In a very similar fashion we obtain the same results for $\varphi _{s}.$

\begin{lemma}
Let $A\left( t\right) =\exp _{s}\left( t\xi \right) $ and $\gamma \in 
\mathfrak{p},$ then%
\begin{equation}
\frac{d\varphi _{A\left( t\right) s}\left( \gamma \right) }{dt}=\left[ \xi
,\varphi _{A\left( t\right) s}\left( \gamma \right) \right] ^{A\left(
t\right) s}+\gamma \cdot \xi
\end{equation}%
and hence%
\begin{equation}
\varphi _{A\left( t\right) s}\left( \gamma \right) =U_{t\xi }^{\left(
s\right) }\varphi _{s}\left( \gamma \right) +\left( U_{t\xi }^{\left(
s\right) }\int_{0}^{t}\left( U_{\tau \xi }^{\left( s\right) }\right)
^{-1}d\tau \right) \left( \gamma \cdot \xi \right) .
\end{equation}%
Moreover, if $\mathbb{L}$ is compact and given a metric on $M$ and an inner
product on $\mathfrak{l},$ 
\begin{equation}
\left\vert \varphi _{A\left( t\right) s}\left( \gamma \right) \right\vert
\leq e^{Ct\left\vert \xi \right\vert }\left( \left\vert \varphi _{s}\left(
\gamma \right) \right\vert +t\left\vert \gamma \cdot \xi \right\vert \right)
.
\end{equation}
\end{lemma}

Thus we see that three important quantities $\theta _{A\left( t\right) s}$ , 
$U_{t\xi }^{\left( s\right) },$ $\varphi _{A\left( t\right) s}$ satisfy
similar ODEs:%
\begin{eqnarray*}
\frac{d\theta _{A\left( t\right) s}}{dt} &=&\left[ \xi ,\theta _{A\left(
t\right) s}\right] ^{A\left( t\right) s}+d\xi \\
\frac{dU_{t\xi }^{\left( s\right) }}{dt} &=&\left[ \xi ,U_{t\xi }^{\left(
s\right) }\right] ^{A\left( t\right) s} \\
\frac{d\varphi _{A\left( t\right) s}}{dt} &=&\left[ \xi ,\varphi _{A\left(
t\right) s}\right] ^{A\left( t\right) s}+\mathop{\rm id}\nolimits\cdot \xi .
\end{eqnarray*}%
Suppose we have an affine connection $\nabla $ on $M$, then by
differentiating the above ODEs, we can obtain expressions for derivatives of 
$\theta _{As},\ U_{\xi }^{\left( s\right) }$, and $\varphi _{As}.$ However,
first we have a helpful technical lemma.

\begin{lemma}
Suppose $A:M\longrightarrow \mathbb{L}^{\prime }$ and $s:M\longrightarrow 
\mathbb{L},$ and $\alpha _{A,s}$ is a $k$-linear map on $\mathfrak{l}.$
Then, for $\xi _{1},...,\xi _{k}\in \mathfrak{l},$ 
\begin{eqnarray}
\left( d\alpha _{A,s}\right) \left( \xi _{1},...,\xi _{k}\right) &=&\alpha
_{A,s}^{\left( 1\right) }\left( \xi _{1},...,\xi _{k},\theta _{s}\right)
\label{dalphaAs} \\
&&+\alpha _{A,s}^{\left( 2\right) }\left( \xi _{1},...,\xi _{k},\theta
_{As}\right) ,  \notag
\end{eqnarray}%
where $\alpha _{A,s}^{\left( 1\right) }$ and $\alpha _{A,s}^{\left( 2\right)
}$ are $\left( k+1\right) $-linear maps on $\mathfrak{l}.$

In particular, given a metric on $M$ and a norm on $\mathfrak{l,}$ we have
the following pointwise bound%
\begin{equation}
\left\vert \nabla ^{k}db_{s}\right\vert \leq f^{\left( k\right) }\left(
s\right) \sum_{\left( i_{1},...,i_{k+1}\right) \in I}\left\vert \theta
_{s}\right\vert ^{i_{1}}\left\vert \nabla \theta _{s}\right\vert
^{i_{2}}...\left\vert \nabla ^{k}\theta _{s}\right\vert ^{i_{k+1}}
\label{nablakbest}
\end{equation}%
where $I=\left\{ \left( i_{1},...,i_{k+1}\right) \in \mathbb{N}_{0}^{k+1},\ 
\text{such that }\sum_{m=1}^{k+1}mi_{m}=k+1\right\} $ and $f^{\left(
k\right) }:\mathbb{L}\longrightarrow \mathbb{R}_{+}\ $is a continuous
function for each $k$.
\end{lemma}

\begin{proof}
Since $\alpha $ depends on $s$ and $A,$ $d\alpha _{A,s}$, by chain rule we
have 
\begin{eqnarray*}
d\alpha _{A,s} &=&\frac{\partial \alpha }{\partial s}\circ ds+\frac{\partial
\alpha }{\partial A}\circ dA \\
&=&\left( \frac{\partial \alpha }{\partial s}\circ \rho _{s}\right) \left(
\rho _{s}^{-1}ds\right) +\left( \frac{\partial \alpha }{\partial A}\circ
\rho _{A}^{\left( s\right) }\right) \left( \left( \rho _{A}^{\left( s\right)
}\right) ^{-1}dA\right) \\
&=&\left( \frac{\partial \alpha }{\partial s}\circ \rho _{s}+\frac{\partial
\alpha }{\partial A}\circ \rho _{A}^{\left( s\right) }\circ \mathop{\rm Ad}%
\nolimits_{A}^{\left( s\right) }\right) \theta _{s} \\
&&+\left( \frac{\partial \alpha }{\partial A}\circ \rho _{A}^{\left(
s\right) }\right) \theta _{As},
\end{eqnarray*}%
where we have used (\ref{thetaAs2}). So now we can set 
\begin{eqnarray*}
\alpha _{A,s}^{\left( 1\right) }\left( \xi _{1},...,\xi _{k},\xi
_{k+1}\right) &=&\left( \frac{\partial \alpha \left( \xi _{1},...,\xi
_{k}\right) }{\partial s}\circ \rho _{s}\right) \xi _{k+1} \\
&&+\left( \frac{\partial \alpha \left( \xi _{1},...,\xi _{k}\right) }{%
\partial A}\circ \rho _{A}^{\left( s\right) }\circ \mathop{\rm Ad}%
\nolimits_{A}^{\left( s\right) }\right) \xi _{k+1} \\
\alpha _{A,s}^{\left( 2\right) }\left( \xi _{1},...,\xi _{k},\xi
_{k+1}\right) &=&\left( \frac{\partial \alpha \left( \xi _{1},...,\xi
_{k}\right) }{\partial A}\circ \rho _{A}^{\left( s\right) }\right) \xi
_{k+1}.
\end{eqnarray*}%
Thus, we obtain (\ref{dalphaAs}).

From (\ref{dbrack}), we know that for $\xi ,\eta \in \mathfrak{l},$ 
\begin{equation*}
\left( db_{s}\right) \left( \xi ,\eta \right) =a_{s}\left( \xi ,\eta ,\theta
_{s}\right) ,
\end{equation*}%
where the alternating associator $a_{s}$ is a trilinear form on $\mathfrak{l.%
}$ Hence, we have the following estimate 
\begin{equation*}
\left\vert db_{s}\right\vert \leq C\left\vert a_{s}\right\vert \left\vert
\theta _{s}\right\vert ,
\end{equation*}%
where $C$ is some universal constant. However $a_{s}$ is smooth in $s$, so $%
\left\vert a_{s}\right\vert $ is in particular a continuous real-valued
function on $\mathbb{L}.$ Hence we can write 
\begin{equation}
\left\vert db_{s}\right\vert \leq f^{\left( 0\right) }\left( s\right)
\left\vert \theta _{s}\right\vert  \label{dbsest}
\end{equation}%
for some positive real-valued function $f^{0}.$ Now, as we have just shown, 
\begin{equation*}
da_{s}\left( \xi _{1},\xi _{2},\xi _{3},\xi _{4}\right) =a_{s}^{\left(
1\right) }\left( \xi _{1},\xi _{2},\xi _{3},\theta _{s}\right) ,
\end{equation*}%
for some $4$-linear form $a_{s}^{\left( 1\right) }.$ Therefore, for a vector
field $X_{1}$ on $M$, 
\begin{eqnarray*}
\left( \nabla _{X_{1}}db_{s}\right) \left( \xi ,\eta \right) &=&\left(
d_{X_{1}}a_{s}\right) \left( \xi ,\eta ,\theta _{s}\right) +a_{s}\left( \xi
,\eta ,\nabla _{X_{1}}\theta _{s}\right) \\
&=&a_{s}^{\left( 1\right) }\left( \xi ,\eta ,\theta _{s},\theta _{s}\left(
X_{1}\right) \right) +a_{s}\left( \xi ,\eta ,\nabla _{X_{1}}\theta
_{s}\right) ,
\end{eqnarray*}%
so, we have the following estimate 
\begin{eqnarray}
\left\vert \nabla db\right\vert &\leq &C\left( \left\vert a_{s}^{\left(
1\right) }\right\vert \left\vert \theta _{s}\right\vert ^{2}+\left\vert
a_{s}\right\vert \left\vert \nabla \theta _{s}\right\vert \right)  \notag \\
&\leq &f^{\left( 1\right) }\left( s\right) \left( \left\vert \theta
_{s}\right\vert ^{2}+\left\vert \nabla \theta _{s}\right\vert \right) ,
\label{d2best}
\end{eqnarray}%
where $f^{\left( 1\right) }\left( s\right) =C\max \left( \left\vert
a_{s}^{\left( 1\right) }\right\vert ,\left\vert a_{s}\right\vert \right) $
is continuous.

Similarly, we obtain the expression for the second derivative of $db_{s},$
for a multilinear maps $a_{s}^{\left( 2\right) }$ on $\mathfrak{l:}$ 
\begin{eqnarray*}
\left( \nabla _{X_{2}}\nabla _{X_{1}}db_{s}\right) \left( \xi ,\eta \right)
&=&a_{s}^{\left( 2\right) }\left( \xi ,\eta ,\theta _{s},\theta _{s}\left(
X_{1}\right) ,\theta _{s}\left( X_{2}\right) \right) +a_{s}^{\left( 1\right)
}\left( \xi ,\eta ,\nabla _{X_{2}}\theta _{s},\theta _{s}\left( X_{1}\right)
\right) \\
&&+a_{s}^{\left( 1\right) }\left( \xi ,\eta ,\nabla _{X_{1}}\theta
_{s},\theta _{s}\left( X_{2}\right) \right) +a_{s}^{\left( 1\right) }\left(
\xi ,\eta ,\theta _{s},\nabla _{X_{2}}\theta _{s}\left( X_{1}\right) \right)
\\
&&+a_{s}\left( \xi ,\eta ,\nabla _{X_{2}}\nabla _{X_{1}}\theta _{s}\right) ,
\end{eqnarray*}%
where $X_{1}$ and $X_{2}$ are vector fields on $M$. Hence, 
\begin{eqnarray*}
\left\vert \nabla ^{2}db_{s}\right\vert &\leq &C\left( \left\vert
a_{s}^{\left( 2\right) }\right\vert \left\vert \theta _{s}\right\vert
^{3}+3\left\vert a_{s}^{\left( 1\right) }\right\vert \left\vert \theta
_{s}\right\vert \left\vert \nabla \theta _{s}\right\vert +\left\vert
a_{s}\right\vert \left\vert \nabla ^{2}\theta _{s}\right\vert \right) \\
&\leq &f^{\left( 2\right) }\left( s\right) \left( \left\vert \theta
_{s}\right\vert ^{3}+\left\vert \theta _{s}\right\vert \left\vert \nabla
\theta _{s}\right\vert +\left\vert \nabla ^{2}\theta _{s}\right\vert \right)
,
\end{eqnarray*}%
for a continuous function $f^{\left( 2\right) }.$

Note that in these cases for $k=0,1,2,$ we can symbolically write 
\begin{equation}
\nabla ^{k}db_{s}=\sum_{\left( i_{1},...,i_{k+1}\right) \in
I}a_{i_{1}...i_{k+1}}\left( \left( \theta _{s}\right) ^{i_{1}},\left( \nabla
\theta _{s}\right) ^{i_{2}},...,\left( \nabla ^{k}\theta _{s}\right)
^{i_{k+1}}\right)  \label{nablakdbs}
\end{equation}%
where $a_{i_{1}...i_{k+1}}$ are multilinear maps that depend on $s$ and $%
I=\left\{ \left( i_{1},...,i_{k+1}\right) \in \mathbb{N}_{0}^{k+1},\ \text{%
such that }\sum_{m=1}^{k+1}mi_{m}=k+1\right\} .$ Proceeding by induction,
consider 
\begin{eqnarray}
\nabla ^{k+1}db_{s} &=&\nabla \sum_{\left( i_{1},...,i_{k+1}\right) \in
I}a_{i_{1}...i_{k+1}}\left( \left( \theta _{s}\right) ^{i_{1}},\left( \nabla
\theta _{s}\right) ^{i_{2}},...,\left( \nabla ^{k}\theta _{s}\right)
^{i_{k+1}}\right)  \label{nablakp1db} \\
&=&\sum_{\left( i_{1},...,i_{k+1}\right) \in I}a_{i_{1}...i_{k+1}}^{\left(
0\right) }\left( \left( \theta _{s}\right) ^{i_{1}+1},\left( \nabla \theta
_{s}\right) ^{i_{2}},...,\left( \nabla ^{k}\theta _{s}\right)
^{i_{k+1}}\right)  \notag \\
&&+\sum_{\left( i_{1},...,i_{k+1}\right) \in I}a_{i_{1}...i_{k+1}}^{\left(
1\right) }\left( \left( \theta _{s}\right) ^{i_{1}-1},\left( \nabla \theta
_{s}\right) ^{i_{2}+1},...,\left( \nabla ^{k}\theta _{s}\right)
^{i_{k+1}}\right)  \notag \\
&&+...+\sum_{\left( i_{1},...,i_{k+1}\right) \in
I}a_{i_{1}...i_{k+1}}^{\left( k+1\right) }\left( \left( \theta _{s}\right)
^{i_{1}},\left( \nabla \theta _{s}\right) ^{i_{2}},...,\left( \nabla
^{k-1}\theta _{s}\right) ^{i_{k}-1},\left( \nabla ^{k}\theta _{s}\right)
^{i_{k+1}+1}\right)  \notag \\
&&+\sum_{\left( i_{1},...,i_{k+1}\right) \in I}a_{i_{1}...i_{k+1}}^{\left(
k+2\right) }\left( \left( \theta _{s}\right) ^{i_{1}},\left( \nabla \theta
_{s}\right) ^{i_{2}},...,\left( \nabla ^{k}\theta _{s}\right)
^{i_{k+1}-1},\nabla ^{k+1}\theta _{s}\right)  \notag
\end{eqnarray}%
where $a_{i_{1}...i_{k+1}}^{\left( j\right) }$ are new multilinear forms.
The form $a_{i_{1}...i_{k+1}}^{\left( 0\right) }$ is obtained from $\nabla
a_{i_{1}...i_{k+1}},$ and adds another $\theta _{s}.$ Note that since $%
\sum_{m=1}^{k+1}mi_{m}=k+1,$ replacing $i_{1}$ with $i_{1}+1,$ increases
this sum by $1$.

The remaining terms in (\ref{nablakp1db}) are obtained from differentiating
the derivatives of $\theta _{s}.$ Symbolically,%
\begin{equation*}
\nabla \left( \left( \nabla ^{j}\theta \right) ^{i_{j+1}}\right) \sim \left(
\nabla ^{j}\theta \right) ^{i_{j+1}-1}\left( \nabla ^{j+1}\theta \right) ,
\end{equation*}%
so differentiation of each term decreases the power of $\nabla ^{j}\theta $
by one, and adds another $\nabla ^{j+1}\theta .$ Again, in the sum $%
\sum_{m=1}^{k+1}mi_{m}=k+1,$ replacing $ji_{j}+\left( j+1\right) i_{j+1}$
with $j\left( i_{j}-1\right) +\left( j+1\right) \left( i_{j+1}+1\right) $
increases the sum by $1.$

Overall we can then rewrite (\ref{nablakp1db}) in the form 
\begin{equation*}
\nabla ^{k+1}db_{s}=\sum_{\left( i_{1},...,i_{k+1},i_{k+2}\right) \in
I^{\prime }}a_{i_{1}...i_{k+2}}\left( \left( \theta _{s}\right)
^{i_{1}},\left( \nabla \theta _{s}\right) ^{i_{2}},...,\left( \nabla
^{k+1}\theta _{s}\right) ^{i_{k+2}}\right)
\end{equation*}%
where $I^{\prime }=\left\{ \left( i_{1},...,i_{k+2}\right) \in \mathbb{N}%
_{0}^{k+2},\ \text{such that }\sum_{m=1}^{k+2}mi_{m}=k+2\right\} ,$ hence
proving the inductive step. The estimate then follows immediately.
\end{proof}

\begin{lemma}
\label{lemThetaPWest}We have the following pointwise estimates 
\begin{equation*}
\left\vert \nabla ^{j}\theta _{A\left( t\right) s}\right\vert \leq
Ce^{C\left( j+1\right) t\left\vert \xi \right\vert }p_{j}\left( t\right) ,
\end{equation*}%
where 
\begin{equation}
p_{j}\left( t\right) =t\left\vert \nabla ^{j}d\xi \right\vert
+\sum_{J_{j}}t^{k_{1}+...+k_{j}}\left\vert \theta _{s}\right\vert
^{i_{1}}\left\vert \nabla \theta _{s}\right\vert ^{i_{2}}...\left\vert
\nabla ^{j}\theta _{s}\right\vert ^{i_{j}}\left\vert d\xi \right\vert
^{k_{1}}\left\vert \nabla d\xi \right\vert ^{k_{2}}...\left\vert \nabla
^{j-1}d\xi \right\vert ^{k_{j}},  \label{pjt}
\end{equation}%
with $J_{j}=\left\{ \left( i_{1},...,i_{j},k_{1},...,k_{j}\right) \in 
\mathbb{N}_{0}^{2j}:\sum_{m=1}^{j}mi_{m}+\sum_{m^{\prime }=1}^{j}m^{\prime
}k_{m^{\prime }}=j+1\right\} .$

In particular, for $k=1$ and $k=2,$ we have 
\begin{eqnarray}
\left\vert \nabla \theta _{A\left( t\right) s}\right\vert &\leq
&Ce^{2Ct\left\vert \xi \right\vert }\left( \left\vert \theta _{s}\right\vert
^{2}+t\left\vert d\xi \right\vert \left\vert \theta _{s}\right\vert
+t^{2}\left\vert d\xi \right\vert ^{2}+\left\vert \nabla \theta
_{s}\right\vert +t\left\vert \nabla d\xi \right\vert \right) \\
\left\vert \nabla ^{2}\theta _{A\left( t\right) s}\right\vert &\leq
&Ce^{3Ct\left\vert \xi \right\vert }\left( \left\vert \nabla ^{2}\theta
_{s}\right\vert +\left\vert \nabla \theta _{s}\right\vert \left\vert \theta
_{s}\right\vert +\left\vert \theta _{s}\right\vert ^{3}+t\left\vert \nabla
\theta _{s}\right\vert \left\vert d\xi \right\vert +t\left\vert \theta
_{s}\right\vert ^{2}\left\vert d\xi \right\vert +t\left\vert \theta
_{s}\right\vert \left\vert \nabla d\xi \right\vert \right. \\
&&\left. +t^{2}\left\vert \theta _{s}\right\vert \left\vert d\xi \right\vert
^{2}+t^{3}\left\vert d\xi \right\vert ^{3}+t\left\vert \nabla ^{2}d\xi
\right\vert \right) .  \notag
\end{eqnarray}
\end{lemma}

\begin{proof}
By differentiating (\ref{dthetadt}), we see that the $k$-th covariant
derivative of $\theta _{A\left( t\right) s}$ satisfies the following initial
value problem 
\begin{eqnarray*}
\frac{d\nabla ^{k}\theta _{A\left( t\right) s}}{dt} &=&%
\sum_{k_{1}+k_{2}+k_{3}=k}\left( \nabla ^{k_{1}}b_{A\left( t\right)
s}\right) \left( \nabla ^{k_{2}}\xi ,\nabla ^{k_{3}}\theta _{A\left(
t\right) s}\right) +\nabla ^{k}d\xi \\
\nabla ^{k}\theta _{A\left( 0\right) s} &=&\nabla ^{k}\theta _{s}
\end{eqnarray*}%
for $k_{1},k_{2},k_{3}\geq 0.$ In particular, 
\begin{equation*}
\frac{d\nabla ^{k}\theta _{A\left( t\right) s}}{dt}=\left[ \xi ,\nabla
^{k}\theta _{A\left( t\right) s}\right] ^{A\left( t\right) s}+\sum 
_{\substack{ k_{1}+k_{2}+k_{3}=k  \\ k_{3}<k}}\left( \nabla
^{k_{1}}b_{A\left( t\right) s}\right) \left( \nabla ^{k_{2}}\xi ,\nabla
^{k_{3}}\theta _{A\left( t\right) s}\right) +\nabla ^{k}d\xi ,
\end{equation*}%
and thus the solution of the ODE is%
\begin{eqnarray}
\nabla ^{k}\theta _{A\left( t\right) s} &=&U_{t\xi }^{\left( s\right)
}\nabla ^{k}\theta _{s}+\left( U_{t\xi }^{\left( s\right)
}\int_{0}^{t}\left( U_{\tau \xi }^{\left( s\right) }\right) ^{-1}d\tau
\right) \nabla ^{k}d\xi  \label{nabktheta} \\
&&+U_{t\xi }^{\left( s\right) }\int_{0}^{t}\left( U_{\tau \xi }^{\left(
s\right) }\right) ^{-1}\sum_{\substack{ k_{1}+k_{2}+k_{3}=k  \\ k_{3}<k}}%
\left( \nabla ^{k_{1}}b_{A\left( \tau \right) s}\right) \left( \nabla
^{k_{2}}\xi ,\nabla ^{k_{3}}\theta _{A\left( \tau \right) s}\right) d\tau . 
\notag
\end{eqnarray}%
To estimate $\left\vert \nabla ^{k}\theta _{A\left( t\right) s}\right\vert $%
, consider 
\begin{eqnarray*}
\left\vert U_{t\xi }^{\left( s\right) }\nabla ^{k}\theta _{A\left( t\right)
s}\right\vert &\leq &e^{Ct\left\vert \xi \right\vert }\left\vert \nabla
^{k}\theta _{s}\right\vert \\
\left\vert \left( U_{t\xi }^{\left( s\right) }\int_{0}^{t}\left( U_{\tau \xi
}^{\left( s\right) }\right) ^{-1}d\tau \right) \nabla ^{k}d\xi \right\vert
&\leq &te^{Ct\left\vert \xi \right\vert }\left\vert \nabla ^{k}d\xi
\right\vert
\end{eqnarray*}%
using (\ref{Uest}) and (\ref{intUest}), and moreover, 
\begin{equation*}
\left\vert \left( \nabla ^{k_{1}}b_{A\left( \tau \right) s}\right) \left(
\nabla ^{k_{2}}\xi ,\nabla ^{k_{3}}\theta _{A\left( \tau \right) s}\right)
\right\vert \leq \left\vert \nabla ^{k_{1}}b_{A\left( \tau \right)
s}\right\vert \left\vert \nabla ^{k_{2}}\xi \right\vert \left\vert \nabla
^{k_{3}}\theta _{A\left( \tau \right) s}\right\vert .
\end{equation*}%
From (\ref{nablakbest}), and using the fact that $\mathbb{L}$ is compact, we
know that 
\begin{equation}
\left\vert \nabla ^{k_{1}}b_{A\left( \tau \right) s}\right\vert \leq
C\sum_{\left( i_{1},...,i_{k_{1}}\right) \in I_{k_{1}}}\left\vert \theta
_{A\left( \tau \right) s}\right\vert ^{i_{1}}\left\vert \nabla \theta
_{A\left( \tau \right) s}\right\vert ^{i_{2}}...\left\vert \nabla
^{k_{1}-1}\theta _{A\left( \tau \right) s}\right\vert ^{i_{k_{1}}}
\end{equation}%
where $I_{k_{1}}=\left\{ \left( i_{1},...,i_{k_{1}}\right) \in \mathbb{N}%
^{k_{1}},\ \text{such that }\sum_{m=1}^{k_{1}}mi_{m}=k_{1}\right\} $ and $C$
is a constant that depends only on $\mathbb{L}.$

To proceed with an induction argument, we now need to complete the base
step. First, from (\ref{thetaAsest}), we know that 
\begin{equation*}
\left\vert \theta _{A\left( t\right) s}\right\vert \leq e^{Ct\left\vert \xi
\right\vert }\left( \left\vert \theta _{s}\right\vert +t\left\vert d\xi
\right\vert \right) .
\end{equation*}%
Thus, for $k=1$, 
\begin{eqnarray*}
\left\vert \nabla \theta _{A\left( t\right) s}\right\vert  &\leq
&e^{Ct\left\vert \xi \right\vert }\left( \left\vert \nabla \theta
_{s}\right\vert +t\left\vert \nabla d\xi \right\vert \right)  \\
&&+Ce^{Ct\left\vert \xi \right\vert }\int_{0}^{t}e^{-C\tau \left\vert \xi
\right\vert }\left( \left\vert b_{A\left( \tau \right) s}\right\vert
\left\vert d\xi \right\vert \left\vert \theta _{A\left( \tau \right)
s}\right\vert +\left\vert db_{A\left( \tau \right) s}\right\vert \left\vert
\xi \right\vert \left\vert \theta _{A\left( \tau \right) s}\right\vert
\right) d\tau  \\
&\leq &e^{Ct\left\vert \xi \right\vert }\left( \left\vert \nabla \theta
_{s}\right\vert +t\left\vert \nabla d\xi \right\vert \right)  \\
&&+Ce^{Ct\left\vert \xi \right\vert }\int_{0}^{t}e^{-C\tau \left\vert \xi
\right\vert }\left( \left\vert d\xi \right\vert \left\vert \theta _{A\left(
\tau \right) s}\right\vert +\left\vert \theta _{A\left( \tau \right)
s}\right\vert ^{2}\left\vert \xi \right\vert \right) d\tau 
\end{eqnarray*}%
where we used the estimate $\left\vert db_{A\left( \tau \right)
s}\right\vert \leq C\left\vert \theta _{A\left( \tau \right) s}\right\vert $
(\ref{dbsest}). Moreover, 
\begin{eqnarray*}
\int_{0}^{t}e^{-C\tau \left\vert \xi \right\vert }\left\vert \nabla \xi
\right\vert \left\vert \theta _{A\left( \tau \right) s}\right\vert d\tau 
&\leq &\int_{0}^{t}\left\vert \nabla \xi \right\vert \left( \left\vert
\theta _{s}\right\vert +\tau \left\vert d\xi \right\vert \right) d\tau  \\
&\leq &t\left\vert \nabla \xi \right\vert \left\vert \theta _{s}\right\vert
+t^{2}\left\vert d\xi \right\vert ^{2} \\
\int_{0}^{t}e^{-C\tau \left\vert \xi \right\vert }\left\vert \theta
_{A\left( \tau \right) s}\right\vert ^{2}\left\vert \xi \right\vert d\tau 
&\leq &\int_{0}^{t}e^{C\tau \left\vert \xi \right\vert }\left( \left(
\left\vert \theta _{s}\right\vert +\tau \left\vert d\xi \right\vert \right)
^{2}\left\vert \xi \right\vert \right) d\tau  \\
&\leq &\frac{1}{C}e^{Ct\left\vert \xi \right\vert }\left( \left\vert \theta
_{s}\right\vert +t\left\vert d\xi \right\vert \right) ^{2}.
\end{eqnarray*}%
Hence, for some new overall constant $C$, we have 
\begin{equation}
\left\vert \nabla \theta _{A\left( t\right) s}\right\vert \leq
Ce^{2Ct\left\vert \xi \right\vert }\left( \left\vert \theta _{s}\right\vert
^{2}+t\left\vert d\xi \right\vert \left\vert \theta _{s}\right\vert
+t^{2}\left\vert d\xi \right\vert ^{2}+t\left\vert \nabla d\xi \right\vert
+\left\vert \nabla \theta _{s}\right\vert \right) .  \label{nabthetaest}
\end{equation}%
Let us also complete the $k=2$ case. 
\begin{eqnarray*}
\left\vert \nabla ^{2}\theta _{A\left( t\right) s}\right\vert  &\leq
&e^{Ct\left\vert \xi \right\vert }\left( \left\vert \nabla ^{2}\theta
_{s}\right\vert +t\left\vert \nabla ^{2}d\xi \right\vert \right)  \\
&&+Ce^{Ct\left\vert \xi \right\vert }\int_{0}^{t}e^{-C\tau \left\vert \xi
\right\vert }\left( \sum_{k_{1}+k_{2}+k_{3}=2}\left\vert \nabla
^{k_{1}}b_{A\left( \tau \right) s}\right\vert \left\vert \nabla ^{k_{2}}\xi
\right\vert \left\vert \nabla ^{k_{3}}\theta _{A\left( \tau \right)
s}\right\vert \right) d\tau .
\end{eqnarray*}%
More explicitly, 
\begin{eqnarray*}
\sum_{\substack{ k_{1}+k_{2}+k_{3}=2 \\ k_{3}<2}}\left\vert \nabla
^{k_{1}}b_{A\left( \tau \right) s}\right\vert \left\vert \nabla ^{k_{2}}\xi
\right\vert \left\vert \nabla ^{k_{3}}\theta _{A\left( \tau \right)
s}\right\vert  &=&\left\vert \nabla db_{A\left( \tau \right) s}\right\vert
\left\vert \xi \right\vert \left\vert \theta _{A\left( \tau \right)
s}\right\vert  \\
&&+\left\vert db_{A\left( \tau \right) s}\right\vert \left\vert d\xi
\right\vert \left\vert \theta _{A\left( \tau \right) s}\right\vert  \\
&&+\left\vert db_{A\left( \tau \right) s}\right\vert \left\vert \xi
\right\vert \left\vert \nabla \theta _{A\left( \tau \right) s}\right\vert  \\
&&+\left\vert b_{A\left( \tau \right) s}\right\vert \left\vert \nabla d\xi
\right\vert \left\vert \theta _{A\left( \tau \right) s}\right\vert .
\end{eqnarray*}%
Let 
\begin{eqnarray*}
p_{0}\left( \tau \right)  &=&\left\vert \theta _{s}\right\vert +\tau
\left\vert d\xi \right\vert  \\
p_{1}\left( \tau \right)  &=&\left\vert \theta _{s}\right\vert ^{2}+\tau
\left\vert d\xi \right\vert \left\vert \theta _{s}\right\vert +\tau
^{2}\left\vert d\xi \right\vert ^{2}+\left\vert \nabla \theta
_{s}\right\vert +\tau \left\vert \nabla d\xi \right\vert 
\end{eqnarray*}%
so that 
\begin{eqnarray*}
\left\vert \theta _{A\left( \tau \right) s}\right\vert  &\leq &e^{C\tau
\left\vert \xi \right\vert }p_{0}\left( \tau \right)  \\
\left\vert \nabla \theta _{A\left( t\right) s}\right\vert  &\leq &Ce^{2C\tau
\left\vert \xi \right\vert }p_{1}\left( \tau \right) .
\end{eqnarray*}%
Using also (\ref{dbsest}) and (\ref{d2best}) we have 
\begin{eqnarray*}
\left\vert \nabla db_{A\left( \tau \right) s}\right\vert \left\vert \xi
\right\vert \left\vert \theta _{A\left( \tau \right) s}\right\vert  &\leq
&Ce^{3C\left\vert \xi \right\vert \tau }\left\vert \xi \right\vert
p_{0}\left( p_{0}^{2}+p_{1}\right)  \\
\left\vert db_{A\left( \tau \right) s}\right\vert \left\vert d\xi
\right\vert \left\vert \theta _{A\left( \tau \right) s}\right\vert  &\leq
&Ce^{2C\left\vert \xi \right\vert \tau }p_{0}^{2}\left\vert d\xi \right\vert 
\\
\left\vert db_{A\left( \tau \right) s}\right\vert \left\vert \xi \right\vert
\left\vert \nabla \theta _{A\left( \tau \right) s}\right\vert  &\leq
&Ce^{2C\left\vert \xi \right\vert \tau }\left\vert \xi \right\vert p_{0}p_{1}
\\
\left\vert b_{A\left( \tau \right) s}\right\vert \left\vert \nabla d\xi
\right\vert \left\vert \theta _{A\left( \tau \right) s}\right\vert  &\leq
&Ce^{C\left\vert \xi \right\vert \tau }\left\vert \nabla d\xi \right\vert
p_{0}.
\end{eqnarray*}%
Since $p_{0}$ and $p_{1}$ are non-decreasing functions of $\tau $, we can
evaluate them at $\tau =t$. In particular, 
\begin{eqnarray*}
e^{Ct\left\vert \xi \right\vert }\int_{0}^{t}e^{-C\tau \left\vert \xi
\right\vert }\left\vert \nabla db_{A\left( \tau \right) s}\right\vert
\left\vert \xi \right\vert \left\vert \theta _{A\left( \tau \right)
s}\right\vert d\tau  &\leq &C\left\vert \xi \right\vert p_{0}\left(
p_{0}^{2}+p_{1}\right) e^{Ct\left\vert \xi \right\vert
}\int_{0}^{t}e^{2C\tau \left\vert \xi \right\vert }d\tau  \\
&\leq &Ce^{3C\tau \left\vert \xi \right\vert }p_{0}\left(
p_{0}^{2}+p_{1}\right)  \\
e^{Ct\left\vert \xi \right\vert }\int_{0}^{t}e^{-C\tau \left\vert \xi
\right\vert }\left\vert db_{A\left( \tau \right) s}\right\vert \left\vert
d\xi \right\vert \left\vert \theta _{A\left( \tau \right) s}\right\vert
d\tau  &\leq &Cte^{2\left\vert \xi \right\vert t}p_{0}^{2}\left\vert d\xi
\right\vert \leq Cte^{3\left\vert \xi \right\vert t}p_{0}^{2}\left\vert d\xi
\right\vert ,
\end{eqnarray*}%
and similarly for other terms. Overall, we find 
\begin{eqnarray*}
\left\vert \nabla ^{2}\theta _{A\left( t\right) s}\right\vert  &\leq
&Ce^{3Ct\left\vert \xi \right\vert }\left( p_{0}\left(
p_{0}^{2}+p_{1}\right) +tp_{0}^{2}\left\vert d\xi \right\vert
+p_{0}p_{1}+t\left\vert \nabla d\xi \right\vert p_{0}+\left\vert \nabla
^{2}\theta _{s}\right\vert +t\left\vert \nabla ^{2}d\xi \right\vert \right) 
\\
&\leq &Ce^{3Ct\left\vert \xi \right\vert }\left( \left\vert \nabla
^{2}\theta _{s}\right\vert +\left\vert \nabla \theta _{s}\right\vert
\left\vert \theta _{s}\right\vert +\left\vert \theta _{s}\right\vert
^{3}+t\left\vert \nabla \theta _{s}\right\vert \left\vert d\xi \right\vert
+t\left\vert \theta _{s}\right\vert ^{2}\left\vert d\xi \right\vert
+t\left\vert \theta _{s}\right\vert \left\vert \nabla d\xi \right\vert
\right.  \\
&&\left. +t\left\vert \nabla ^{2}d\xi \right\vert +t^{2}\left\vert \theta
_{s}\right\vert \left\vert d\xi \right\vert ^{2}+t^{3}\left\vert d\xi
\right\vert ^{3}\right) .
\end{eqnarray*}%
Suppose for each $j<k,$ 
\begin{equation*}
\left\vert \nabla ^{j}\theta _{A\left( t\right) s}\right\vert \leq
Ce^{\left( j+1\right) Ct\left\vert \xi \right\vert }p_{j}\left( t\right) ,
\end{equation*}%
where 
\begin{equation*}
p_{j}\left( t\right) =t\left\vert \nabla ^{j}d\xi \right\vert
+\sum_{J_{j}}t^{k_{1}+...+k_{j}}\left\vert \theta _{s}\right\vert
^{i_{1}}\left\vert \nabla \theta _{s}\right\vert ^{i_{2}}...\left\vert
\nabla ^{j}\theta _{s}\right\vert ^{i_{j+1}}\left\vert d\xi \right\vert
^{k_{1}}\left\vert \nabla d\xi \right\vert ^{k_{2}}...\left\vert \nabla
^{j-1}d\xi \right\vert ^{k_{j}},
\end{equation*}%
where $J_{j}=\left\{ \left( i_{1},...,i_{j+1},k_{1},...,k_{j+1}\right) \in 
\mathbb{N}^{2j+2}:\sum_{m=1}^{j+1}mi_{m}+\sum_{m=1}^{j}mk_{m}=j+1\right\} .$

Therefore, 
\begin{eqnarray*}
\left\vert \sum_{\substack{ k_{1}+k_{2}+k_{3}=k \\ k_{3}<k}}\left( \nabla
^{k_{1}}b_{A\left( \tau \right) s}\right) \left( \nabla ^{k_{2}}\xi ,\nabla
^{k_{3}}\theta _{A\left( \tau \right) s}\right) \right\vert  &\leq
&C\sum_{\left( i_{0},i_{1},...,i_{k}\right) \in I_{k}^{\prime }}\left\vert
\nabla ^{i_{0}}\xi \right\vert \left\vert \theta _{A\left( \tau \right)
s}\right\vert ^{i_{1}}\left\vert \nabla \theta _{A\left( \tau \right)
s}\right\vert ^{i_{2}}...\left\vert \nabla ^{k-1}\theta _{A\left( \tau
\right) s}\right\vert ^{i_{k}} \\
&\leq &C\sum_{\left( i_{0},i_{1},...,i_{k}\right) \in I_{k}^{\prime
}}e^{C\left( k+1-i_{0}\right) \tau \left\vert \xi \right\vert }\left\vert
\nabla ^{i_{0}}\xi \right\vert
p_{0}^{i_{1}}p_{1}^{i_{2}}...p_{k_{1}-1}^{i_{k_{1}}} \\
&\leq &C\sum_{\substack{ \left( i_{0},i_{1},...,i_{k}\right) \in
I_{k}^{\prime } \\ i_{0}>0}}e^{C\left( k+1-i_{0}\right) \tau \left\vert \xi
\right\vert }\left\vert \nabla ^{i_{0}}\xi \right\vert
p_{0}^{i_{1}}p_{1}^{i_{2}}...p_{k-1}^{i_{k}} \\
&&+C\sum_{\left( 0,i_{1},...,i_{k}\right) \in I_{k}^{\prime }}e^{C\left(
k+1\right) \tau \left\vert \xi \right\vert }\left\vert \xi \right\vert
p_{0}^{i_{1}}p_{1}^{i_{2}}...p_{k_{1}-1}^{i_{k_{1}}},
\end{eqnarray*}%
where $I_{k}^{\prime }=\left\{ \left( i_{0},i_{1},...,i_{k}\right) \in 
\mathbb{N}_{0}^{k+1},\ \text{such that }\sum_{m=1}^{k}mi_{m}+i_{0}=k+1\right%
\} .$

Now, from (\ref{nabktheta}), we find 
\begin{eqnarray*}
\left\vert \nabla ^{k}\theta _{A\left( t\right) s}\right\vert &\leq
&Ce^{Ct\left\vert \xi \right\vert }\left\vert \nabla ^{k}\theta
_{s}\right\vert +te^{Ct\left\vert \xi \right\vert }\left\vert \nabla
^{k}d\xi \right\vert \\
&&+Ce^{Ct\left\vert \xi \right\vert }\int_{0}^{t}\sum_{\substack{ \left(
i_{1},...,i_{k_{1}}\right) \in I_{k_{1}}  \\ k_{1}+k_{2}=k,\ k_{2}>0}}%
e^{Ck_{1}\tau \left\vert \xi \right\vert
}p_{0}^{i_{1}}p_{1}^{i_{2}}...p_{k_{1}-1}^{i_{k_{1}}}\left\vert \nabla
^{k_{2}}\xi \right\vert d\tau \\
&&+Ce^{Ct\left\vert \xi \right\vert }\int_{0}^{t}\sum_{\left(
i_{1},...,i_{k}\right) \in I_{k}}e^{Ck\tau \left\vert \xi \right\vert
}p_{0}^{i_{1}}p_{1}^{i_{2}}...p_{k-1}^{i_{k}}\left\vert \xi \right\vert d\tau
\\
&\leq &Ce^{Ct\left\vert \xi \right\vert }\left\vert \nabla ^{k}\theta
_{s}\right\vert +te^{Ct\left\vert \xi \right\vert }\left\vert \nabla
^{k}d\xi \right\vert \\
&&+C\sum_{\substack{ \left( i_{1},...,i_{k_{1}}\right) \in I_{k_{1}}  \\ %
k_{1}+k_{2}=k,\ k_{2}>0}}te^{C\left( k_{1}+1\right) t\left\vert \xi
\right\vert }p_{0}^{i_{1}}p_{1}^{i_{2}}...p_{k_{1}-1}^{i_{k_{1}}}\left\vert
\nabla ^{k_{2}}\xi \right\vert \\
&&+C\sum_{\left( i_{1},...,i_{k}\right) \in I_{k}}e^{C\left( k+1\right)
t\left\vert \xi \right\vert }p_{0}^{i_{1}}p_{1}^{i_{2}}...p_{k-1}^{i_{k}},
\end{eqnarray*}%
where we have bounded $p_{i}\left( \tau \right) \leq p_{i}\left( t\right) ,$
since these functions are non-decreasing. Also, in the first integral, we
bounded $\int_{0}^{t}e^{Ck_{1}\tau \left\vert \xi \right\vert }d\tau \leq
te^{Ck_{1}t\left\vert \xi \right\vert }$ and in the second integral, we used 
$\int_{0}^{t}e^{Ck\tau \left\vert \xi \right\vert }\left\vert \xi
\right\vert d\tau \leq C^{\prime }e^{Ck\left\vert \xi \right\vert }$ for
some new constant $C^{\prime }.$ Further, we can bound%
\begin{eqnarray*}
\left\vert \nabla ^{k}\theta _{A\left( t\right) s}\right\vert &\leq
&Ce^{C\left( k+1\right) t\left\vert \xi \right\vert }\left( \left\vert
\nabla ^{k}\theta _{s}\right\vert +\left\vert \nabla ^{k}d\xi \right\vert
\right. \\
&&+\sum_{\substack{ \left( i_{1},...,i_{k_{1}}\right) \in I_{k_{1}}  \\ %
k_{1}+k_{2}=k,\ k_{2}>0}}%
tp_{0}^{i_{1}}p_{1}^{i_{2}}...p_{k_{1}-1}^{i_{k_{1}}}\left\vert \nabla
^{k_{2}}\xi \right\vert \\
&&\left. +\sum_{\left( i_{1},...,i_{k}\right) \in
I_{k}}p_{0}^{i_{1}}p_{1}^{i_{2}}...p_{k-1}^{i_{k}}\right) \\
&\leq &Ce^{C\left( k+1\right) t\left\vert \xi \right\vert }p_{k}\left(
t\right) .
\end{eqnarray*}
\end{proof}

\begin{corollary}
For $k>0,\ b_{A\left( t\right) s}\ \ $satisfies 
\begin{equation*}
\left\vert \nabla ^{k}b_{A\left( t\right) s}\right\vert \lesssim
e^{Ckt\left\vert \xi \right\vert }p_{k-1}\left( t\right) ,
\end{equation*}%
where $p_{k}\left( t\right) $ is given by (\ref{pjt}).
\end{corollary}

\begin{proof}
From (\ref{nablakbest}), 
\begin{equation}
\left\vert \nabla ^{k}b_{As}\right\vert \lesssim \sum_{\left(
i_{1},...,i_{k}\right) \in I_{k}}\left\vert \theta _{As}\right\vert
^{i_{1}}\left\vert \nabla \theta _{As}\right\vert ^{i_{2}}...\left\vert
\nabla ^{k-1}\theta _{As}\right\vert ^{i_{k}}
\end{equation}%
where $I_{k}=\left\{ \left( i_{1},...,i_{k}\right) \in \mathbb{N}_{0}^{k},\ 
\text{such that }\sum_{m=1}^{k}mi_{m}=k\right\} .$ However, from Lemma \ref%
{lemThetaPWest}, 
\begin{equation*}
\left\vert \nabla ^{j}\theta _{A\left( t\right) s}\right\vert \lesssim
e^{C\left( j+1\right) t\left\vert \xi \right\vert }p_{j}\left( t\right) ,
\end{equation*}%
where 
\begin{equation*}
p_{j}\left( t\right) =t\left\vert \nabla ^{j}d\xi \right\vert
+\sum_{J_{j}}t^{k_{1}+...+k_{j}}\left\vert \theta _{s}\right\vert
^{i_{1}}\left\vert \nabla \theta _{s}\right\vert ^{i_{2}}...\left\vert
\nabla ^{j}\theta _{s}\right\vert ^{i_{j+1}}\left\vert d\xi \right\vert
^{k_{1}}\left\vert \nabla d\xi \right\vert ^{k_{2}}...\left\vert \nabla
^{j-1}d\xi \right\vert ^{k_{j}},
\end{equation*}%
with $J_{j}=\left\{ \left( i_{1},...,i_{j+1},k_{1},...,k_{j}\right) \in 
\mathbb{N}_{0}^{2j+1}:\sum_{m=1}^{j+1}mi_{m}+\sum_{m^{\prime
}=1}^{j}m^{\prime }k_{m^{\prime }}=j+1\right\} .\ $So 
\begin{eqnarray*}
\left\vert \nabla ^{k}b_{A\left( t\right) s}\right\vert &\lesssim
&e^{Ckt\left\vert \xi \right\vert }\sum_{\left( i_{1},...,i_{k}\right) \in
I_{k}}p_{0}^{i_{1}}p_{1}^{i_{2}}...p_{k-1}^{i_{k}} \\
&\lesssim &e^{Ckt\left\vert \xi \right\vert }p_{k-1}
\end{eqnarray*}
\end{proof}

More generally, suppose we have $1$-parameter family $X\left( t\right) $ of $%
\mathfrak{l}$-valued maps that satisfies 
\begin{equation}
\left\{ 
\begin{array}{c}
\frac{dX\left( t\right) }{dt}=\left[ \xi ,X\left( t\right) \right] ^{A\left(
t\right) s}+Y \\ 
X\left( 0\right) =X_{0},%
\end{array}%
\right. ,  \label{Xivp}
\end{equation}%
where $Y$ is also an $\mathfrak{l}$-valued map. We know that 
\begin{equation}
X\left( t\right) =U_{t\xi }^{\left( s\right) }X_{0}+\left( U_{t\xi }^{\left(
s\right) }\int_{0}^{t}\left( U_{\tau \xi }^{\left( s\right) }\right)
^{-1}d\tau \right) Y,
\end{equation}%
and in particular, 
\begin{equation}
\left\vert X\left( t\right) \right\vert \lesssim e^{tC\left\vert \xi
\right\vert }\left( \left\vert X_{0}\right\vert +t\left\vert Y\right\vert
\right) .  \label{Xtest}
\end{equation}%
Differentiating (\ref{Xivp}), we obtain estimates for higher derivatives of $%
X.$

\begin{lemma}
\label{lemNablaX}Suppose $X\left( t\right) $ is a $1$-parameter family $%
X\left( t\right) $ of $\mathfrak{l}$-valued maps that satisfies (\ref{Xivp}%
). Then, 
\begin{equation*}
\left\vert \nabla ^{k}X\left( t\right) \right\vert \lesssim e^{\left(
k+1\right) Ct\left\vert \xi \right\vert }\sum_{k^{\prime }+k^{\prime \prime
}=k}p_{k^{\prime }-1}\left( t\right) \left( \left\vert \nabla ^{k^{\prime
\prime }}X_{0}\right\vert +t\left\vert \nabla ^{k^{\prime \prime
}}Y\right\vert \right) ,
\end{equation*}%
with $p_{-1}\left( t\right) =1$ and for $k\geq 0$, $p_{k}\left( t\right) $
is given by (\ref{pjt}).
\end{lemma}

\begin{proof}
Differentiating (\ref{Xivp}), for $k\geq 1$, we get 
\begin{equation*}
\frac{d\nabla ^{k}X\left( t\right) }{dt}=\left[ \xi ,\nabla ^{k}X\left(
t\right) \right] ^{A\left( t\right) s}+\sum_{\substack{ k_{1}+k_{2}+k_{3}=k 
\\ k_{3}<k}}\left( \nabla ^{k_{1}}b_{A\left( t\right) s}\right) \left(
\nabla ^{k_{2}}\xi ,\nabla ^{k_{3}}X\left( t\right) \right) +\nabla ^{k}Y,
\end{equation*}%
and hence, 
\begin{eqnarray*}
\nabla ^{k}X\left( t\right) &=&U_{t\xi }^{\left( s\right) }\left( \nabla
^{k}X_{0}\right) +\left( U_{t\xi }^{\left( s\right) }\int_{0}^{t}\left(
U_{\tau \xi }^{\left( s\right) }\right) ^{-1}d\tau \right) \nabla ^{k}Y \\
&&+\sum_{\substack{ k_{1}+k_{2}+k_{3}=k  \\ k_{3}<k}}\left( U_{t\xi
}^{\left( s\right) }\int_{0}^{t}\left( U_{\tau \xi }^{\left( s\right)
}\right) ^{-1}\left( \nabla ^{k_{1}}b_{A\left( \tau \right) s}\right) \left(
\nabla ^{k_{2}}\xi ,\nabla ^{k_{3}}X\left( \tau \right) \right) d\tau \right)
\end{eqnarray*}%
Let $q_{0}\left( t\right) =\left\vert X_{0}\right\vert +t\left\vert
Y\right\vert $, so that 
\begin{equation*}
\left\vert X\left( t\right) \right\vert \lesssim e^{tC\left\vert \xi
\right\vert }q_{0}\left( t\right) .
\end{equation*}%
Suppose for all $j<k$, 
\begin{equation*}
\left\vert \nabla ^{j}X\left( t\right) \right\vert \lesssim e^{\left(
j+1\right) Ct\left\vert \xi \right\vert }q_{j}\left( t\right) ,
\end{equation*}%
where $q_{j}\left( t\right) $ is non-decreasing. Then, 
\begin{eqnarray*}
\left\vert U_{t\xi }^{\left( s\right) }\int_{0}^{t}\left( U_{\tau \xi
}^{\left( s\right) }\right) ^{-1}\left( \nabla ^{k_{1}}b_{A\left( \tau
\right) s}\right) \left( \nabla ^{k_{2}}\xi ,\nabla ^{k_{3}}X\left( \tau
\right) \right) d\tau \right\vert &\lesssim &e^{Ct\left\vert \xi \right\vert
}\int_{0}^{t}e^{k_{3}C\tau \left\vert \xi \right\vert }\left\vert \nabla
^{k_{1}}b_{A\left( \tau \right) s}\right\vert \left\vert \nabla ^{k_{2}}\xi
\right\vert q_{k_{3}}\left( \tau \right) d\tau \\
&\lesssim &e^{Ct\left\vert \xi \right\vert }\int_{0}^{t}e^{\left(
k_{1}+k_{3}\right) C\tau \left\vert \xi \right\vert }p_{k_{1}-1}\left\vert
\nabla ^{k_{2}}\xi \right\vert q_{k_{3}}\left( \tau \right) d\tau
\end{eqnarray*}%
For $k_{2}=0,$ 
\begin{equation*}
e^{Ct\left\vert \xi \right\vert }\int_{0}^{t}e^{k_{3}C\tau \left\vert \xi
\right\vert }\left\vert \nabla ^{k_{1}}b_{A\left( \tau \right) s}\right\vert
\left\vert \xi \right\vert q_{k_{3}}\left( \tau \right) d\tau \lesssim
e^{\left( k+1\right) Ct\left\vert \xi \right\vert
}p_{k_{1}-1}q_{k_{3}}\left( t\right)
\end{equation*}%
For $k_{2}>0$, 
\begin{eqnarray*}
e^{Ct\left\vert \xi \right\vert }\int_{0}^{t}e^{k_{3}C\tau \left\vert \xi
\right\vert }\left\vert \nabla ^{k_{1}}b_{A\left( \tau \right) s}\right\vert
\left\vert \nabla ^{k_{2}}\xi \right\vert q_{k_{3}}\left( \tau \right) d\tau
&\lesssim &e^{\left( k_{1}+k_{3}+1\right) Ct\left\vert \xi \right\vert
}p_{k_{1}-1}\left\vert \nabla ^{k_{2}}\xi \right\vert q_{k_{3}}\left(
t\right) \\
&\lesssim &e^{\left( k+1\right) Ct\left\vert \xi \right\vert
}p_{k_{1}+k_{2}-1}q_{k_{3}}\left( t\right) .
\end{eqnarray*}%
Thus, 
\begin{equation*}
\left\vert \nabla ^{k}X\left( t\right) \right\vert \lesssim e^{\left(
k+1\right) Ct\left\vert \xi \right\vert }\left( \left\vert \nabla
^{k}X_{0}\right\vert +t\left\vert \nabla ^{k}Y\right\vert +\sum_{k^{\prime
}+k^{\prime \prime }=k-1}p_{k^{\prime }}q_{k^{\prime \prime }}\right)
\end{equation*}%
Therefore, for $k\geq 1$, 
\begin{equation*}
q_{k}=\left\vert \nabla ^{k}X_{0}\right\vert +t\left\vert \nabla
^{k}Y\right\vert +\sum_{k^{\prime }+k^{\prime \prime }=k-1}p_{k^{\prime
}}q_{k^{\prime \prime }}.
\end{equation*}%
Setting $x_{k}=\left\vert \nabla ^{k}X_{0}\right\vert +t\left\vert \nabla
^{k}Y\right\vert ,$ it is then easy to see that 
\begin{equation*}
q_{k}=\sum_{L_{k}}p_{0}^{j_{0}}p_{1}^{j_{1}}...p_{k-1}^{j_{k-1}}x_{l}
\end{equation*}%
where 
\begin{equation*}
L_{k}=\left\{ \left( j_{0},...,j_{k-1},x_{l}\right) \in \mathbb{N}%
_{0}^{k+2}:l+\sum_{m=0}^{j-1}\left( m+1\right) j_{m}=k\right\} ,
\end{equation*}%
and thus 
\begin{equation*}
q_{k}\lesssim x_{k}+\sum_{k^{\prime }+k^{\prime \prime }=k-1}p_{k^{\prime
}}x_{k^{\prime \prime }}.
\end{equation*}%
Therefore, 
\begin{equation*}
\left\vert \nabla ^{k}X\left( t\right) \right\vert \lesssim e^{\left(
k+1\right) Ct\left\vert \xi \right\vert }\sum_{k^{\prime }+k^{\prime \prime
}=k}p_{k^{\prime }-1}\left( \left\vert \nabla ^{k^{\prime \prime
}}X_{0}\right\vert +t\left\vert \nabla ^{k^{\prime \prime }}Y\right\vert
\right) ,
\end{equation*}%
with $p_{-1}=1.$
\end{proof}

We will need to be able to define loop-valued maps with Sobolev regularity.
First, let us recall Sobolev spaces $W^{k,p}$ of functions between manifolds.

\begin{lemma}
\label{lemSobManifolds}Suppose $M$ is a compact $n$-dimensional manifold and
suppose $N$ is an $l$-dimensional manifold. Let $k$ be a non-negative
integer and $r\geq 0$ such that $kr>n.$ Let $\Phi :N\longrightarrow \mathbb{R%
}^{2l}$ be a smooth embedding (by Whitney Embedding Theorem) and suppose $%
\left\{ \left( U_{\alpha },\phi _{\alpha }\right) \right\} $ is an atlas for 
$N.$ Suppose $f:M\longrightarrow N$ is a continuous map. Then, the following
are equivalent:

\begin{enumerate}
\item $f\in W^{k,r}\left( M,N\right) $

\item $\Phi \circ f\in W^{k,r}\left( M,\mathbb{R}^{2l}\right) $

\item $\phi _{\alpha }\circ f\in W^{k,r}\left( \phi _{\alpha }^{-1}\left(
U_{\alpha }\right) ,\mathbb{R}^{l}\right) $ for any chart $\left( U_{\alpha
},\phi _{\alpha }\right) $

\item $f^{\ast }\theta _{G}$ $\in W^{k-1,r}\left( M,T^{\ast }M\otimes 
\mathfrak{g}\right) $ in case when if $N=G,$ a compact Lie group, with Lie
algebra $\mathfrak{g}\ $and $\theta _{G}$ is the Maurer-Cartan form on $G.$
\end{enumerate}

In particular, conditions (2) and (3) are independent of the choice of the
embedding $\Phi $ and the atlas $\left( U_{\alpha },\phi _{\alpha }\right) ,$
respectively.
\end{lemma}

Note that the condition $kr>n$ is needed in Lemma \ref{lemSobManifolds} due
to the Sobolev embedding $W^{k,r}\subset C^{0}$ for $kr>n.$ We will prove a
characterization of loop-valued $W^{k,p}$-maps in terms of the loop
Maurer-Cartan form that is similar to item (4) in Lemma \ref{lemSobManifolds}%
.

\begin{lemma}
\label{lemSobLoops}Suppose $M$ is a compact $n$-dimensional manifold and
suppose $\mathbb{L}$ is a smooth loop of dimension $l,$ with tangent algebra 
$\mathfrak{l}$ and $\mathfrak{l}$-valued Maurer-Cartan form $\theta $. Let $%
k $ be a non-negative integer and $r\geq 0$ such that $kr>n.$ Suppose $%
s:M\longrightarrow \mathbb{L}$ is a continuous map. Then, $s\in
W^{k,r}\left( M,\mathbb{L}\right) $ if and only if $\theta _{s}\in
W^{k-1,r}\left( M,T^{\ast }M\otimes \mathfrak{l}\right) .$
\end{lemma}

\begin{proof}
Suppose $s\in W^{k,r}\left( M,\mathbb{L}\right) .$ By Lemma \ref%
{lemSobManifolds}, if $\left\{ \left( U_{\alpha },\phi _{\alpha }\right)
\right\} $ is an atlas for $\mathbb{L},$ then for each chart $\left(
U_{\alpha },\phi _{\alpha }\right) ,$ $\phi _{\alpha }\circ s\in
W^{k,r}\left( \phi _{\alpha }^{-1}\left( U_{\alpha }\right) ,\mathbb{R}%
^{l}\right) .$ Now, $\left\{ s^{-1}\left( U_{\alpha }\right) \right\} $ is
an open cover of $M,$ but using compactness of $M,$ let $\left\{
s^{-1}\left( U_{i}\right) \right\} $ be an finite subcover, and suppose $%
\left\{ u_{i}\right\} $ is a smooth partition of unity subordinate to this
subcover. Then, we can write 
\begin{eqnarray*}
\theta _{s} &=&\theta s_{\ast }=\sum_{i}\left( u_{i}\theta \right) s_{\ast }
\\
&=&\sum_{i}\left( \left( u_{i}\theta \right) \left( \phi _{i}^{-1}\right)
_{\ast }\right) \left( \left( \phi _{i}\right) _{\ast }s_{\ast }\right) .
\end{eqnarray*}%
For each $i$, $\left( \phi _{i}\right) _{\ast }s_{\ast }=\left( \phi
_{i}\circ s\right) _{\ast }\in W^{k-1,r}\left( \phi _{i}^{-1}\left(
U_{i}\right) ,\mathbb{R}^{l}\right) .$ On the other hand $\ \left(
u_{i}\theta \right) \left( \phi _{i}^{-1}\right) _{\ast }$ is a smooth
function, and hence composition with it is a continuous map $W^{k-1,r}\left(
\phi _{i}^{-1}\left( U_{i}\right) ,\mathbb{R}^{l}\right) \longrightarrow
W^{k-1,r}\left( \phi _{i}^{-1}\left( U_{i}\right) ,T^{\ast }U_{i}\otimes 
\mathfrak{l}\right) $ (using \cite[Lemma B.8]{WehrheimBook}). Overall, we
see that each term of this finite sum is bounded in the $W^{k-1,r}$ norm,
and thus $\theta _{s}\in W^{k-1,r}\left( M,T^{\ast }M\otimes \mathfrak{g}%
\right) .$

Conversely, suppose now $\theta _{s}\in W^{k-1,r}\left( M,T^{\ast }M\otimes 
\mathfrak{g}\right) .$ We will use item (2) in Lemma \ref{lemSobManifolds}
to show that $s\in W^{k,r}\left( M,\mathbb{L}\right) .$ This adapts the
proof of \cite[Lemma B.5]{WehrheimBook}. Let $\Phi :\mathbb{L}%
\longrightarrow \mathbb{R}^{2l}$ be a smooth embedding, so that $\Phi \circ
s $ is continuous. In particular, $\Phi \circ s\in L^{r}\left( M,\mathbb{R}%
^{2l}\right) .$ Now, let $x\in M$ and consider 
\begin{eqnarray*}
\left. d\left( \Phi \circ s\right) \right\vert _{x} &=&\left. d\Phi
\right\vert _{s\left( x\right) }\left. ds\right\vert _{x} \\
&=&\left( \left. d\Phi \right\vert _{s\left( x\right) }\rho _{s\left(
x\right) }\right) \left( \rho _{s\left( x\right) }^{-1}\left. ds\right\vert
_{x}\right) \\
&=&E\left( s\left( x\right) \right) \left( \theta _{s}\right) _{x},
\end{eqnarray*}%
where, for each $p\in \mathbb{L},$ we have the linear map $E\left( p\right)
=\left. d\Phi \right\vert _{p}\rho _{p}:\mathfrak{l}\longrightarrow \mathbb{R%
}^{2l}$, and the map $p\mapsto $ $E\left( p\right) $ is a smooth map from $%
\mathbb{L}$ to $\mathop{\rm Hom}\nolimits\left( \mathfrak{l},\mathbb{R}%
^{2l}\right) .$ Thus, we can write 
\begin{equation*}
d\left( \Phi \circ s\right) =\left( E\circ s\right) \theta _{s},
\end{equation*}%
with $E\circ s$ being bounded in the operator norm, since $s$ is continuous.
Hence, there exists a constant $C>0,$ such that 
\begin{equation*}
\left\Vert d\left( \Phi \circ s\right) \right\Vert _{L^{r}}\leq C\left\Vert
\theta _{s}\right\Vert _{L^{r}}\leq C\left\Vert \theta _{s}\right\Vert
_{W^{k-1,r}}.
\end{equation*}%
This shows that $\Phi \circ s\in W^{1,r}\left( M,\mathbb{R}^{2l}\right) .$
To show further that $\Phi \circ s\in W^{k,r}\left( M,\mathbb{R}^{2l}\right)
,$ $k\geq 2,$ similar estimates are obtained by considering higher
derivatives.
\end{proof}

\begin{theorem}
Let $M$ be a compact Riemannian manifold. Suppose $kr>n=\dim M.$ Let $s\in
W^{k,r}\left( M,\mathbb{L}\right) $ and $\xi \in W^{k,r}\left( M,\mathfrak{l}%
\right) ,$ and suppose $A=\exp _{s}\left( \xi \right) .$ Then, 
\begin{equation}
\left\Vert \theta _{As}\right\Vert _{W^{k-1,r}}\lesssim e^{Ck\left\Vert \xi
\right\Vert _{C^{0}}}\left( \Theta ^{k}+\Theta \right) ,  \label{ThetaAsWest}
\end{equation}%
where $\Theta =\left\Vert \theta _{s}\right\Vert _{W^{k-1,r}}+\left\Vert \xi
\right\Vert _{W^{k,r}}.$

Similarly, if $X=X\left( 1\right) ,$ where $X\left( t\right) $ is $1$%
-parameter family of $\mathfrak{l}$-valued maps that satisfies (\ref{Xivp}).
Then, 
\begin{equation*}
\left\Vert X\right\Vert _{W^{k,r}}\lesssim e^{C\left( k+1\right) \left\Vert
\xi \right\Vert _{C^{0}}}\left( \left\Vert X_{0}\right\Vert
_{W^{k,r}}+\left\Vert Y\right\Vert _{W^{k,r}}\right) \left( \Theta
^{k+1}+\Theta \right)
\end{equation*}
\end{theorem}

\begin{proof}
From Lemma \ref{lemThetaPWest}, for each $j\leq k,$ we have the pointwise
estimate 
\begin{equation*}
\left\vert \nabla ^{j-1}\theta _{As}\right\vert \lesssim e^{Cj\left\vert \xi
\right\vert }p_{j-1},
\end{equation*}%
where 
\begin{equation}
p_{j-1}=\sum_{J_{j-1}}\left\vert \theta _{s}\right\vert ^{i_{1}}\left\vert
\nabla \theta _{s}\right\vert ^{i_{2}}...\left\vert \nabla ^{j-1}\theta
_{s}\right\vert ^{i_{j}}\left\vert d\xi \right\vert ^{k_{1}}\left\vert
\nabla d\xi \right\vert ^{k_{2}}...\left\vert \nabla ^{j-1}d\xi \right\vert
^{k_{j}},  \label{pjm1}
\end{equation}%
with $J_{j-1}=\left\{ \left( i_{1},...,i_{j},k_{1},...,k_{j}\right) \in 
\mathbb{N}_{0}^{2j}:\sum_{m=1}^{j}mi_{m}+\sum_{m=1}^{j}mk_{m}=j\right\} .$
Thus, 
\begin{equation*}
\left\Vert \nabla ^{j-1}\theta _{As}\right\Vert _{L^{r}}\lesssim
e^{Cj\left\Vert \xi \right\Vert _{C^{0}}}\left\Vert p_{j-1}\right\Vert
_{L^{r}}.
\end{equation*}%
Now, from Lemma \ref{lemSobProd}, if $\sum_{i^{\prime }=1}^{k^{\prime
}}q_{i^{\prime }}m_{i^{\prime }}\leq k,$ then 
\begin{equation}
\left\Vert \prod_{i^{\prime }=1}^{k^{\prime }}A_{i^{\prime }}^{m_{i^{\prime
}}}\right\Vert _{L^{r}}\lesssim \prod_{i^{\prime }=1}^{k^{\prime
}}\left\Vert A_{i^{\prime }}\right\Vert _{W^{k-q_{i^{\prime
}},r}}^{m_{i^{\prime }}}.
\end{equation}%
We can apply this to (\ref{pjm1}), with the weight $q_{i}=i$ for each $%
\left\vert \nabla ^{i}\theta _{s}\right\vert $ or $\left\vert \nabla
^{i-1}d\xi \right\vert $ factor. Then, 
\begin{eqnarray*}
\left\Vert p_{j-1}\right\Vert _{L^{r}} &\lesssim
&\sum_{J_{j-1}}\prod_{i^{\prime }=1}^{j-1}\left\Vert \nabla ^{i^{\prime
}-1}\theta _{s}\right\Vert _{W^{k-i^{\prime },r}}^{i_{i^{\prime
}}}\prod_{i^{\prime \prime }=1}^{j}\left\Vert \nabla ^{i^{\prime \prime
}-1}d\xi \right\Vert _{W^{k-i^{\prime \prime },r}}^{k_{i^{\prime \prime }}}
\\
&=&\sum_{J_{j-1}}\left\Vert \theta _{s}\right\Vert
_{W^{k-1,r}}^{i_{1}}...\left\Vert \nabla ^{j-1}\theta _{s}\right\Vert
_{W^{k-\left( j-1\right) ,r}}^{i_{j-1}}\left\Vert d\xi \right\Vert
_{W^{k-1,r}}^{k_{1}}...\left\Vert \nabla ^{j-1}d\xi \right\Vert
_{W^{k-j,r}}^{k_{j}}.
\end{eqnarray*}%
Since for each $i^{\prime }\leq k,$ $\left\Vert \nabla ^{i^{\prime
}-1}\theta _{s}\right\Vert _{W^{k-i^{\prime },r}}\lesssim \left\Vert \theta
_{s}\right\Vert _{W^{k-1,r}}$ and$\left\Vert \nabla ^{i^{\prime }-1}d\xi
\right\Vert _{W^{k-i^{\prime },r}}\lesssim \left\Vert \xi \right\Vert
_{W^{k,r}}$, we obtain 
\begin{equation}
\left\Vert p_{j-1}\right\Vert _{L^{r}}\lesssim \sum_{J_{j-1}}\left\Vert
\theta _{s}\right\Vert _{W^{k-1,r}}^{i_{1}+...+i_{j-1}}\left\Vert \xi
\right\Vert _{W^{k,r}}^{k_{1}+...+k_{j}}.  \label{pjLr}
\end{equation}%
The right hand-side of (\ref{pjLr}) is thus a polynomial in $\left\Vert
\theta _{s}\right\Vert _{W^{k-1,r}}$ and $\left\Vert \xi \right\Vert
_{W^{k,r}},$ and from the definition of $J_{j-1},$ the degree of this
polynomial is $j$ and the lowest order terms are $\left\Vert \theta
_{s}\right\Vert _{W^{k-1,r}}$ and $\left\Vert \xi \right\Vert _{W^{k,r}}.$
Hence, we can write 
\begin{equation*}
\left\Vert \nabla ^{j-1}\theta _{As}\right\Vert _{L^{r}}\lesssim
e^{Cj\left\Vert \xi \right\Vert _{C^{0}}}\left( \Theta ^{j}+\Theta \right)
\end{equation*}%
where $\Theta =\left\Vert \theta _{s}\right\Vert _{W^{k-1,r}}+\left\Vert \xi
\right\Vert _{W^{k,r}}.$ In particular, 
\begin{equation*}
\left\Vert \theta _{As}\right\Vert _{L^{r}}\lesssim e^{C\left\Vert \xi
\right\Vert _{C^{0}}}\Theta \lesssim e^{Cj\left\Vert \xi \right\Vert
_{C^{0}}}\left( \Theta ^{j}+\Theta \right) .
\end{equation*}%
Now, since 
\begin{equation*}
\left\Vert \theta _{As}\right\Vert _{W^{k-1,r}}\lesssim \left\Vert \theta
_{As}\right\Vert _{L^{r}}+\left\Vert \nabla ^{k-1}\theta _{As}\right\Vert
_{L^{r}}\lesssim e^{Ck\left\Vert \xi \right\Vert _{C^{0}}}\left( \Theta
^{k}+\Theta \right) ,
\end{equation*}%
which gives us (\ref{ThetaAsWest}).

Now from Lemma \ref{lemNablaX}, for each $j$, 
\begin{equation*}
\left\vert \nabla ^{j}X\right\vert \lesssim e^{\left( j+1\right) C\left\vert
\xi \right\vert }\sum_{j^{\prime }+j^{\prime \prime }=j}p_{j^{\prime
}-1}\left( \left\vert \nabla ^{j^{\prime \prime }}X_{0}\right\vert
+\left\vert \nabla ^{j^{\prime \prime }}Y\right\vert \right) ,
\end{equation*}%
Hence, 
\begin{eqnarray*}
\left\Vert \nabla ^{j}X\right\Vert _{L^{r}} &\lesssim &e^{\left( j+1\right)
C\left\Vert \xi \right\Vert _{C^{0}}}\sum_{j^{\prime }+j^{\prime \prime
}=j}\left\Vert p_{j^{\prime }-1}\left( \left\vert \nabla ^{j^{\prime \prime
}}X_{0}\right\vert +\left\vert \nabla ^{j^{\prime \prime }}Y\right\vert
\right) \right\Vert _{L^{r}} \\
&\lesssim &e^{\left( j+1\right) C\left\Vert \xi \right\Vert
_{C^{0}}}\sum_{J_{j-1}^{\prime }}\left\Vert \left\vert \theta
_{s}\right\vert ^{i_{1}}\left\vert \nabla \theta _{s}\right\vert
^{i_{2}}...\left\vert \nabla ^{j-1}\theta _{s}\right\vert
^{i_{j-1}}\left\vert d\xi \right\vert ^{k_{1}}\left\vert \nabla d\xi
\right\vert ^{k_{2}}...\left\vert \nabla ^{j-1}d\xi \right\vert
^{k_{j}}\times \right. \\
&&\left. \times \left( \left\vert \nabla ^{j^{\prime \prime
}}X_{0}\right\vert +\left\vert \nabla ^{j^{\prime \prime }}Y\right\vert
\right) \right\Vert _{L^{r}}
\end{eqnarray*}%
where 
\begin{equation*}
J_{j-1}^{\prime }=\left\{ \left( i_{1},...,i_{j},k_{1},...,k_{j},j^{\prime
\prime }\right) \in \mathbb{N}_{0}^{2j+1}:\sum_{m=1}^{j}mi_{m}+%
\sum_{m=1}^{j}mk_{m}+j^{\prime \prime }=j\right\} .
\end{equation*}%
Hence, from Lemma \ref{lemSobProd}, 
\begin{eqnarray*}
\left\Vert \nabla ^{j}X\right\Vert _{L^{r}} &\lesssim &e^{\left( j+1\right)
C\left\Vert \xi \right\Vert _{C^{0}}}\sum_{J_{j-1}^{\prime }}\left(
\left\Vert \nabla ^{j^{\prime \prime }}X_{0}\right\Vert _{W^{k-j^{\prime
\prime },r}}+\left\Vert \nabla ^{j^{\prime \prime }}Y\right\Vert
_{W^{k-j^{\prime \prime },r}}\right) \prod_{i^{\prime }=1}^{j-1}\left\Vert
\nabla ^{i^{\prime }-1}\theta _{s}\right\Vert _{W^{k-i^{\prime
},r}}^{i_{i^{\prime }}}\times \\
&&\times \prod_{i^{\prime \prime }=1}^{j}\left\Vert \nabla ^{i^{\prime
\prime }-1}d\xi \right\Vert _{W^{k-i^{\prime \prime },r}}^{k_{i^{\prime
\prime }}} \\
&\lesssim &e^{\left( j+1\right) C\left\Vert \xi \right\Vert _{C^{0}}}\left(
\left\Vert X_{0}\right\Vert _{W^{k,r}}+\left\Vert Y\right\Vert
_{W^{k,r}}\right) \sum_{J_{j-1}^{\prime }}\left\Vert \theta _{s}\right\Vert
_{W^{k-1,r}}^{i_{1}+...+i_{j}}\left\Vert \xi \right\Vert
_{W^{k,r}}^{k_{1}+...+k_{j}} \\
&\lesssim &e^{C\left( j+1\right) \left\Vert \xi \right\Vert _{C^{0}}}\left(
\left\Vert X_{0}\right\Vert _{W^{k,r}}+\left\Vert Y\right\Vert
_{W^{k,r}}\right) \left( \Theta ^{j}+\Theta \right) ,
\end{eqnarray*}%
similarly as before. Hence, we conclude that 
\begin{equation*}
\left\Vert X\right\Vert _{W^{k,r}}\lesssim e^{C\left( k+1\right) \left\Vert
\xi \right\Vert _{C^{0}}}\left( \left\Vert X_{0}\right\Vert
_{W^{k,r}}+\left\Vert Y\right\Vert _{W^{k,r}}\right) \left( \Theta
^{k+1}+\Theta \right) .
\end{equation*}
\end{proof}

\begin{corollary}
Suppose $A\in C^{0}\left( M,\mathbb{L}^{\prime }\right) $ and $s\in
W^{k,r}\left( M,\mathbb{L}\right) $, where $kr>n=\dim M.$ Then, $A\in
W^{k,r}\left( M,\mathbb{L}^{\prime }\right) $ if and only if $\theta
_{As}\in W^{k-1,r}\left( M,T^{\ast }M\otimes \mathfrak{l}\right) $.
\end{corollary}

\begin{proof}
The map $\mu :\mathbb{L}^{\prime }\times \mathbb{L}\longrightarrow \mathbb{L}
$ given by $\left( A,s\right) \mapsto As$ is smooth, hence the composition
with $\mu $ is a continuous map from $W^{k,r}\left( M,\mathbb{L}^{\prime
}\times \mathbb{L}\right) $ to $W^{k,r}\left( M,\mathbb{L}\right) .$ If $%
A\in W^{k,r}\left( M,\mathbb{L}^{\prime }\right) ,$ then since $s\in
W^{k,r}\left( M,\mathbb{L}\right) \subset C^{0}\left( M,\mathbb{L}\right) ,$ 
$As\in W^{k,r}\left( M,\mathbb{L}\right) ,$ and hence from Lemma \ref%
{lemSobLoops}, $\theta _{As}\in W^{k-1,r}\left( M,T^{\ast }M\otimes 
\mathfrak{l}\right) ,$ and thus $\theta _{A}^{\left( s\right) }\in
W^{k-1,r}\left( M,T^{\ast }M\otimes \mathfrak{l}\right) .$

Conversely, if $\theta _{As}\in W^{k-1,r}\left( M,T^{\ast }M\otimes 
\mathfrak{l}\right) ,$ then $As\in W^{k,r}\left( M,\mathbb{L}\right) .$
Since right division is a smooth map, and $s\in C^{0}\left( M,\mathbb{L}%
\right) ,$ we conclude that $A\in W^{k,r}\left( M,\mathbb{L}^{\prime
}\right) .$
\end{proof}

\section{Gauge theory}

\setcounter{equation}{0}\label{sectGauge}Let $M$ be a smooth,
finite-dimensional manifold with a $\Psi $-principal bundle $\pi :\mathcal{P}%
\longrightarrow M.$ 

\begin{definition}
Let $s:$ $\mathcal{P}\longrightarrow \mathbb{L}$ be an equivariant map. In
particular, given $p\in \mathcal{P},$ the equivalence class $\left\lfloor
p,s_{p}\right\rfloor _{\Psi }$ defines a section of the associated bundle $%
\mathcal{Q}=\mathcal{P\times }_{\Psi }\mathbb{L}$, where $\left\lfloor
p,s_{p}\right\rfloor _{\Psi }$ is the equivalence class with respect to the
action of $\Psi :$%
\begin{equation}
\left( p,s_{p}\right) \sim \left( ph,l_{h^{-1}}\left( s_{p}\right) \right)
=\left( ph,s_{ph}\right) \ \ \text{for any }h\in \Psi .
\end{equation}%
We will refer to $s$ as \emph{the defining map} or \emph{defining section}.
\end{definition}

We will define several associated bundles related to $\mathcal{P}.$ As it is
well-known, sections of associated bundles are equivalent to equivariant
maps. With this in mind, we also give properties of equivariant maps that
correspond to sections of these bundles. Let $h\in \Psi $ and, as before,
denote by $h^{\prime }$ the partial action of $h$.

\begin{equation}
\begin{tabular}{l|l|l}
\textbf{Bundle} & \textbf{Equivariant map} & \textbf{Equivariance property}
\\ \hline
$\mathcal{P}$ & $k:\mathcal{P}\longrightarrow \Psi $ & $k_{ph}=h^{-1}k_{p}$
\\ 
$\mathcal{Q}^{\prime }=\mathcal{P}\times _{\Psi ^{\prime }}\mathbb{L}%
^{\prime }$ & $q:\mathcal{P}\longrightarrow \mathbb{L}^{\prime }$ & $%
q_{ph}=\left( h^{\prime }\right) ^{-1}q_{p}$ \\ 
$\mathcal{Q}=\mathcal{P\times }_{\Psi }\mathbb{L}$ & $r:\mathcal{P}%
\longrightarrow \mathbb{L}$ & $r_{ph}=h^{-1}\left( r_{p}\right) $ \\ 
$\mathcal{A}=\mathcal{P\times }_{\Psi _{\ast }^{\prime }}\mathfrak{l}$ & $%
\eta :\mathcal{P}\longrightarrow \mathfrak{l}$ & $\eta _{ph}=\left(
h^{\prime }\right) _{\ast }^{-1}\eta _{p}$ \\ 
$\mathfrak{p}_{\mathcal{P}}=\mathcal{P\times }_{\left( \mathop{\rm Ad}%
\nolimits_{\xi }\right) _{\ast }}\mathfrak{p}$ & $\xi :\mathcal{P}%
\longrightarrow \mathfrak{p}$ & $\xi _{ph}=\left( \mathop{\rm Ad}%
\nolimits_{h}^{-1}\right) _{\ast }\xi _{p}$ \\ 
$\mathop{\rm Ad}\nolimits\left( \mathcal{P}\right) =\mathcal{P}\times _{%
\mathop{\rm Ad}\nolimits_{\Psi }}\Psi $ & $u:\mathcal{P}\longrightarrow \Psi 
$ & $u_{ph}=h^{-1}u_{p}h$%
\end{tabular}
\label{equimap2}
\end{equation}

Given equivariant maps $q,r:\mathcal{P}\longrightarrow \mathbb{L}^{\prime }$%
, define an equivariant product using $s$, given for any $p\in \mathcal{P}$
by%
\begin{equation}
\left. q\circ _{s}r\right\vert _{p}=q_{p}\circ _{s_{p}}r_{p}.
\label{equiprod}
\end{equation}%
Due to Lemma \ref{lemPseudoHom}, the corresponding map $q\circ _{s}r:%
\mathcal{P}\longrightarrow \mathbb{L}^{\prime }$ is equivariant, and hence $%
\circ _{s}$ induces a fiberwise product on sections of $\mathcal{Q}$.
Analogously, we define fiberwise quotients of sections of $\mathcal{Q}.$
Similarly, we define an equivariant bracket $\left[ \cdot ,\cdot \right]
^{\left( s\right) }$ and the equivariant map $\varphi _{s}$. Similarly, the
Killing form $K^{\left( s\right) }$ is then also equivariant.

Suppose the principal $\Psi $-bundle $\mathcal{P}$ has a principal Ehresmann
connection given by the decomposition $T\mathcal{P}=\mathcal{HP}\oplus 
\mathcal{VP}$ and the corresponding vertical $\mathfrak{p}$-valued
connection $1$-form $\omega .$ Given an equivariant map $f:\mathcal{P}%
\longrightarrow S$, define 
\begin{equation}
d^{\omega }f:=f_{\ast }\circ \mathop{\rm proj}\nolimits_{\mathcal{H}}:T%
\mathcal{P}\longrightarrow \mathcal{HP}\longrightarrow TS.  \label{dHftilde}
\end{equation}%
This is then a horizontal map since it vanishes on any vertical vectors. The
map $d^{\mathcal{\omega }}f$ is moreover still equivariant, and hence
induces a covariant derivative on sections of the associated bundle $%
\mathcal{P}\times _{\Psi }S$. If $S$ is a vector space, then this reduces to
the usual definition of the exterior covariant derivative of a vector
bundle-valued function and $d^{\mathcal{\omega }}f$ is a
vector-bundle-valued $1$-form. 

Following \cite{GrigorianLoops}, let us define the torsion of the defining
map $s$ with respect to the connection $\omega $.

\begin{definition}
\label{defTors}The $\emph{torsion}$ $T^{\left( s,\omega \right) }$ of the
defining map $s$ with respect to $\omega $ is a horizontal $\mathfrak{l}$%
-valued $1$-form on $\mathcal{P}$ given by $T^{\left( s,\omega \right)
}=\left( s^{\ast }\theta \right) \circ \mathop{\rm proj}\nolimits_{\mathcal{H%
}}$, where $\theta $ is Maurer-Cartan form of $\mathbb{L}$. Equivalently, at 
$p\in \mathcal{P}$, we have%
\begin{equation}
\left. T^{\left( s,\omega \right) }\right\vert _{p}=\left(
R_{s_{p}}^{-1}\right) _{\ast }\left. d^{\omega }s\right\vert _{p}.
\label{torsdef}
\end{equation}
\end{definition}

Thus, $T^{\left( s,\omega \right) }$ is the horizontal component of $\theta
_{s}=s^{\ast }\theta .$ We also easily see that it is $\Psi $-equivariant.
Thus, $T^{\left( s,\omega \right) }$ is a \emph{basic} (i.e. horizontal and
equivariant) $\mathfrak{l}$-valued $1$-form on $\mathcal{P}$, and thus
defines a $1$-form on $M$ with values in the associated vector bundle $%
\mathcal{A=P\times }_{\Psi _{\ast }^{\prime }}\mathfrak{l}.$ 

%
%
%

Recall that the curvature $F^{\left( \omega \right) }\in \Omega ^{2}\left( 
\mathcal{P},\mathfrak{p}\right) $ of the connection $\omega $ on $\mathcal{P}
$ is given by 
\begin{equation}
F^{\left( \omega \right) }=d\omega \circ \mathop{\rm proj}\nolimits_{%
\mathcal{H}}=d\omega +\frac{1}{2}\left[ \omega ,\omega \right] _{\mathfrak{p}%
},  \label{curvom}
\end{equation}%
where wedge product is implied. Given the defining map $s$, define $\hat{F}%
^{\left( s,\omega \right) }\in \Omega ^{2}\left( \mathcal{P},\mathfrak{l}%
\right) $ to be the projection of the curvature $F^{\left( \omega \right) }$
to $\mathfrak{l}$ with respect to $s$, such that for any $X_{p},Y_{p}\in
T_{p}\mathcal{P},$ 
\begin{equation}
\hat{F}^{\left( s,\omega \right) }=\varphi _{s}\left( F^{\left( \omega
\right) }\right) .  \label{Fhat}
\end{equation}

\begin{theorem}[{\protect\cite[Theorem 4.19]{GrigorianLoops}}]
\label{thmFTstruct}$\hat{F}^{\left( s,\omega \right) }$ and $T^{\left(
s,\omega \right) }$ satisfy the following structure equation 
\begin{equation}
\hat{F}^{\left( s,\omega \right) }=d^{\mathcal{\omega }}T^{\left( s,\omega
\right) }-\frac{1}{2}\left[ T^{\left( s,\omega \right) },T^{\left( s,\omega
\right) }\right] ^{\left( s\right) },  \label{dHT}
\end{equation}%
where a wedge product between the $1$-forms $T^{\left( s,\omega \right) }$
is implied.
\end{theorem}

In the case of an octonion bundle over a $7$-dimensional manifold, this
relationship between the torsion and a curvature component has been shown in 
\cite{GrigorianOctobundle}.

As discussed earlier, equivariant horizontal forms on $\mathcal{P}$ give
rise to sections of corresponding associated bundles over the base manifold $%
M.$ So let us now switch perspective, and work in terms of sections of
bundles. In particular, now we will consider $s$ to be a smooth section of
the bundle $\mathcal{Q},$ so that we will say $s\in \Gamma \left( \mathcal{Q}%
\right) ,$ and will refer to it as the \emph{defining section}. Similarly,
we can also consider sections $A\in \Gamma \left( \mathcal{Q}^{\prime
}\right) ,$ which admit the partial action of $\Psi .$ The product on
elements of $\mathbb{L}^{\prime }$ and $\mathbb{L},$ then carries over to
sections of bundles, so that we have a product $\Gamma \left( \mathcal{Q}%
^{\prime }\right) \times \Gamma \left( \mathcal{Q}\right) \longrightarrow
\Gamma \left( \mathcal{Q}\right) .$

The connection $\omega $ on $\mathcal{P}$ then induces connections on the
associated bundles and correspondingly, covariant derivatives on sections of
these bundles. The torsion $T^{\left( s,\omega \right) }$, as defined
earlier, was a horizontal and equivariant $1$-form on $\mathcal{P}$ with
values in $\mathfrak{l},$ so it uniquely corresponds to a $1$-form on $M$
with values in the bundle $\mathcal{A},$ i.e., now we will consider $%
T^{\left( s,\omega \right) }\in \Omega ^{1}\left( \mathcal{A}\right) .$

In standard gauge theory, the key object is the connection, however, in the
non-associative theory, in addition to the connection $\omega $ we also the
defining section $s.$ We then make the following definition.

\begin{definition}
A \emph{non-associative gauge theory }is defined by the following objects:

\begin{enumerate}
\item A smooth loop $\mathbb{L}$ with a finite-dimensional
pseudoautomorphism Lie group $\Psi $ and tangent algebra $\mathfrak{l}$ at
identity.

\item A smooth manifold $M$ with a principal $\Psi $-bundle $\mathcal{P},$
and associated bundles $\mathcal{Q},$ $\mathcal{Q}^{\prime },$ $\mathcal{A},$
with fibers $\mathbb{L}$, $\mathbb{L}^{\prime },$ and $\mathfrak{l},$
respectively

\item A \emph{configuration }$\left( s,\omega \right) $, where $s\in \Gamma
\left( \mathcal{Q}\right) $ is a defining section and $\omega $ is a
connection on $\mathcal{P}.$ Each configuration carries torsion $T^{\left(
s,\omega \right) }\in \Omega ^{1}\left( \mathcal{A}\right) .$
\end{enumerate}
\end{definition}

As we see, the key components are the loop $\mathbb{L},$ with its
pseudoautomorphism group, and the corresponding principal bundle $\mathcal{P}%
\longrightarrow M.$ Up to a choice of the configuration $\left( s,\omega
\right) ,$ everything else follows uniquely. In particular, the associated
bundles are unique because particular actions of $\Psi $ are used to define
them.

The group $\Psi $ acts via standard gauge transformations on $\omega $ and
also acts on the section $s.$ These actions are related in the following
way, as shown in \cite{GrigorianLoops}, 
\begin{equation}
T^{\left( s,h^{\ast }\omega \right) }=\left( h_{\ast }^{\prime }\right)
^{-1}T^{\left( h\left( s\right) ,\omega \right) }\text{,}  \label{Tshs}
\end{equation}%
where $h$ is a section of $\mathop{\rm Ad}\nolimits\mathcal{P},$ so is
fiberwise in $\Psi .$ However, we will define loop gauge transformations in
the following way.

\begin{definition}
A \emph{loop gauge transformation }is a transformation of the defining
section $s$ by right multiplication by a section $A\in \Gamma \left( 
\mathcal{Q}^{\prime }\right) ,$ such that $s\mapsto As,$ and hence $%
T^{\left( s,\omega \right) }\mapsto T^{\left( As,\omega \right) }.$
\end{definition}

With respect to a loop gauge transformation, the torsion and curvature $\hat{%
F}$ transform in the following way.

\begin{lemma}[{\protect\cite[Theorem 4.28]{GrigorianLoops}}]
Suppose $A\in \Gamma \left( \mathcal{Q}^{\prime }\right) $ and $s\in \Gamma
\left( \mathcal{Q}\right) $. Then, 
\begin{subequations}
\begin{eqnarray}
T^{\left( As,\omega \right) } &=&\left( \rho _{A}^{\left( s\right) }\right)
^{-1}d^{\omega }A+\mathop{\rm Ad}\nolimits_{A}^{\left( s\right) }T^{\left(
s,\omega \right) }  \label{Trom} \\
\hat{F}^{\left( As,\omega \right) } &=&\left( \rho _{A}^{\left( s\right)
}\right) ^{-1}\left( F^{\prime }\cdot A\right) +\mathop{\rm Ad}%
\nolimits_{A}^{\left( s\right) }\hat{F}^{\left( s,\omega \right) },
\label{From}
\end{eqnarray}%
\end{subequations}%
where $F^{\prime }\cdot A$ denotes the infinitesimal action of $\mathfrak{p}$
on $\mathbb{L}.$
\end{lemma}

Let us fix the connection $\omega $, and suppose we have a path $\exp
_{s}\left( t\xi \right) s\in \Gamma \left( \mathcal{Q}\right) .$ Then from
Lemma \ref{lemThetaPath}, just by taking the horizontal projection, we
immediately obtain that the corresponding one-parameter family of torsions
satisfy a similar ODE.

\begin{lemma}
Suppose $T^{\left( s,\omega \right) }$ is the torsion with respect to a
defining section $s\in \Gamma \left( \mathcal{Q}\right) $ and a connection $%
\omega .$ Suppose $A_{t}=\exp _{s}\left( t\xi \right) \in \Gamma \left( 
\mathcal{Q}^{\prime }\right) ,$ then 
\begin{equation}
\frac{d}{dt}T^{\left( A_{t}s,\omega \right) }=\left[ \xi ,T^{\left(
A_{t}s,\omega \right) }\right] ^{\left( A_{t}s\right) }+d^{\omega }\xi 
\label{dTAs}
\end{equation}
\end{lemma}

Using (\ref{Trom}) and (\ref{thetaexp}), given $\xi \in \Gamma \left( 
\mathcal{A}\right) ,$ we get 
\begin{eqnarray}
T^{\left( \left( \exp _{s}\xi \right) s,\omega \right) } &=&U_{\xi }^{\left(
s\right) }T^{\left( s,\omega \right) }  \label{Texps} \\
&&+U_{\xi }^{\left( s\right) }\left( \int_{0}^{1}U_{\xi }^{\left( s\right)
}\left( \tau \right) ^{-1}d\tau \right) d^{\omega }\xi .  \notag
\end{eqnarray}

Now suppose the base manifold $M$ is compact and Riemannian with a metric $g$
and also that the loop $\mathbb{L}$ admits a non-degenerate Killing form on $%
\mathfrak{l}.$ Then, define the functional 
\begin{equation}
\mathcal{E}_{\omega }\left( s\right) =\int_{M}\left\vert T_{{}}^{\left(
s,\omega \right) }\right\vert _{\left( s\right) }^{2}\mathop{\rm vol}%
\nolimits_{g},  \label{Esom}
\end{equation}%
where $\left\vert {}\right\vert _{\left( s\right) }$ is a combination of the
metric $g$ on $M$ and the Killing form $\left\langle {}\right\rangle
^{\left( s\right) }$ on sections of $\mathcal{A}.$ Critical points then
become analogues of the Coulomb gauge condition in gauge theory \cite%
{DonaldsonGauge,GrigorianOctobundle,GrigorianIsoflow,GrigorianIsoFlowSurvey,GrigorianLoops,SaEarpLoubeau}

\begin{theorem}
\label{thmTcrit}Suppose $\mathbb{L}$ is a semisimple Moufang loop, then the
critical points of the functional (\ref{Esom}) with respect to deformations
of the defining section $s$ are those for which 
\begin{equation}
\left( d^{\omega }\right) ^{\ast }T^{\left( s,\omega \right) }=0.
\end{equation}
\end{theorem}

\begin{proof}
From Lemma \ref{lemKMalcev}, we know that for a Moufang loop, $K^{\left(
s\right) }$ is actually independent of $s.$ Moreover, it is invariant under $%
\mathop{\rm ad}\nolimits^{\left( s\right) }.$ Let us consider deformations
of $s$. The semisimple condition implies $K^{\left( s\right) }$ is
non-degenerate. Consider a path $s_{t}=\exp _{s}\left( t\xi \right) s$ where 
$\xi \in \Gamma \left( \mathcal{A}\right) .$ Then, 
\begin{eqnarray*}
\left. \frac{d}{dt}\mathcal{E}\left( s_{t},\omega \right) \right\vert _{t=0}
&=&\left. \frac{d}{dt}\int_{M}\left\vert T_{{}}^{\left( s_{t},\omega \right)
}\right\vert _{s}^{2}\mathop{\rm vol}\nolimits_{g}\right\vert _{t=0} \\
&=&2\int_{M}\left\langle T^{\left( s,\omega \right) },\left. \frac{d}{dt}%
T^{\left( s_{t},\omega \right) }\right\vert _{t=0}\right\rangle \mathop{\rm
vol}\nolimits_{g} \\
&=&2\int_{M}\left\langle T^{\left( s,\omega \right) },\left[ \xi ,T^{\left(
s,\omega \right) }\right] ^{\left( s\right) }+d^{\omega }\xi \right\rangle %
\mathop{\rm vol}\nolimits_{g},
\end{eqnarray*}%
where we have used (\ref{dTAs}). Note that 
\begin{eqnarray*}
\left\langle T^{\left( s,\omega \right) },\left[ \xi ,T^{\left( s,\omega
\right) }\right] ^{\left( s\right) }\right\rangle  &=&g^{ab}\left\langle
T_{a}^{\left( s,\omega \right) },\left[ \xi ,T_{b}^{\left( s,\omega \right) }%
\right] ^{\left( s\right) }\right\rangle  \\
&=&-g^{ab}\left\langle \left[ T_{a}^{\left( s,\omega \right) },T_{b}^{\left(
s,\omega \right) }\right] ,\xi ^{\left( s\right) }\right\rangle  \\
&=&0.
\end{eqnarray*}%
Hence, 
\begin{equation*}
\left. \frac{d}{dt}\mathcal{E}\left( s_{t},\omega \right) \right\vert
_{t=0}=2\int_{M}\left\langle \left( d^{\omega }\right) ^{\ast }T^{\left(
s,\omega \right) },\xi \right\rangle \mathop{\rm vol}\nolimits_{g}.
\end{equation*}%
Thus critical points of $\mathcal{E}$ with respect to deformations of $s$
satisfy 
\begin{equation}
\left( d^{\omega }\right) ^{\ast }T^{\left( s,\omega \right) }=0.
\label{divT0}
\end{equation}
\end{proof}

\begin{remark}
In Theorem \ref{thmTcrit}, we use the fact that the tangent algebra of a
Moufang loop is a Malcev algebra, i.e. is alternative and satisfies the
additional identity \ref{MalcevId}. Moreover, the semisimple condition
implies that the Killing form is non-degenerate. As noted in Remark \ref%
{remMalc}, the full Malcev algebra condition is likely to be too strong, and
a weaker assumption may be sufficient to obtain these key properties and in
fact obtain $\left( d^{\omega }\right) ^{\ast }T^{\left( s,\omega \right)
}=0 $ as the equation for critical points. On the other hand, other
techniques, such as introducing a different metric (such as the
Killing-Ricci form on Lie triple systems \cite{KikkawaKillingRicci}) or
introducing modified connections may produce similar results in other
settings.
\end{remark}

To prove existence of transformations of $s$ that lead to $\left( d^{\omega
}\right) ^{\ast }T^{\left( s,\omega \right) }=0,$ we will adapt the
procedures from \cite{Feehan2020}, and in particular will apply the Banach
Space Implicit Function Theorem (Theorem \ref{thmIFT}) . The relevant Banach
spaces for us will be spaces of sections with appropriate regularity. The
previously used notations $\Gamma $ and $\Omega ^{k}$ will always denote
smooth sections and smooth bundle-valued forms, respectively. Given a smooth
defining section $s\in \Gamma \left( \mathcal{Q}\right) $ and a smooth
connection $\omega ,$ for any $k\in \mathbb{N}$ and $q\in \left[ 1,\infty %
\right] ,$ denote by $W_{\left( s,\omega \right) }^{k,q}\left( \Lambda
^{l}T^{\ast }M\otimes \mathcal{A}\right) $ the Sobolev space of sections of $%
\Lambda ^{l}T^{\ast }M\otimes \mathcal{A}$ with the norm given by 
\begin{equation*}
\left\Vert \chi \right\Vert _{W_{\left( s,\omega \right) }^{k,q}}=\left(
\int_{M}\left\vert \chi \right\vert _{\left( s\right) }^{q}\mathop{\rm vol}%
\nolimits_{g}\right) ^{\frac{1}{q}}+\left( \int_{M}\left\vert \left(
d^{\omega }\right) ^{k}\chi \right\vert _{\left( s\right) }^{q}\mathop{\rm
vol}\nolimits_{g}\right) ^{\frac{1}{q}},
\end{equation*}%
for $1\leq q<\infty $ and 
\begin{equation*}
\left\Vert \chi \right\Vert _{W_{\left( s,\omega \right) }^{k,\infty }}=%
\mathop{\rm ess}\nolimits\sup_{X}\left\vert \chi \right\vert _{\left(
s\right) }+\mathop{\rm ess}\nolimits\sup_{X}\left\vert \left( d^{\omega
}\right) ^{k}\chi \right\vert _{\left( s\right) }.
\end{equation*}%
Similarly we will denote $W_{\left( s,\omega \right) }^{0,r}$ as $L_{\left(
s\right) }^{r}.$

By Definition \ref{defTors}, the torsion of $\left( s,\omega \right) $ is
just the horizontal component of $\theta _{s},$ so we can immediately adapt
the estimates from Section \ref{sectMaps}, we obtain the following estimates
for torsion.

\begin{lemma}
\label{lemTAsest}Suppose $\mathbb{L}$ is a smooth compact loop with tangent
algebra $\mathfrak{l}$ and pseudoautomorphism group $\Psi .$ Let $\left(
M,g\right) $ be a closed, smooth Riemannian manifold of dimension $n\geq 2,$
and let $\mathcal{P}$ be a $\Psi $-principal bundle over $M$ with and let $%
\mathcal{A}$ be the associated vector bundle to $\mathcal{P}$ with fibers
isomorphic to $\mathfrak{l}$. Let $\omega $ be a smooth connection on $%
\mathcal{P}\ $and let $s\in \Gamma \left( \mathcal{Q}\right) $ be a smooth
defining section. Also, suppose $k$ is a non-negative integer and $r\geq 0$
such that $kr>n.$ Let $\xi \in W_{\left( s,\omega \right) }^{k,r}\left( 
\mathcal{A}\right) ,$ and suppose $A=\exp _{s}\left( \xi \right) .$ Then, 
\begin{equation}
\left\Vert T^{\left( As,\omega \right) }\right\Vert _{W_{\left( s,\omega
\right) }^{k-1,r}}\lesssim e^{Ck\left\Vert \xi _{\left( s,\omega \right)
}\right\Vert _{C^{0}}}\left( \Theta ^{k}+\Theta \right) ,  \label{TAsest}
\end{equation}%
where $\Theta =\left\Vert T^{\left( s,\omega \right) }\right\Vert
_{W_{\left( s,\omega \right) }^{k-1,r}}+\left\Vert \xi \right\Vert
_{W_{\left( s,\omega \right) }^{k,r}}.$
\end{lemma}

\begin{lemma}
\label{lemAreg}Now suppose that $\left( k^{\prime }-1\right) r\geq n$. Given
other hypotheses the same as in Lemma \ref{lemTAsest}, if $\xi \in W_{\left(
s,\omega \right) }^{k^{\prime },r}\left( \mathcal{A}\right) ,$ and given $%
A=\exp _{s}\left( \xi \right) $ such that 
\begin{equation}
\left( d^{\omega }\right) ^{\ast }T^{\left( As,\omega \right) }=0,
\label{domstT}
\end{equation}%
then in fact $A$ is smooth.
\end{lemma}

\begin{proof}
Using the Whitney Embedding Theorem, suppose $\mathbb{L}$ is smoothly
embedded in some $\mathbb{R}^{N}.$ We can define a loop product and quotient
on the image of the embedding. Hence the bundles $\mathcal{Q}$ and $\mathcal{%
Q}^{\prime }$ can be regarded as subbundles of a vector bundle over $M$. In
particular, since $s$ is smooth and $\exp _{s}:\mathfrak{l}\longrightarrow 
\mathbb{L}$ is also a smooth map, we find that since $k^{\prime }r>n$ and $%
\xi \in W^{k^{\prime },r}\left( \mathcal{A}\right) ,$ then $A=\exp
_{s}\left( \xi \right) \in W^{k^{\prime },r}\left( \mathcal{Q}^{\prime
}\right) \subset C^{0}\left( \mathcal{Q}^{\prime }\right) .$ Using (\ref%
{Trom}), we have 
\begin{eqnarray*}
\left( d^{\omega }\right) ^{\ast }T^{\left( As,\omega \right) } &=&\left(
d^{\omega }\right) ^{\ast }\left( \left( \rho _{A}^{\left( s\right) }\right)
^{-1}d^{\omega }A+\mathop{\rm Ad}\nolimits_{A}^{\left( s\right) }T^{\left(
s,\omega \right) }\right) \\
&=&\left( \rho _{A}^{\left( s\right) }\right) ^{-1}\left( d^{\omega }\right)
^{\ast }d^{\omega }A-\left\langle d^{\omega }\left( \rho _{A}^{\left(
s\right) }\right) ^{-1},d^{\omega }A\right\rangle _{TM}+\left( d^{\omega
}\right) ^{\ast }\left( \mathop{\rm Ad}\nolimits_{A}^{\left( s\right)
}T^{\left( s,\omega \right) }\right) ,
\end{eqnarray*}%
where $\left\langle \cdot ,\cdot \right\rangle _{TM}$ is the inner product
on $TM.$ Thus, we can rewrite (\ref{domstT}) as%
\begin{eqnarray}
\left( d^{\omega }\right) ^{\ast }d^{\omega }A &=&\rho _{A}^{\left( s\right)
}\left\langle d^{\omega }\left( \rho _{A}^{\left( s\right) }\right)
^{-1},d^{\omega }A\right\rangle _{TM}-\rho _{A}^{\left( s\right) }\left(
\left( d^{\omega }\right) ^{\ast }\left( \mathop{\rm Ad}\nolimits_{A}^{%
\left( s\right) }T^{\left( s,\omega \right) }\right) \right)  \label{LapA} \\
&=&-\left\langle \left( d^{\omega }\rho _{A}^{\left( s\right) }\right)
\left( \rho _{A}^{\left( s\right) }\right) ^{-1},d^{\omega }A\right\rangle
_{TM}+\rho _{A}^{\left( s\right) }\left\langle d^{\omega }\mathop{\rm Ad}%
\nolimits_{A}^{\left( s\right) },T^{\left( s,\omega \right) }\right\rangle \\
&&-\rho _{A}^{\left( s\right) }\left( \mathop{\rm Ad}\nolimits_{A}^{\left(
s\right) }\left( d^{\omega }\right) ^{\ast }T^{\left( s,\omega \right)
}\right) .
\end{eqnarray}%
Now since $A\in C^{0}$, and $T^{\left( s,\omega \right) }$ is smooth, for
any $p>0$, 
\begin{equation}
\left\Vert \left( d^{\omega }\right) ^{\ast }d^{\omega }A\right\Vert
_{L^{p}}\leq c\left( A\right) \left( \left\Vert \left\vert d^{\omega
}A\right\vert ^{2}\right\Vert _{L^{p}}+\left\Vert d^{\omega }A\right\Vert
_{L^{p}}+\left\Vert \left( d^{\omega }\right) ^{\ast }T^{\left( s,\omega
\right) }\right\Vert _{L^{p}}\right)  \label{LapA2}
\end{equation}%
Also, $\left\Vert \left\vert d^{\omega }A\right\vert ^{2}\right\Vert
_{L^{p}}\leq \left\Vert d^{\omega }A\right\Vert _{L^{2p}}^{2}.$ Since $\frac{%
k^{\prime }-1}{n}\geq \frac{1}{r},$ we see that $\frac{k^{\prime }-1}{n}>%
\frac{1}{r}-\frac{1}{q}$ for any $q>0.$ By the Sobolev Embedding Theorem,
this shows that 
\begin{equation*}
\left\Vert d^{\omega }A\right\Vert _{L^{q}}\lesssim \left\Vert d^{\omega
}A\right\Vert _{W^{k^{\prime }-1,r}}\lesssim \left\Vert A\right\Vert
_{W^{k^{\prime },r}.}
\end{equation*}%
Thus, (\ref{LapA2}) shows that $\left\Vert \left( d^{\omega }\right) ^{\ast
}d^{\omega }A\right\Vert _{L^{p}}$ is bounded. By elliptic regularity, this
implies that $A\in W^{2,p}.$ In particular, if $p>n,$ then $k^{\prime }>%
\frac{n}{p}+1,$ and thus $A\in C^{1}.$ Bootstrapping the elliptic regularity
argument we then obtain the smoothness of $A$. In particular, note that this
does not depend on the choice of embedding.
\end{proof}

\begin{remark}
The proof of Lemma \ref{lemAreg} is an adaptation of the proof of \cite[%
Proposition 2.3.4]{DonaldsonKronheimer}, where in particular the regularity
of a gauge transformation to the Coulomb gauge was proved. In that case, $r=2
$ and $n=4,$ so the conditions $kr>n$ and $\left( k-1\right) r\geq n$ were
equivalent since $k$ is an integer. More generally, the condition that is
needed for smoothness is somewhat stronger than the one needed for
continuity.
\end{remark}

Then, we have the main theorem.

\begin{theorem}
\label{thmMain}Suppose $\mathbb{L}$ is a smooth compact loop with tangent
algebra $\mathfrak{l}$ and pseudoautomorphism group $\Psi .$ Let $\left(
M,g\right) $ be a closed, smooth Riemannian manifold of dimension $n\geq 2,$
and let $\mathcal{P}$ be a $\Psi $-principal bundle over $M$ with and let $%
\mathcal{A}$ be the associated vector bundle to $\mathcal{P}$ with fibers
isomorphic to $\mathfrak{l}$. Let $\omega $ be a smooth connection on $%
\mathcal{P}.$ Also, suppose $k$ is a non-negative integer and $r\geq 0$ such
that $kr>n.$ Then, there exist constants $\delta \in (0,1]$ and $K\in \left(
0,\infty \right) ,$ such that if $s\in \Gamma \left( \mathcal{Q}\right) $ is
a smooth defining section for which 
\begin{equation*}
\left\Vert T^{\left( s,\omega \right) }\right\Vert _{W_{\left( s,\omega
\right) }^{k-1,r}}<\delta ,
\end{equation*}%
then there exists a section $A\in W^{k,r}\left( \mathcal{Q}^{\prime }\right)
,$ such that 
\begin{equation*}
\left( d^{\omega }\right) ^{\ast }T^{\left( As,\omega \right) }=0
\end{equation*}%
and 
\begin{equation}
\left\Vert T^{\left( As,\omega \right) }\right\Vert _{W_{\left( s,\omega
\right) }^{k-1,r}}<K\left\Vert T^{\left( s,\omega \right) }\right\Vert
_{W^{k-1,r}}\left( 1+\left\Vert T^{\left( s,\omega \right) }\right\Vert
_{W^{k-1,r}}^{k-1}\right) .  \label{TAsEst2}
\end{equation}%
If moreover, $\left( k-1\right) r\geq n,$ then $A$ is smooth.
\end{theorem}

\begin{proof}
Consider $\xi \in $ $W_{\left( s,\omega \right) }^{k,r}\left( \mathcal{A}%
\right) \ \ $and $a\in W_{\left( s,\omega \right) }^{\left( k-1\right)
,r}\left( T^{\ast }M\otimes \mathcal{A}\right) .$ For now, let us drop the $%
\left( s,\omega \right) $ subscript in function spaces. Since $k$ and $r$
satisfy $kr>n,$ by the Sobolev Embedding Theorem, $W^{k,r}$ embeds in $C^{0}$%
. Define the function%
\begin{equation*}
G:W^{\left( k-1\right) ,r}\left( T^{\ast }M\otimes \mathcal{A}\right) \times
W^{k,r}\left( \mathcal{A}\right) \longrightarrow W^{\left( k-2\right)
,r}\left( \mathcal{A}\right) 
\end{equation*}%
by 
\begin{equation}
G\left( a,\xi \right) =\left( d^{\omega }\right) ^{\ast }\left( U_{\xi
}^{\left( s\right) }a+U_{\xi }^{\left( s\right) }\left( \int_{0}^{1}U_{\xi
}^{\left( s\right) }\left( \tau \right) ^{-1}d\tau \right) d^{\omega }\xi
\right) .  \label{Gdef}
\end{equation}%
The assumption that $\xi \in C^{0},$ together with the smoothness of $U_{\xi
}^{\left( s\right) }$ and the derivative maps, leads to the conclusion that $%
G$ is a smooth map of Banach spaces. Note that using (\ref{Texps}), we can
write 
\begin{equation}
G\left( a,\xi \right) =\left( d^{\omega }\right) ^{\ast }\left( U_{\xi
}^{\left( s\right) }\left( a-T^{\left( s,\omega \right) }\right) +T^{\left(
\left( \exp _{s}\xi \right) s,\omega \right) }\right)   \label{Gdef2}
\end{equation}%
Using the connection $\omega ,$ let us define the bundle-valued Hodge
Laplacian 
\begin{equation*}
\Delta ^{\left( \omega \right) }=\left( d^{\omega }\right) ^{\ast }d^{\omega
}+\left( d^{\omega }\right) ^{\ast }d^{\omega }.
\end{equation*}%
On $0$-forms it reduces to $\Delta ^{\left( \omega \right) }=\left(
d^{\omega }\right) ^{\ast }d^{\omega }.$ It extends as an operator of
Sobolev spaces as 
\begin{equation*}
\Delta ^{\left( \omega \right) }:W^{k,r}\left( \Lambda ^{l}T^{\ast }M\otimes 
\mathcal{A}\right) \longrightarrow W^{\left( k-2\right) ,r}\left( \Lambda
^{l}T^{\ast }M\otimes \mathcal{A}\right) ,
\end{equation*}%
and by standard elliptic theory is Fredholm with index $0$ and a closed range%
\begin{equation*}
\left( \ker \Delta ^{\left( \omega \right) }\right) ^{\perp }\cap W^{\left(
k-2\right) ,r}\left( \Lambda ^{l}T^{\ast }M\otimes \mathcal{A}\right) ,
\end{equation*}%
where $\perp $ denotes the $L^{2}$-orthogonal complement.

To be able to apply the Implicit Function Theorem (Theorem \ref{thmIFT}), in
(\ref{Gdef}), let us constrain $\xi \in \left( \ker \Delta ^{\left( \omega
\right) }\right) ^{\perp },$ and we also see that $\mathop{\rm Im}\nolimits %
G\subset \left( \ker \Delta ^{\left( \omega \right) }\right) ^{\perp }$.
This can be seen immediately. Suppose $\sigma =\left( d^{\left( \omega
\right) }\right) ^{\ast }\rho $ for some $\rho \in W^{\left( k-1\right)
,r}\left( T^{\ast }M\otimes \mathcal{A}\right) $ and $\gamma \in \ker \Delta
^{\left( \omega \right) }\subset W^{\left( k-2\right) ,r}\left( \mathcal{A}%
\right) $, then 
\begin{equation*}
\left\langle \sigma ,\gamma \right\rangle _{L^{2}}=\left\langle \left(
d^{\left( \omega \right) }\right) ^{\ast }\rho ,\gamma \right\rangle
_{L^{2}}=\left\langle \rho ,d^{\left( \omega \right) }\gamma \right\rangle
_{L^{2}}=0\,,
\end{equation*}%
since on a compact manifold, $\gamma \in \ker \Delta ^{\left( \omega \right)
}$ if and only if $d^{\left( \omega \right) }\gamma =0.$ Hence the image of $%
G$ is contained in $\left( \ker \Delta ^{\left( \omega \right) }\right)
^{\perp },$ which we'll denote for brevity by $K^{\perp }$, and so in fact, 
\begin{equation*}
G:W^{\left( k-1\right) ,r}\left( T^{\ast }M\otimes \mathcal{A}\right) \times
\left( K^{\perp }\cap W^{k,r}\left( \mathcal{A}\right) \right)
\longrightarrow K^{\perp }\cap W_{A_{1}}^{\left( k-2\right) ,r}\left( 
\mathcal{A}\right) .
\end{equation*}%
Now let us consider the differential of $G$ at $\left( a,\xi \right) =0$ in
the direction $\left( b,\eta \right) \in W^{\left( k-1\right) ,r}\left(
T^{\ast }M\otimes \mathcal{A}\right) \times \left( K^{\perp }\cap
W^{k,r}\left( \mathcal{A}\right) \right) :$ 
\begin{eqnarray*}
\left. DG\right\vert _{\left( 0,0\right) }\left( b,\eta \right)  &=&\left. 
\frac{d}{dt}\left( d^{\omega }\right) ^{\ast }\left( U_{t\eta }^{\left(
s\right) }\left( tb\right) +U_{t\eta }^{\left( s\right) }\left(
\int_{0}^{1}U_{t\eta }^{\left( s\right) }\left( \tau \right) ^{-1}d\tau
\right) d^{\omega }\left( t\eta \right) \right) \right\vert _{t=0} \\
&=&\left( d^{\omega }\right) ^{\ast }b+\left( d^{\omega }\right) ^{\ast
}d^{\omega }\eta ,
\end{eqnarray*}%
since $U_{0}^{\left( s\right) }=\mathop{\rm id}\nolimits_{\mathfrak{l}}.$ In
particular, the partial derivative in the second direction is given by%
\begin{equation*}
\left. \partial _{2}G\right\vert _{\left( 0,0\right) }\left( \eta \right)
=\Delta ^{\left( \omega \right) }\eta .
\end{equation*}%
In Theorem \ref{thmIFT}, let 
\begin{eqnarray*}
X &=&W^{\left( k-1\right) ,r}\left( T^{\ast }M\otimes \mathcal{A}\right)  \\
Y &=&K^{\perp }\cap W^{k,r}\left( \mathcal{A}\right)  \\
Z &=&K^{\perp }\cap W_{A_{1}}^{\left( k-2\right) ,r}\left( \mathcal{A}%
\right) .
\end{eqnarray*}%
Then, the map $\left. \partial _{2}G\right\vert _{\left( 0,0\right)
}:Y\longrightarrow Z$ is an isomorphism, and we define 
\begin{equation*}
N=\left\Vert \left( \left. \partial _{2}G\right\vert _{\left( 0,0\right)
}\right) ^{-1}\right\Vert _{\mathop{\rm Hom}\nolimits\left( Z,Y\right) }.
\end{equation*}%
Let%
\begin{eqnarray*}
U &=&\left\{ x\in W^{\left( k-1\right) ,r}\left( T^{\ast }M\otimes \mathcal{A%
}\right) :\ \left\Vert x\right\Vert _{W^{\left( k-1\right) ,r}}<\zeta
\right\} \subset X \\
V &=&\left\{ y\in K^{\perp }\cap W^{k,r}\left( \mathcal{A}\right) :\
\left\Vert y\right\Vert _{W^{k,r}}<\zeta \right\} \subset Y,
\end{eqnarray*}%
where $\zeta \in (0,1]$ is small enough such that 
\begin{equation*}
\sup_{\left( x,y\right) \in U\times V}\left\Vert \left. \partial
_{2}G\right\vert _{\left( x,y\right) }-\left. \partial _{2}G\right\vert
_{\left( 0,0\right) }\right\Vert _{_{\mathop{\rm Hom}\nolimits\left(
Y,Z\right) }}\leq \frac{1}{2N}.
\end{equation*}%
Also define the constant $\beta $ as 
\begin{equation*}
\beta =\sup_{\left( x,y\right) \in U\times V}\left\Vert \left. \partial
_{1}G\right\vert _{\left( x,y\right) }\right\Vert _{\mathop{\rm Hom}\nolimits%
\left( X,Z\right) }<\infty .
\end{equation*}%
Then, by the conclusion of Theorem \ref{thmIFT}, there exist an open set $%
\tilde{U}\subset U,$ given by 
\begin{equation*}
\tilde{U}=\left\{ x\in W^{\left( k-1\right) ,r}\left( T^{\ast }M\otimes 
\mathcal{A}\right) :\ \left\Vert x\right\Vert _{W^{\left( k-1\right)
,r}}<\delta \right\} \subset X,
\end{equation*}%
where $\delta \in \left( 0,\min \left\{ \zeta ,\frac{\zeta }{2\beta N}%
\right\} \right] $, and a unique smooth map 
\begin{equation*}
\xi :\tilde{U}\longrightarrow V,
\end{equation*}%
such that $\xi \left( 0\right) =0,$ and 
\begin{eqnarray*}
G\left( a,\xi \left( a\right) \right)  &=&0,\ \forall a\in \tilde{U} \\
\left. D\xi \right\vert _{a} &=&-\left( \left. \partial _{2}G\right\vert
_{\left( a,\xi \left( a\right) \right) }\right) ^{-1}\left. \partial
_{1}G\right\vert _{\left( a,\xi \left( a\right) \right) }\in \mathop{\rm Hom}%
\nolimits\left( X,Y\right) ,\ \forall a\in \tilde{U} \\
\left\Vert \xi \left( a_{1}\right) -\xi \left( a_{2}\right) \right\Vert _{Y}
&\leq &2\beta N\left\Vert a_{1}-a_{2}\right\Vert _{X,}\ \forall
a_{1},a_{2}\in \tilde{U}.
\end{eqnarray*}

In particular, for any $a\in W^{\left( k-1\right) ,r}\left( T^{\ast
}M\otimes \mathcal{A}\right) $ with $\left\Vert a\right\Vert
_{W_{{}}^{\left( k-1\right) ,r}}<\delta ,$ there exists a section $A\left(
a\right) =\exp _{s}\left( \xi \left( a\right) \right) $ with $\xi \in
W^{k,r}\left( \mathcal{A}\right) ,$ for which 
\begin{equation*}
\left( d^{\omega }\right) ^{\ast }\left( U_{\xi }^{\left( s\right) }\left(
a-T^{\left( s,\omega \right) }\right) +T^{\left( As,\omega \right) }\right)
=0,
\end{equation*}%
and 
\begin{equation*}
\left\Vert \xi \left( a\right) \right\Vert _{W^{k,r}}\leq 2\beta N\left\Vert
a\right\Vert _{W^{\left( k-1\right) ,r}}.
\end{equation*}%
Since $s$ is smooth and $\exp _{s}:\mathfrak{l}\longrightarrow \mathbb{L}$
is a smooth map, this shows that $A$ $\in W^{k,r}\left( \mathcal{Q}^{\prime
}\right) .$

Now suppose $s$ and $\omega $ are such that $\left\Vert T^{\left( s,\omega
\right) }\right\Vert _{W_{{}}^{\left( k-1\right) ,r}}<\delta ,$ then setting 
$a=T^{\left( s,\omega \right) }$ gives $\xi _{\left( s,\omega \right) }=\xi
\left( T^{\left( s,\omega \right) }\right) ,$ for which 
\begin{eqnarray*}
\left( d^{\omega }\right) ^{\ast }\left( T^{\left( As,\omega \right)
}\right) &=&0 \\
\left\Vert \xi _{\left( s,\omega \right) }\right\Vert _{W^{k,r}} &<&\zeta \\
\left\Vert \xi _{\left( s,\omega \right) }\right\Vert _{W^{k,r}} &\leq
&2\beta N\left\Vert T^{\left( s,\omega \right) }\right\Vert _{W^{\left(
k-1\right) ,r}},
\end{eqnarray*}%
where $A=\exp _{s}\left( \xi _{\left( s,\omega \right) }\right) .$ From (\ref%
{TAsest}), we have 
\begin{equation}
\left\Vert T^{\left( As,\omega \right) }\right\Vert _{W^{k-1,r}}\lesssim
e^{Ck\left\Vert \xi _{\left( s,\omega \right) }\right\Vert _{C^{0}}}\left(
\Theta ^{k}+\Theta \right) ,
\end{equation}%
where $\Theta =\left( \left\Vert T^{\left( s,\omega \right) }\right\Vert
_{W^{k-1,r}}+\left\Vert \xi _{\left( s,\omega \right) }\right\Vert
_{W^{k,r}}\right) .$ Now, using the estimate for $\xi $ in terms of $T$, we
get 
\begin{equation*}
\Theta \lesssim \left( 1+2\beta N\right) \left( \left\Vert T^{\left(
s,\omega \right) }\right\Vert _{W^{k-1,r}}\right) ,
\end{equation*}%
and since $kr>n$, $\left\Vert \xi _{\left( s,\omega \right) }\right\Vert
_{C^{0}}\lesssim \left\Vert \xi _{\left( s,\omega \right) }\right\Vert
_{W^{k,r}}<\zeta .$ Overall, combining the constants into a single constant $%
K,$ we obtain 
\begin{equation}
\left\Vert T^{\left( As,\omega \right) }\right\Vert _{W^{k-1,r}}<K\left\Vert
T^{\left( s,\omega \right) }\right\Vert _{W^{k-1,r}}\left( 1+\left\Vert
T^{\left( s,\omega \right) }\right\Vert _{W^{k-1,r}}^{k-1}\right) ,
\label{TAsEst2a}
\end{equation}%
and hence (\ref{TAsEst2}).

If $\left( k-1\right) r\geq n,$ then by Lemma \ref{lemAreg}, we see that $A$
is smooth.
\end{proof}

\section{$G_{2}$-manifolds}

\setcounter{equation}{0}\label{secG2}The general picture considered in the
previous sections can now be specialized to the case of manifolds with $%
G_{2} $-structure. The 14-dimensional group $G_{2}$ is the smallest of the
five exceptional Lie groups and is defined as the automorphism group of the
loop of unit octonions $U\mathbb{O}$. Let $M$ be a compact $7$-dimensional
manifold with vanishing first and second Stiefel-Whitney classes, so that
the manifold is both orientable and admit a spin structure. Then, as it is
well-known \cite{FernandezGray, FriedrichNPG2}, $M$ admits a $G_{2}$%
-structure, that is a reduction of the structure group of the frame bundle
to $G_{2}.$ Since $G_{2}$ is a subgroup of $SO\left( 7\right) $, the $G_{2}$%
-structure can be extended uniquely to an $SO\left( 7\right) $-structure,
and thus defines a Riemannian metric $g$ and orientation on $M.$
Equivalently, given a Riemannian metric $g$, an $SO\left( 7\right) $%
-structure on $M$ lifts to a spin structure, which is a principal $%
\mathop{\rm Spin}\nolimits\left( 7\right) $-structure. Given the spin
structure, we can then construct a spinor bundle $\mathcal{S}$ which will
necessarily admit a nowhere vanishing section. Any such spinor section will
then reduce the spin structure to a $G_{2}$-structure on $M$. Indeed, any
unit spinor will hence define a $G_{2}$-structure that is compatible with
the metric $g.$

Recall that $\mathop{\rm Spin}\nolimits\left( 7\right) $ has three
low-dimensional real irreducible representations: $1$-dimensional
representation $V_{1},$ $7$-dimensional \textquotedblleft
vector\textquotedblright\ representation $V_{7},$ and the $8$-dimensional
\textquotedblleft spinor\textquotedblright\ representation $S_{7}$ \cite%
{BaezOcto}. The representations $V_{1}$ and $V_{7}$ descend to
representations of $SO\left( 7\right) .$ Moreover, the Clifford product
gives the map%
\begin{equation}
V_{7}\times S_{7}\longrightarrow S_{7}.  \label{cliffprod}
\end{equation}%
Setting $V_{8}=V_{1}\oplus V_{7},$ we can then extend this map to $%
m:V_{8}\times S_{7}\longrightarrow S_{7}.$ This product is non-degenerate,
and fixing $\xi \in S_{7}$ allows to identify $V_{8}$ with $S_{7}.$ Both
spaces are then identified with the octonions and the product $m$ then gives
rise to octonion multiplication. The element $\xi $ is identified with $1\in 
\mathbb{O}.$ The stabilizer of $\xi \in S_{7}$ under the action of $%
\mathop{\rm Spin}\nolimits\left( 7\right) $ is isomorphic to $G_{2}.$ Note
that $V_{8}$ here then corresponds to the irreducible \textquotedblleft
vector\textquotedblright\ representation of $\mathop{\rm
Spin}\nolimits\left( 8\right) ,$ while the two copies of $S_{7}$ are
identified with the irreducible $8$-dimensional chiral spinor
representations $S_{8}^{\pm }$ of $\mathop{\rm Spin}\nolimits\left( 8\right)
,$ and thus gives the normed triality of $\mathop{\rm Spin}\nolimits\left(
8\right) $ \cite{BaezOcto}. Since the map $m$ preserves norms, it restricts
to unit spheres in $V_{8}$ and $S_{7},$ which we will denote by $U\mathbb{O}%
^{\prime }$ and $U\mathbb{O},$ respectively, because they correspond to $%
\mathbb{L}^{\prime }$ and $\mathbb{L}$ in the general theory in Section \ref%
{sectLoop}. Clearly, $U\mathbb{O}$ is a compact smooth loop.The tangent
space at $1$ to $U\mathbb{O}$ is then isomorphic to $\mathbb{R}^{7}\cong %
\mathop{\rm Im}\nolimits\mathbb{O}.$ We thus have the following
identification of objects.

\begin{equation*}
\begin{tabular}{lll}
\textbf{Object} & \textbf{Loops} & \textbf{Octonions} \\ \hline
Pseudoautomorphism group & \multicolumn{1}{|l}{$\Psi $} & 
\multicolumn{1}{|l}{$\mathop{\rm Spin}\nolimits\left( 7\right) $} \\ 
Partial pseudoautomorphism group & \multicolumn{1}{|l}{$\Psi ^{\prime }$} & 
\multicolumn{1}{|l}{$SO\left( 7\right) $} \\ 
Automorphism group & \multicolumn{1}{|l}{$H$} & \multicolumn{1}{|l}{$G_{2}$}
\\ 
Lie algebra of $\Psi $ & \multicolumn{1}{|l}{$\mathfrak{p}$} & 
\multicolumn{1}{|l}{$\mathfrak{so}\left( 7\right) $} \\ 
Loop with full action of $\Psi $ & \multicolumn{1}{|l}{$\mathbb{L}$} & 
\multicolumn{1}{|l}{$U\mathbb{O\subset }S_{7}$} \\ 
Loop with partial action of $\Psi $ & \multicolumn{1}{|l}{$\mathbb{L}%
^{\prime }$} & \multicolumn{1}{|l}{$U\mathbb{O}^{\prime }\subset V_{8}$} \\ 
Tangent algebra & \multicolumn{1}{|l}{$\mathfrak{l}$} & \multicolumn{1}{|l}{$%
\mathop{\rm Im}\nolimits\mathbb{O\cong }V_{7}\cong \mathbb{R}^{7}$}%
\end{tabular}%
\end{equation*}

Therefore, on the manifold $M$ as above, the spin structure corresponds to a
principal $\Psi $-bundle in the general theory, the unit spinor bundle $U%
\mathcal{S}$ corresponds to the bundle $\mathcal{Q}$ and the unit subbundle $%
U\mathbb{O}M\ $of $\mathbb{O}M\cong \Lambda ^{0}\oplus TM$ corresponds to $%
\mathcal{Q}^{\prime }.$ This is precisely the octonion bundle introduced in 
\cite{GrigorianOctobundle}. Hence, we have the following dictionary relating
objects in the general loop bundle theory and $G_{2}$-geometry. 
\begin{equation*}
\begin{tabular}{l|l}
\textbf{Loop} \textbf{bundles} & $G_{2}$\textbf{-geometry} \\ \hline
$\mathcal{P}$ & Spin structure: principal $\mathop{\rm Spin}\nolimits\left(
7\right) $-bundle over $M$ \\ 
$\mathcal{Q}^{\prime }=\mathcal{P}\times _{\Psi ^{\prime }}\mathbb{L}%
^{\prime }$ & Unit octonion bundle $U\mathbb{O}M$ \\ 
$\mathcal{Q}=\mathcal{P\times }_{\Psi }\mathbb{L}$ & Unit spinor bundle $U%
\mathcal{S}$ \\ 
$\mathcal{A}=\mathcal{P\times }_{\Psi _{\ast }^{\prime }}\mathfrak{l}$ & 
Bundle of imaginary octonions: $TM$ \\ 
$\mathfrak{p}_{\mathcal{P}}=\mathcal{P\times }_{\left( \mathop{\rm Ad}%
\nolimits_{\xi }\right) _{\ast }}\mathfrak{p}$ & $\mathfrak{so}\left(
7\right) $-bundle over $M$ $\cong \Lambda ^{2}T^{\ast }M$ \\ 
$\mathop{\rm Ad}\nolimits\left( \mathcal{P}\right) =\mathcal{P}\times _{%
\mathop{\rm Ad}\nolimits_{\Psi }}\Psi $ & $SO\left( 7\right) $ gauge
transformations%
\end{tabular}%
\end{equation*}

$G_{2}$-structures can also be described using differential forms since $%
G_{2}$ is alternatively defined as the subgroup of $GL\left( 7,\mathbb{R}%
\right) $ that preserves a particular $3$-form $\varphi _{0}$ \cite%
{Joycebook}.

\begin{definition}
Let $\left( e^{1},e^{2},...,e^{7}\right) $ be the standard basis for $\left( 
\mathbb{R}^{7}\right) ^{\ast }$, and denote $e^{i}\wedge e^{j}\wedge e^{k}$
by $e^{ijk} $. Then define $\varphi _{0}$ to be the $3$-form on $\mathbb{R}%
^{7}$ given by 
\begin{equation}
\varphi _{0}=e^{123}+e^{145}+e^{167}+e^{246}-e^{257}-e^{347}-e^{356}.
\label{phi0def}
\end{equation}%
Then $G_{2}$ is defined as the subgroup of $GL\left( 7,\mathbb{R}\right) $
that preserves $\varphi _{0}$.
\end{definition}

It turns out that there is a $1$-$1$ correspondence between $G_{2}$%
-structures on a $7$-manifold and smooth $3$-forms $\varphi $ for which the $%
7$-form-valued bilinear form $B_{\varphi }$ as defined by (\ref{Bphi}) is
positive definite (for more details, see \cite{Bryant-1987} and the arXiv
version of \cite{Hitchin:2000jd}). 
\begin{equation}
B_{\varphi }\left( u,v\right) =\frac{1}{6}\left( u\lrcorner \varphi \right)
\wedge \left( v\lrcorner \varphi \right) \wedge \varphi .  \label{Bphi}
\end{equation}%
Here the symbol $\lrcorner $ denotes contraction of a vector with the
differential form: $\left( u\lrcorner \varphi \right) _{mn}=u^{a}\varphi
_{amn}.$

A smooth $3$-form $\varphi $ is said to be \emph{positive }if $B_{\varphi }$
is the tensor product of a positive-definite bilinear form and a
nowhere-vanishing $7$-form. In this case, it defines a unique Riemannian
metric $g_{\varphi }$ and volume form $\mathrm{\mathop{\rm vol}\nolimits}%
_{\varphi }$ such that for vectors $u$ and $v$, the following holds 
\begin{equation}
g_{\varphi }\left( u,v\right) \mathrm{\mathop{\rm vol}\nolimits}_{\varphi }=%
\frac{1}{6}\left( u\lrcorner \varphi \right) \wedge \left( v\lrcorner
\varphi \right) \wedge \varphi .  \label{gphi}
\end{equation}

An equivalent way of defining a positive $3$-form $\varphi $, is to say that
at every point, $\varphi $ is in the $GL\left( 7,\mathbb{R}\right) $-orbit
of $\varphi _{0}$. It can be easily checked that the metric (\ref{gphi}) for 
$\varphi =\varphi _{0}$ is in fact precisely the standard Euclidean metric $%
g_{0}$ on $\mathbb{R}^{7}$. Therefore, every $\varphi $ that is in the $%
GL\left( 7,\mathbb{R}\right) $-orbit of $\varphi _{0}$ has an \emph{%
associated} Riemannian metric $g$ that is in the $GL\left( 7,\mathbb{R}%
\right) $-orbit of $g_{0}.$ The only difference is that the stabilizer of $%
g_{0}$ (along with orientation) in this orbit is the group $SO\left(
7\right) $, whereas the stabilizer of $\varphi _{0}$ is $G_{2}\subset
SO\left( 7\right) $. This shows that positive $3$-forms forms that
correspond to the same metric, i.e., are \emph{isometric}, are parametrized
by $SO\left( 7\right) /G_{2}\cong \mathbb{RP}^{7}\cong S^{7}/\mathbb{Z}_{2}$%
. Therefore, on a Riemannian manifold, metric-compatible $G_{2}$-structures
are parametrized by sections of an $\mathbb{RP}^{7}$-bundle, or
alternatively, by sections of an $S^{7}$-bundle, with antipodal points
identified. The precise parametrization of isometric $G_{2}$-structures is
given in Theorem \ref{ThmSamegfam}.

\begin{theorem}[\protect\cite{bryant-2003}]
\label{ThmSamegfam}Let $M$ be a $7$-dimensional smooth manifold. Suppose $%
\varphi $ is a positive $3$-form on $M$ with associated Riemannian metric $g$%
. Then, any positive $3$-form $\tilde{\varphi}$ for which $g$ is also the
associated metric, is given by the following expression:%
\begin{equation}
\tilde{\varphi}=\sigma _{A}\left( \varphi \right) =\left( a^{2}-\left\vert
\alpha \right\vert ^{2}\right) \varphi -2a\alpha \lrcorner \left( \ast
\varphi \right) +2\alpha \wedge \left( \alpha \lrcorner \varphi \right)
\label{phisameg1}
\end{equation}%
where $A=\left( a,\alpha \right) $ is a pair with $a$ a scalar function on $%
M $ and $\alpha $ a vector field such that 
\begin{equation}
a^{2}+\left\vert \alpha \right\vert ^{2}=1  \label{asqvsq}
\end{equation}
\end{theorem}

The pair $A=\left( a,\alpha \right) $ can in fact be also interpreted as a 
\emph{unit octonion} section, where $a$ is the real part, and $\alpha $ is
the imaginary part. The relationship between octonion bundles and $G_{2}$%
-structures was developed in detail in \cite{GrigorianOctobundle}. In
particular, sections of a unit octonion bundle over $M$ parametrize $G_{2}$%
-structures that are associated to the same metric.

\begin{definition}
The \emph{octonion bundle }$\mathbb{O}M$ on $M$ is the rank $8$ real vector
bundle given by 
\begin{equation}
\mathbb{O}M\cong \Lambda ^{0}\oplus TM  \label{OMdef}
\end{equation}%
where $\Lambda ^{0}\cong M\times \mathbb{R}\ $is a trivial line bundle. At
each point $p\in M$, $\mathbb{O}_{p}M\cong \mathbb{R}\oplus T_{p}M.$
\end{definition}

The definition (\ref{OMdef}) gives a natural decomposition of octonions on $M
$ into real and imaginary parts. We may write $A=\left( \mathop{\rm Re}A,%
\mathop{\rm Im}A\right) $ or $A=\left( 
\begin{array}{c}
\mathop{\rm Re}A \\ 
\mathop{\rm Im}A%
\end{array}%
\right) $. Since $\mathbb{O}M$ is defined as a tensor bundle, the Riemannian
metric $g$ on $M$ induces a metric on $\mathbb{O}M.$ Let $A=\left( a,\alpha
\right) \in \Gamma \left( \mathbb{O}M\right) .$ Then, 
\begin{equation}
\left\vert A\right\vert ^{2}=a^{2}+\left\vert \alpha \right\vert _{g}^{2}
\label{Omet}
\end{equation}%
The metric allows to define the subbundle $U\mathbb{O}M$ of octonions of
unit norm and allows allows to define a vector cross product on $TM.$

\begin{definition}
Given the $G_{2}$-structure $\varphi $ on $M,$ we define a \emph{vector
cross product with respect to }$\varphi $ on $M.$ Let $\alpha $ and $\beta $
be two vector fields, then define%
\begin{equation}
\left\langle \alpha \times _{\varphi }\beta ,\gamma \right\rangle =\varphi
\left( \alpha ,\beta ,\gamma \right)  \label{vcrossdef}
\end{equation}%
for any vector field $\gamma $ \cite{Gray-VCP,karigiannis-2005-57}.
\end{definition}

Using the inner product and the cross product, we can now define the \emph{%
octonion product }on $\mathbb{O}M$.

\begin{definition}
Let $A,B\in \Gamma \left( \mathbb{O}M\right) .$ Suppose $A=\left( a,\alpha
\right) $ and $B=\left( b,\beta \right) $. Given the vector cross product (%
\ref{vcrossdef}) on $M,$ we define the \emph{octonion product }$A\circ
_{\varphi }B$\emph{\ with respect to }$\varphi $\emph{\ }as follows:%
\begin{equation}
A\circ _{\varphi }B=\left( 
\begin{array}{c}
ab-\left\langle \alpha ,\beta \right\rangle \\ 
a\beta +b\alpha +\alpha \times _{\varphi }\beta%
\end{array}%
\right)  \label{octoproddef}
\end{equation}
\end{definition}

If there is no ambiguity as to which $G_{2}$-structure is being used to
define the octonion product, we will simply write$\ AB$ to denote it. In
particular, $\left\vert AB\right\vert =\left\vert A\right\vert \left\vert
B\right\vert .$

Given a $G_{2}$-structure $\varphi $ with an associated metric $g$, we may
use the metric to define the Levi-Civita connection $\nabla $. The \emph{%
intrinsic torsion }of a $G_{2}$-structure is then defined by $\nabla \varphi 
$. Following \cite{GrigorianG2Torsion1,karigiannis-2007}, we can write 
\begin{equation}
\nabla _{a}\varphi _{bcd}=-2T_{a}^{\ e}\psi _{ebcd}^{{}}  \label{codiffphi}
\end{equation}%
where $T_{ab}$ is the \emph{full torsion tensor}. Similarly, we can also
write 
\begin{equation}
\nabla _{a}\psi _{bcde}=8T_{a[b}\varphi _{cde]}  \label{psitorsion}
\end{equation}%
We can also invert (\ref{codiffphi}) to get an explicit expression for $T$ 
\begin{equation}
T_{a}^{\ m}=-\frac{1}{48}\left( \nabla _{a}\varphi _{bcd}\right) \psi
^{mbcd}.
\end{equation}%
This $2$-tensor fully defines $\nabla \varphi $ \cite{GrigorianG2Torsion1}.

\begin{remark}
The torsion tensor $T$ as defined here is actually corresponds to $-T$ in 
\cite{GrigorianOctobundle}, $-\frac{1}{2}T$ in \cite{GrigorianG2Torsion1}
and $\frac{1}{2}T$ in \cite{karigiannis-2007}. Even though this requires
extra care when translating various results, it will turn out to be more
convenient.
\end{remark}

Given a unit norm spinor section $\xi \in \Gamma \left( \mathcal{S}\right) ,$
a $G_{2}$-structure $3$-form $\varphi _{\xi }$ is defined in the following
way: 
\begin{equation}
\varphi _{\xi }\left( \alpha ,\beta ,\gamma \right) =-\left\langle \xi
,\alpha \cdot \left( \beta \cdot \left( \gamma \cdot \xi \right) \right)
\right\rangle _{S},  \label{phixispin}
\end{equation}%
where $\cdot $ denotes Clifford multiplication, $\alpha ,\beta ,\gamma $ are
arbitrary vector fields and $\left\langle \cdot ,\cdot \right\rangle _{S}$
is the inner product on the spinor bundle. The Levi-Civita connection lifts
to the spinor bundle $\mathcal{S},$ giving the spinorial covariant
derivative $\nabla ^{S}.$ Then, the torsion $T^{\left( \xi \right) }$ of $%
\varphi _{\xi }$ is given by \cite[Definition 4.2 and Lemma 4.3]%
{AgricolaSpinors} 
\begin{equation}
\nabla _{X}^{S}\xi =T_{X}^{\left( \xi \right) }\cdot \xi ,  \label{LCspinxi}
\end{equation}%
Note that in \cite{AgricolaSpinors}, the torsion endomorphism is denoted by $%
S.$

Comparing with Definition \ref{defTors} and noting that the unit spinor
bundle $U\mathcal{S}$ corresponds to the loop bundle $\mathcal{Q},$ we see
that the torsion $T^{\left( \xi \right) }$ of the $G_{2}$-structure $\varphi
_{\xi }$ precisely corresponds to the torsion $T^{\left( \xi ,\nabla \right)
}$ of the section $\xi $ with respect to the Levi-Civita connection $\nabla
. $ Similarly, given a unit octonion section $A\in \Gamma \left( U\mathbb{O}%
M\right) ,$ $A\cdot \xi $ is again a unit spinor which defines a $G_{2}$%
-structure $\varphi _{A\cdot \xi }.$ Considering both $A$ and $\xi $ as
octonions in $U\mathbb{O}^{\prime }$ and $U\mathbb{O},$ respectively, this
is just octonion multiplication $A\xi ,$ and $\varphi _{A\cdot \xi }=\varphi
_{A\xi }=\sigma _{A}\left( \varphi _{\xi }\right) .$ Therefore, all
isometric $G_{2}$-structures are given by $\varphi _{A\xi }$ for some unit
octonion section $A.$ The curvature component $\hat{F}$ corresponds to the a
particular component of the Riemann curvature tensor. These relationships
are explored in detail in \cite{GrigorianOctobundle}. Thus we can
reformulate Theorem \ref{thmMain} for $G_{2}$-structures.

\begin{theorem}
\label{thmMainG2}Suppose $M$ is a closed $7$-dimensional manifold with a
smooth $G_{2}$-structure $\varphi $ with torsion $T$ with respect to the
Levi-Civita connection $\nabla .$ Also, suppose $k$ is a positive integer
and $p$ is a positive real number such that $kp>7.$ Then, there exist
constants $\delta \in (0,1]$ and $K\in \left( 0,\infty \right) ,$ such that
if $T$ satisfies 
\begin{equation*}
\left\Vert T\right\Vert _{W^{k,p}}<\delta ,
\end{equation*}%
then there exists a smooth section $V\in \Gamma \left( U\mathbb{O}M\right) ,$
such that 
\begin{equation*}
\mathop{\rm div}T^{\left( V\right) }=0
\end{equation*}%
and 
\begin{equation}
\left\Vert T^{\left( V\right) }\right\Vert _{W^{k,p}}<K\left\Vert
T\right\Vert _{W^{k,p}}\left( 1+\left\Vert T\right\Vert
_{W^{k,p}}^{k}\right) .
\end{equation}
\end{theorem}

\begin{remark}
If we choose $p=2$ to work with Hilbert spaces, then for a smooth section $V,
$ we need $k\geq 4,$ so the condition on $T$ is to be sufficiently small in
the $W^{4,2}$-norm.
\end{remark}

\appendix

\section{Appendix}

\setcounter{equation}{0}

\begin{lemma}
\label{lemSobProd}Let $k,$ $k^{\prime },n$ be positive integers and $kp>n$,
for a positive real number $p,$ and let $A_{1},...,A_{k^{\prime }}$ be
real-valued functions on a compact $n$-dimensional Riemannian manifold $M$.
Also, suppose $m_{1},...,m_{k^{\prime }}$ are non-negative integers and $%
q_{1},...,q_{k^{\prime }}$ are positive integers such that $%
\sum_{j=1}^{k^{\prime \prime }}q_{j}m_{j}\leq k,$ then 
\begin{equation}
\left\Vert \prod_{j=1}^{k^{\prime }}A_{j}^{m_{j}}\right\Vert
_{L^{p}}\lesssim \prod_{j=1}^{k^{\prime }}\left\Vert A_{j}\right\Vert
_{W^{k-q_{j},p}}^{m_{j}}.  \label{Sobprodformula}
\end{equation}
\end{lemma}

\begin{proof}
Let $k^{\prime \prime }=\sum_{j=1}^{k^{\prime \prime }}q_{j}m_{j}\leq k.$
Then suppose $p_{j}=\frac{pk^{\prime \prime }}{q_{j}m_{j}}$ for all $j$ for
which $m_{j}>0,$ so that $\frac{1}{p_{j}}=\frac{q_{j}m_{j}}{pk^{\prime
\prime }},$ and hence $\sum_{j=1}^{k^{\prime }}\frac{1}{p_{j}}=\frac{1}{p}.$
Thus, from H\"{o}lder's inequality, we have 
\begin{equation*}
\left\Vert \prod_{j=1}^{k^{\prime }}A_{j}^{m_{j}}\right\Vert
_{L^{p}}\lesssim \prod_{j=1}^{k^{\prime }}\left\Vert A_{j}\right\Vert
_{L^{m_{j}p_{j}}}^{m_{j}}.
\end{equation*}%
Now note that using the definition of $p_{j}$, $\frac{q_{j}}{k^{\prime
\prime }}=\frac{p}{p_{j}m_{j}}\leq 1,$ and hence 
\begin{eqnarray*}
\frac{k-q_{j}}{n} &=&\frac{k}{n}\left( 1-\frac{q_{j}}{k}\right) \\
&\geq &\frac{k}{n}\left( 1-\frac{q_{j}}{k^{\prime \prime }}\right) \\
&=&\frac{k}{n}\left( 1-\frac{p}{p_{j}m_{j}}\right) .
\end{eqnarray*}%
Since by assumption, $\frac{k}{n}>\frac{1}{p}$, we obtain%
\begin{equation*}
\frac{k-q_{j}}{n}>\frac{1}{p}-\frac{1}{p_{j}m_{j}}.
\end{equation*}%
Using a version of the Sobolev Embedding Theorem, this shows that indeed, 
\begin{equation*}
\left\Vert A_{j}\right\Vert _{L^{m_{j}p_{j}}}\lesssim \left\Vert
A_{j}\right\Vert _{W^{k-q_{j},p}},
\end{equation*}%
and (\ref{Sobprodformula}) follows.

\end{proof}

\begin{theorem}[{Banach space uantitative implicit function theorem%
\protect\cite[Theorem F.1]{Feehan2020}}]
\label{thmIFT}Let $k\geq 1$ be an integer or $\infty ,$ and let $X,Y,Z$ be
real Banach spaces. Suppose $U\subset X$ and $V\subset Y$ are open
neighborhoods of points $x_{0}\in X$ and $y_{0}\in Y$ and $f:U\times
V\longrightarrow Z$ is a $C^{k}$ map such that $f\left( x_{0},y_{0}\right)
=0 $ and the partial derivative of $f$ at $\left( x_{0},y_{0}\right) $ with
respect to the second variable, $\left. \partial _{2}f\right\vert _{\left(
x_{0},y_{0}\right) }\in \newline
func{Hom}\left( Y,Z\right) $ is an isomorphism of Banach spaces. Define 
\begin{equation*}
N=\left\Vert \left( \left. \partial _{2}f\right\vert _{\left(
x_{0},y_{0}\right) }\right) ^{-1}\right\Vert _{\mathop{\rm Hom}%
\nolimits\left( Z,Y\right) }.
\end{equation*}%
Let $\zeta \in (0,1]$ be small enough such that the open ball $B_{\zeta
}\left( x_{0}\right) \subset U$ and $B_{\zeta }\left( y_{0}\right) \subset
V, $ and assume 
\begin{eqnarray*}
\sup_{\left( x,y\right) \in B_{\zeta }\left( x_{0}\right) \times B_{\zeta
}\left( y_{0}\right) }\left\Vert \left. \partial _{2}f\right\vert _{\left(
x,y\right) }-\left. \partial _{2}f\right\vert _{\left( x_{0},y_{0}\right)
}\right\Vert _{_{\mathop{\rm Hom}\nolimits\left( Y,Z\right) }} &\leq &\frac{1%
}{2N} \\
\beta &=&\sup_{\left( x,y\right) \in B_{\zeta }\left( x_{0}\right) \times
B_{\zeta }\left( y_{0}\right) }\left\Vert \left. \partial _{1}f\right\vert
_{\left( x,y\right) }\right\Vert _{\mathop{\rm Hom}\nolimits\left(
X,Z\right) }<\infty .
\end{eqnarray*}%
Then there exist a constant $\delta \in \left( 0,\min \left\{ \zeta ,\frac{%
\zeta }{2\beta N}\right\} \right] $ and unique $C^{k}$ map $g:B_{\delta
}\left( x_{0}\right) \longrightarrow B_{\zeta }\left( y_{0}\right) $ such
that $y_{0}=g\left( x_{0}\right) $ and 
\begin{eqnarray*}
f\left( g\left( x\right) ,x\right) &=&0,\ \forall x\in B_{\delta }\left(
x_{0}\right) \\
\left. Dg\right\vert _{x} &=&-\left( \left. \partial _{2}f\right\vert
_{\left( x,g\left( x\right) \right) }\right) ^{-1}\left. \partial
_{1}f\right\vert _{\left( x,g\left( x\right) \right) }\in \mathop{\rm Hom}%
\nolimits\left( X,Y\right) ,\ \forall x\in B_{\delta }\left( x_{0}\right) \\
\left\Vert g\left( x_{1}\right) -g\left( x_{2}\right) \right\Vert _{Y} &\leq
&2\beta N\left\Vert x_{1}-x_{2}\right\Vert _{X,}\ \forall x_{1},x_{2}\in
B_{\delta }\left( x_{0}\right) .
\end{eqnarray*}
\end{theorem}

\begin{acknowledgement}
This work was supported by the National Science Foundation grant DMS-1811754.
\end{acknowledgement}

\bibliographystyle{habbrv}
\bibliography{refs2}

\end{document}